\documentclass[a4paper]{article}

\usepackage{amsmath}
\usepackage{amssymb}
\usepackage{amsthm}
\usepackage{mathrsfs}
\usepackage{wasysym}
\usepackage{bbm}		
\usepackage{mathtools}
\usepackage[shortlabels]{enumitem}
\usepackage{braket}		
\usepackage{mdwlist}		
\usepackage{xcolor}
\usepackage{tikz}
\usepackage{graphicx}
\usepackage[a4paper, left=2.7cm, right=2.7cm, top=3.5cm, bottom=2cm]{geometry}
\usetikzlibrary{patterns,positioning,arrows,decorations.markings,calc,decorations.pathmorphing,decorations.pathreplacing}
\usepackage{cleveref}
\theoremstyle{plain}
\newtheorem{theorem}{Theorem}[section]
\newtheorem{lemma}[theorem]{Lemma}
\newtheorem{corollary}[theorem]{Corollary}
\newtheorem{proposition}[theorem]{Proposition}
\newtheorem{definition}[theorem]{Definition}
\newtheorem{notation}[theorem]{Notation}

\newtheorem{question}[theorem]{Question}

\theoremstyle{definition}
\newtheorem{remark}[theorem]{Remark}
\newtheorem{example}[theorem]{Example}

\newcommand{\ZZ}{\mathbb{Z}}			
\newcommand{\NN}{\mathbb{N}}			
\newcommand{\RR}{\mathbb{R}}			
\newcommand{\QQ}{\mathbb{Q}}			
\newcommand{\CC}{\mathbb{C}}			

\newcommand{\Map}{{\rm Map}}
\newcommand{\IE}{\mathtt{IE}}
\newcommand{\R}{\mathtt{A}}

\newcommand{\FF}{\mathcal{F}}
\newcommand{\diam}{{\rm diam}}
\newcommand{\Sym}{{\rm Sym}}

\newcommand{\supp}{
	\operatorname{\mathrm{supp}}%
}

\newcommand{\htop}{h_{\rm{top}}}
\newcommand{\hnaive}{h^{\rm{nv}}_{\rm{top}}}
\newcommand{\hsof}{h^{\Sigma}_{\rm{top}}}
\newcommand{\define}[1]{\textbf{#1}}
\newcommand{\cross}{
	\begin{tikzpicture}[scale = 0.06]
	\draw [black!30, fill = black] (0,0) rectangle (1,1);
	\draw [black!30, fill = black] (1,0) rectangle (2,1);
	\draw [black!30, fill = black] (-1,0) rectangle (0,1);
	\draw [black!30, fill = black] (0,1) rectangle (1,2);
	\draw [black!30, fill = black] (0,-1) rectangle (1,0);
	\end{tikzpicture}
}

\title{
	Markovian properties of continuous group actions: algebraic actions, entropy and the homoclinic group
}

\author{Sebasti\'an Barbieri, Felipe Garc\'ia-Ramos and Hanfeng Li}

\newcommand{\Addresses}{{
  \bigskip

\hskip-\parindent   S.~Barbieri, \textsc{DMCC, Universidad de Santiago de Chile,
	Las Sophoras 173. Estaci\'on Central. Santiago. Chile.}\par\nopagebreak
\textit{E-mail address}: \texttt{sebastian.barbieri@usach.cl}

  \medskip

\hskip-\parindent   F.~Garc\'ia-Ramos, \textsc{CONACyT, M\'exico. \\ Instituto de F\'isica, Universidad Aut\'onoma de San Luis Potos\'i, M\'exico.}\par\nopagebreak
  \textit{E-mail address}: \texttt{fgramos@conacyt.mx}

 \medskip

\hskip-\parindent  H.~Li, \textsc{Center of Mathematics, Chongqing University,
Chongqing 401331, China.\\
Department of Mathematics, SUNY at Buffalo,
Buffalo, NY 14260-2900, USA.}\par\nopagebreak
  \textit{E-mail address}: \texttt{hfli@math.buffalo.edu}

}}

\date{}

\begin{document}

	\maketitle
	
		\begin{abstract}
			We provide a unifying approach which links results on algebraic actions by Lind and Schmidt, Chung and Li, and a topological result by Meyerovitch that relates entropy to the set of asymptotic pairs. In order to do this we introduce a series of Markovian properties and, under the assumption that they are satisfied, we prove several results that relate topological entropy and asymptotic pairs (the homoclinic group in the algebraic case). As new applications of our method, we give a characterization of the homoclinic group of any finitely presented expansive algebraic action of (1) any elementary amenable group with an upper bound on the orders of finite subgroups or (2) any left orderable amenable group, using the language of independence entropy pairs.
		\end{abstract}
	
	\medskip	
		
	\noindent
	\textbf{Key words and phrases:} topological entropy, sofic entropy, naive entropy, algebraic actions, expansive actions, asymptotic pairs, homoclinic points, local entropy theory, topological Markov properties, strong Atiyah conjecture.
	\smallskip
	
	\noindent
	\textbf{MSC2010:} \textit{Primary:}
	37B40, 
	22D40, 
	20C07. 
	\textit{Secondary:}
	22F05, 
	37C85, 
	37C29, 
	37B05.  
	
	\section{Introduction}\label{section:intro}
	
	Inspired by properties introduced in the context of symbolic dynamical systems in~\cite{ChandgotiaHanMarcusMeyerovitchPavlov2014, Richifest} we define a series of Markovian properties for actions of countable groups $G$ on compact metrizable spaces $X$ by homeomorphisms. Among them, the topological Markov property (TMP) and the strong topological Markov property (strong TMP). Every group action which satisfies the pseudo-orbit tracing property (or shadowing) has the strong TMP, and every action with the strong TMP has the TMP. Although these properties were introduced to study supports of measures that arise in the theory of thermodynamic formalism, it turns out that they are specially relevant in the context of algebraic actions on compact metrizable abelian groups.

	A group action $G \curvearrowright X$ satisfies the TMP, if for every $\varepsilon >0$, there is $\delta>0$ such that for every finite set $A\subset G$ there exists a finite subset $B \supset A$ of $G$ such that for every pair $x,y \in X$ whose $G$-orbits are at distance at most $\delta$ in $B\setminus A$, there is $z\in X$ whose $G$-orbit is at distance at most $\varepsilon$ from the $G$-orbit of $x$ in $B$ and at distance at most $\varepsilon$ from the $G$-orbit of $y$ in $G\setminus A$. Intuitively, satisfying the TMP means that for every finite set $A \subset G$, there is a finite subset $B \supset A$ of $G$ such that knowledge of the values (up to $\delta$) of any $G$-orbit on $C\setminus A$ for some $C\supset B$ does not provide further information about the values (up to $\varepsilon$) of the $G$-orbit in $A$ than the mere knowledge of the values in $B\setminus A$, hence the name Markovian. The strong TMP imposes that the set $B$ should take the form $FA$ for some fixed finite set $F \subset G$ which does not depend on $A$, and hence gives a bounded variant of the TMP.
	
An action $G\curvearrowright X$ of a countable group $G$ on a compact metrizable space $X$ by homeomorphisms is called expansive if there is some $c>0$ such that $\sup_{s\in G}d(sx, sy)>c$ for all distinct $x, y\in X$, where $d$ is a given compatible metric on $X$. A pair $(x, y)$ in $X^2$ is called asymptotic if $d(sx, sy)\to 0$ as $G \ni s\to \infty$. The TMP and the strong TMP are especially useful for establishing relations between independence entropy pairs and asymptotic pairs for expansive actions (Theorems~\ref{T-A to IE}, \ref{Theorem_asymptotic_pairs_give_SOFIC_entropy_pairs} and \ref{T-exp sTMP IE to asym}).

By an \define{algebraic action} we mean an action of a countable group $G$ on a compact metrizable abelian group $X$ by continuous automorphisms.
The study of algebraic actions has been an active field, because of rich connections with commutative algebra, operator algebras, and $L^2$-invariants. See for example \cite{Lind1977, MilesThomas1978, Yuzvinskii1968} for algebraic actions of $\ZZ$, \cite{EinsiedlerWard2005, KitchensSchmidt2000, LindSchmidt1999, LindSchmidtVerbitskiy2013, LindSchmidtWard1990, RudolphSchmidt1995, Schmidt1995, SchmidtWard1993} for algebraic actions of $\ZZ^d$, and \cite{Bhattacharya2019, BhattacharyaCeccherini-SilbersteinCoornaert2019, Bowen2011, ChungLi2015, Deninger2006, DeningerSchmidt2007, GaboriauSeward2019, GollSchmidtVerbitskiy2014, Hayes2016, Hayes2019, Hayes2019max, Li2012, LiLiang2018, LiLiang2019, LiThom2014, LindSchmidt2015} for algebraic actions of countable groups. Denote by $\ZZ G$ and $\CC G$ the group rings of $G$ with coefficients in $\ZZ$ and $\CC$ respectively (see Section~\ref{section:algebraico}). For any algebraic action $G \curvearrowright X$, the Pontryagin dual $\widehat{X}$ of $X$ as a compact abelian group is naturally a countable left $\ZZ G$-module.
 In fact, up to isomorphism, there is a natural one-to-one correspondence between algebraic actions of $G$ and countable left $\ZZ G$-modules.
Recall that a unital ring is \define{(left/right) Noetherian} if every (left/right) ideal is finitely generated.
For any positive integers $m,n$ and any $a\in M_{m\times n}(\CC G)$, the von Neumann dimension $\dim_{\rm vN} \ker a$ of the kernel of the bounded operator $M_{n\times 1}(\ell^2(G))\rightarrow M_{m\times 1}(\ell^2(G))$ sending $z$ to $az$ is a real number in $[0, n]$ (see Section~\ref{section:algebraico}). The group $G$ is said to satisfy the \define{strong Atiyah conjecture} if $\dim_{\rm vN} \ker a$ lies in the subgroup of $\QQ$ generated by $1/|H|$ for $H$ ranging over finite subgroups of $G$.
For algebraic actions $G\curvearrowright X$, the asymptotic pairs are determined by the homoclinic group $\Delta(X, G)$ consisting of $x\in X$ such that $sx\to e_X$ as $G \ni s\to \infty$, where $e_X$ denotes the identity element of $X$. For any algebraic action $G\curvearrowright X$ of an amenable group $G$, we may either view the action as a topological dynamical system and speak about the topological entropy, or view it as an action preserving the normalized Haar measure of $X$ and speak about the measure-theoretical entropy (see~\Cref{section:preliminaries}). It turns out that these two entropies always coincide~\cite{Deninger2006}. Furthermore, the action has complete positive entropy in the topological sense (i.e. every nontrivial topological factor has positive topological entropy) exactly when it has completely positive entropy in the measure-preserving sense (i.e. every nontrivial measurable factor has positive measure-theoretical entropy)~\cite{ChungLi2015}.
One of our main results is the following:

\begin{theorem} \label{T-main}
Let $G\curvearrowright X$ be an expansive algebraic action of a countably infinite amenable group.
Assume that at least one of the following conditions holds:
\begin{enumerate}
\item the integral group ring $\ZZ G$  is left Noetherian;
\item $G$ satisfies the strong Atiyah conjecture, there is an upper bound on the orders of finite subgroups of $G$, and  the Pontryagin dual $\widehat{X}$ of $X$ is a finitely presented left $\ZZ G$-module.
\end{enumerate}
Then $G \curvearrowright X$ has the strong TMP. As a consequence, the following hold:
\begin{enumerate}
\item[i.] $G\curvearrowright X$ has positive entropy if and only if $\Delta(X, G)$ is nontrivial;
\item[ii.] $G\curvearrowright X$ has completely positive entropy if and only if $\Delta(X, G)$ is dense in $X$.
\end{enumerate}
\end{theorem}

The ``if'' direction of consequences (i) and (ii) in the above statement, i.e.  non-triviality and denseness of the homoclinic group $\Delta(X, G)$, imply that $G\curvearrowright X$ has positive entropy and completely positive entropy respectively, is in fact valid in much wider generality. It is known for all amenable groups \cite{ChungLi2015}, and for all groups in the sense of naive topological entropy \cite{KerrLi2013}. It is even true without assuming that $X$ is abelian, for all groups in the sense of naive topological entropy and  for all sofic groups in the sense of sofic topological entropy (see \Cref{coralg}). A recent result of Hayes shows that when $\Delta(X, G)$ is replaced by the subgroup of square summable homoclinic points, it is also true for sofic groups without assuming expansivity~\cite[Theorem 1.3]{Hayes2019max}. The ``only if'' direction (the existence of non-trivial homoclinic points from positive entropy and denseness of the homoclinic group from completely positive entropy) is much more difficult to establish. It is still open whether the ``only if'' direction holds for all finitely presented algebraic actions of amenable groups. The crucial statement of the second part of the theorem is
that the ``only if'' direction holds in the aforementioned cases.
Both directions were first established by Lind and Schmidt~\cite{LindSchmidt1999} in the case $G=\ZZ^d$, using commutative algebra tools. The case where $\ZZ G$ is left Noetherian was proven by Chung and Li~\cite{ChungLi2015} using local entropy theory, Peters' entropy formula and Yuzvinskii's addition formula for entropy. Besides local entropy theory, our proof uses techniques of von Neumann algebras to prove the strong TMP, and then straightforward combinatorial arguments to obtain the conclusion. We also present examples where this conclusion does not hold when assuming weaker hypotheses.

For every polycyclic-by-finite group $G$, the group ring $\ZZ G$ is left Noetherian \cite{Hall1954} \cite[Theorem 1.5.12]{McConnellRobson2001}. It is a long standing conjecture that the converse holds.
Recall that the class of
\define{elementary amenable groups} is the smallest class of groups containing all finite groups and all abelian groups and is closed under taking subgroups, quotient groups, extensions, and inductive limits~\cite{Day1957}.
A result of Linnell says that if there is an upper bound on the orders of finite subgroups of an elementary amenable group $G$, then the strong Atiyah conjecture holds for $G$ \cite[Theorem 1.5]{Linnell1993} \cite[Theorem 10.19]{Luck2002}.
In particular, for any polycyclic-by-finite group $G$, there is an upper bound on the orders of  finite subgroups of $G$ and the strong Atiyah conjecture holds for $G$.
Recall that a unital ring $R$ is called a \define{domain} if for any $a, b\in R$ with $ab=0$ one must have either $a=0$ or $b=0$.
If $\CC G$ is a domain, then $G$ is torsion-free. Kaplansky's  zero-divisor conjecture asserts that the converse holds.
For torsion-free amenable groups, the strong Atiyah conjecture is equivalent to Kaplansky's  zero-divisor conjecture \cite[Lemma 10.16]{Luck2002}.
In particular, if $G$ is amenable and $\CC G$ is a domain, then $G$ satisfies the strong Atiyah conjecture and $1$ is an upper bound on the orders of finite subgroups.
Also recall that $G$ is said to be \define{left orderable} \cite{MuraRhemtulla1977} if there is a total order $\le$ on $G$ such that $s\le t$ implies $gs\le gt$ for all $g, s, t\in G$.
Based on his work on the strong Atiyah conjecture, Linnell showed that if $G$ has  a  torsion-free elementary amenable normal subgroup $H$ such that $G/H$ is left orderable, then $\CC G$ is a domain \cite[Theorem 8.9]{Linnell1998}. In particular, if $G$ is torsion-free elementary amenable or left orderable,
then $\CC G$ is a domain.
In \cite{Grigorchuk1986} Grigorchuk constructed a finitely generated group of intermediate growth (so amenable but not elementary amenable \cite{Chou1980}), which was shown to left orderable by Grigorchuk and Maki \cite{GrigorchukMaki1993}.

The strong Atiyah conjecture fails for the lamplighter group $(\ZZ/2\ZZ)\wr\ZZ$ \cite{GrigorchukZuk2001}. So far there is no counterexample for the strong Atiyah conjecture in the case there is an upper bound on the orders of finite subgroups.

The finitely presented condition in (2) of Theorem~\ref{T-main} is natural since for any expansive algebraic action $G\curvearrowright X$ of  a countable group $G$,
the  Pontryagin dual $\widehat{X}$ of $X$ is a finitely generated left $\ZZ G$-module \cite{Schmidt1995}. In particular, when $\ZZ G$ is left Noetherian, $\widehat{X}$ is finitely presented  for  every expansive algebraic action $G\curvearrowright X$ of $G$. In general, when $\ZZ G$ is not left Noetherian, we
expect much better dynamical properties for algebraic actions $G\curvearrowright X$ with $\widehat{X}$ finitely presented than just $\widehat{X}$ finitely generated. For instance, Meyerovitch \cite{Meyerovitch2017} has constructed an expansive algebraic action of an infinite locally finite abelian group with positive entropy and trivial homoclinic group, in which $\widehat{X}$ is finitely generated but not finitely presented.

Homoclinic points were first studied by Poincar\'e \cite{Andersson1994}, and are used in the study of smooth dynamical systems \cite{PujalsSambarino2000, Crovisier2010} (note that Anosov diffeomorphisms are expansive and have the pseudo-orbit tracing property \cite[Theorem 1.2.1]{AokiHiraide1994}).
Recently, Meyerovitch~\cite{Meyerovitch2017} showed that every expansive action of an amenable group on a compact metrizable space by homeomorphisms, which has positive topological entropy and which satisfies the pseudo-orbit tracing property, must have off-diagonal asymptotic pairs. Schmidt showed that expansive actions of polycyclic-by-finite groups on zero-dimensional compact metrizable groups by continuous automorphisms satisfy the pseudo-orbit tracing property \cite[Corollary 2.3, Theorems 3.8 and 4.2]{Schmidt1995}, thus Meyerovitch's result applies to such actions. However, Meyerovitch's result does not apply to all expansive algebraic actions of polycyclic-by-finite groups as Bhattacharya constructed an expansive algebraic action of a polycyclic group which does not have the pseudo-orbit tracing property~\cite{Bhattacharya2019}. We improve Meyerovitch's result by showing that the pseudo-orbit tracing condition in his result may be replaced by the condition of satisfying the strong TMP (Corollary~\ref{C-exp sTMP IE to asym}). This provides a broader dynamical setting than the pseudo-orbit tracing property, which also assures the existence of off-diagonal asymptotic pairs under the assumption of positive topological entropy. Our proof of~\Cref{T-main} rests on the fact that both of the conditions stated in the theorem imply that the action has the strong TMP (\Cref{T-expansive algebraic to STMP}).

	We also study analogous results in the context of non-amenable group actions. Entropy theory beyond the scope of actions of amenable groups has only been introduced recently. For sofic group actions the theory of measurable entropy was introduced by Bowen in~\cite{Bowen2010_2} and its topological counterpart along with a variational principle by Kerr and Li~\cite{KerrLi2011}. Another notion, naive entropy, which applies to actions of any countable group was formally introduced respectively in the measurable and topological settings by Bowen~\cite{bowen2017examples} and Burton~\cite{burton2017naive} although in the amenable case the notion is already present in the work of Downarowicz, Frej and Romagnoli~\cite{DownarowiczFrejRomagnoli2016}.
The naive entropy can only take the values $0$ or $+\infty$ for action of non-amenable groups \cite{bowen2017examples, burton2017naive}, thus is not suitable as an invariant for classifying actions. However, it is very suitable for discussing whether the entropy is positive or not \cite{LiRong2019}.
We provide an example of an expansive algebraic symbolic action of the free group with two generators which has uniformly positive naive entropy but whose homoclinic group is trivial~(\Cref{example1}), thus showing that Theorem~\ref{T-main} does not extend naively into the context of non-amenable group actions.

	The TMP provides a condition which ensures that under simple conditions a group action has positive topological entropy. Assuming the TMP, we show that expansivity and the existence of an off-diagonal asymptotic pair (and a necessary technical condition) is enough to ensure both positive topological naive and sofic entropy~(Theorems~\ref{T-A to IE} and~\ref{Theorem_asymptotic_pairs_give_SOFIC_entropy_pairs}). This gives easy criteria to show positive entropy for all sofic approximation sequences. For example, we use this result to show that expansive algebraic actions with nontrivial homoclinic group (see~\Cref{coralg}) and hard-square models (see~\Cref{hsmodel}) in sofic groups have positive sofic entropy for any sofic approximation sequence.
	
	The results explained in the previous paragraphs are given in the context of local entropy theory. Classical local entropy theory was initiated by Blanchard when he introduced the concept of entropy pairs~\cite{Blanchard1992, Blanchard1993} and developed quickly in \cite{BlanchardGlasnerHost1997, BlanchardGlasnerKolyadaMaass2002, BlanchardHostMaassMartinezRudolph1995, BlanchardHostRuette2002, GlasnerBook2003, HuangMaassRomagnoliYe2004, HuangYe2006, HuangYeZhang2011} (for a survey on local entropy theory see~\cite{GlasnerYe2009}). Later on a combinatorial approach was given by Kerr and Li in \cite{KerrLI2007},
	and further developed  in \cite{ChungLi2015, Hayes2017, HuangYe2009, KerrLi2009, KerrLi2013, KerrLiBook2016, LiRong2019}. One advantage of local entropy theory is that it provides necessary and sufficient conditions for uniform positive entropy (an action which has positive topological entropy with respect to any standard open cover).
	
	The paper is organized as follows. In~\Cref{section:preliminaries} we provide the definitions of several classical notions which will be used through the paper. Particularly, we provide definitions of topological classical, sofic and naive entropy and their local versions. We also provide a few notions on shift spaces and we define the pseudo-orbit tracing property.
	
	In~\Cref{section:resultados_basicos} we introduce our topological Markov properties and their uniform versions. We prove several structural results of the topological Markov properties, in particular, we describe the connection between uniformity and expansivity, which shall be used extensively in the remainder of the paper. We also present several examples of group actions which satisfy different types of Markovian properties, which show that all the classes we introduce are relevant.
	
	In~\Cref{section:algebraico} we study the Markovian properties in the setting of algebraic actions. We show that every action of a countable group on a compact metrizable group by continuous automorphisms has the TMP~(\Cref{P_algebraic_have_wtmp}) and that a large class of finitely presented expansive algebraic actions of amenable groups have the strong TMP~(\Cref{T-expansive algebraic to STMP}).
	
	In~\Cref{section:minimial} we study the Markovian properties in the setting of minimal group actions. In particular we show that minimal expansive group actions have the TMP if and only if they do not admit off-diagonal asymptotic pairs~(Theorem~\ref{T-minimal exp tmp iff noasym}).
	
	In~\Cref{section:ent_asympt} we present our results regarding the connection between topological entropy and asymptotic pairs in the setting of Markovian properties. First we give conditions under which the existence of off-diagonal asymptotic pairs of an action which satisfies the TMP give rise to positive entropy~(\Cref{color_sofic_UPE,color_naive_UPE}). Then we give conditions under which an action with positive entropy which satisfies the strong TMP has off-diagonal asymptotic pairs (Corollary~\ref{C-exp sTMP IE to asym}).
	
	In~\Cref{section:applications} we put together the results from the previous sections and present applications to algebraic actions, minimal actions and subshifts which are the support of some Markovian measure. In particular we prove Theorem~\ref{T-main} and show that a minimal expansive action of an amenable group which satisfies the strong TMP always has zero topological entropy~(\Cref{cor_minimalhasentropy0}). We also provide an example which shows that an analogue of~\Cref{T-main} does not hold in the free group for naive entropy, even for subshifts of finite type.
	
	\noindent {\bf Acknowledgments.}
	The authors wish to thank Tom Meyerovitch for interesting discussions and Ville Salo for sharing a beautiful proof that minimal subshifts on finitely generated groups cannot admit interchangeable patterns, on which our proof of~\Cref{P-minimal uniform TMP no asym} is based, and Tim Austin for pointing out the reference \cite{LyonsNazarov2011} to us. They are also grateful to the referee for helpful comments. Sebasti\'an Barbieri wishes to acknowledge that a considerable portion of this work was done while he was affiliated to the University of British Columbia.
Sebasti\'an Barbieri was partially supported by the ANR project CoCoGro (ANR-16-CE40-0005), the ANR project CODYS (ANR-18-CE40-0007) and FONDECYT grant 11200037.
Felipe Garc\'ia-Ramos was partially supported by CONACyT (287764). Hanfeng Li was partially supported by NSF grants DMS-1600717 and DMS-1900746.
	
	\section{Preliminaries}\label{section:preliminaries}
	
	Throughout this paper $\NN$ denotes the set of positive integers and $G$ denotes a countably infinite group with identity $e_G$. We denote by $F\Subset G$ a finite subset of $G$.
	
	Let $\delta >0$ and $K \Subset G$. A nonempty set $F \Subset G$ is said to be left $(K,\delta)$-invariant if $|KF \Delta F| \leq \delta |F|$. A sequence of nonempty  finite subsets $\{F_{n}\}_{n\in\mathbb{N}}$ of $G$ is said to be left
	\define{asymptotically invariant} or left \define{F\o lner} if it is eventually left $(K,\delta)$-invariant for every nonempty $K \Subset G$ and $\delta >0$. From this point forward we shall omit the usage of the word left and speak plainly about a F\o lner sequence. A countable group $G$ is \define{amenable} if it admits a F\o lner sequence. Elementary amenable groups and finitely generated groups of subexponential growth are all amenable.
	
For $n\in \NN$ we write $\Sym(n)$ for the group of permutations of $\{1,\dots,n\}$. A group $G$ is \define{sofic} if there exist a sequence $\left\{ n_{i}\right\}_{i\in \NN}$ of positive integers which goes to infinity and a sequence $\Sigma=\{\sigma_i \colon G\rightarrow \Sym(n_i) \}_{i=1}^{\infty}$ that satisfies
\begin{align*}
\lim_{i\rightarrow\infty} \frac{1}{n_i} \left\vert \left\{ v\in\left\{ 1,\dots,n_{i}\right\}
: \sigma_{i}(st)v=\sigma_{i}(s)\sigma_{i}(t)v\right\}  \right\vert  &
=1 \mbox{ for every } s,t\in G\\
\lim_{i\rightarrow\infty} \frac{1}{n_i}\left\vert \left\{  v\in\left\{  1,\dots,n_{i}\right\}
: \sigma_{i}(s)v\neq\sigma_{i}(t)v\right\}  \right\vert  & =1 \mbox{ for every } s\neq t\in G.
\end{align*}
In this case we say $\Sigma$ is a \define{sofic approximation sequence} of $G$. Amenable groups are all sofic. We refer the reader to \cite{CapraroLupini2015, Ceccherini-SilbersteinCoornaert2010, Pestov2008} for general information about amenable groups and sofic groups.

A (left) \define{action} of the group $G$ on $X$ is represented by $G \curvearrowright X$. In this paper we shall always assume that $X$ is a compact metrizable space and that $G$ acts by homeomorphisms. We denote by $d$ a compatible metric on $X$.

	We say $G \curvearrowright X$ is \define{expansive} if there exists $c >0$ such that whenever $x,y \in X$, if $x \neq y$ then there exists $g \in G$ such that $d(gx,gy)> c$. The value $c$ is called an \define{expansivity constant} of $G \curvearrowright X$.

Let $G\curvearrowright X$ be an action, $k\ge 2$ and $\varepsilon>0$. We say $(x_1, \dots, x_k)\in X^k$ is an \define{$\varepsilon$-asymptotic tuple} if there exists $F \Subset G$ such that for every $g \notin F$ and $1\le i, j\le k$, one has $d(gx_i, gx_j)\leq\varepsilon$. Furthermore, we say $(x_1, \dots, x_k)$ is an \define{asymptotic tuple} if it is $\varepsilon$-asymptotic for every $\varepsilon>0$.
\begin{notation}
We denote by $\R^{\varepsilon}_k(X,G)$ the set of all $(x_1, \dots, x_k)\in X^k$ which are $\varepsilon$-asymptotic and by $\R_k(X,G) = \bigcap_{\varepsilon >0}\R^{\varepsilon}_k(X,G)$ the set of asymptotic $k$-tuples.
\end{notation}

\begin{remark}
	When $G=\ZZ$, asymptotic pairs are sometimes called \define{bilateral asymptotic pairs} or \define{two-sided asymptotic pairs} to avoid confusion with the weaker notion of forward asymptotic pairs, for which the only requirement is that $\lim_{n \to +\infty}d(nx_1,nx_2) = 0$.
\end{remark}
Let $G \curvearrowright X$ and $G\curvearrowright Y$ be two actions of $G$. A function $\pi\colon X \to Y$ is \define{$G$-equivariant} if $g\pi(x)=\pi(gx)$ for every $g \in G$ and $x \in X$. A $G$-equivariant function as above is called a \define{topological factor} if it is continuous and surjective, and is called a \define{topological conjugacy} if it is a homeomorphism. If there exists a topological conjugacy between $X$ and $Y$ we say $G \curvearrowright X$ and $G\curvearrowright Y$ are \define{topologically conjugate}.

Let $G \curvearrowright X$ be an action. A Borel probability measure $\mu$ on $X$ is \define{$G$-invariant} if for every Borel set $A \subset X$ and $g \in G$, we have $\mu(A) = \mu(g^{-1}A)$. In this case we say that $G \curvearrowright (X, \mu)$ is a \define{probability measure preserving (p.m.p.) action}. For p.m.p. actions $G \curvearrowright (X,\mu)$ and $G\curvearrowright (Y,\nu)$ we say that $G\curvearrowright (Y,\nu)$ is a \define{factor} of $G \curvearrowright (X, \mu)$ if there is a $G$-invariant conull set $X' \subset X$ and a $G$-equivariant measurable map $\pi\colon X' \to Y$ such that $\mu(\pi^{-1}(A)) = \nu(A)$ for every Borel set $A \subset Y$. Furthermore, we say that $G \curvearrowright (X, \mu)$ and $G\curvearrowright (Y,\nu)$ are \define{isomorphic} if there are $G$-invariant conull sets $X' \subset X$ and $Y' \subset Y$ and a $G$-equivariant bimeasurable bijection $\pi\colon X' \to Y'$ such that $\mu(\pi^{-1}(A)) = \nu(A)$ for every Borel set $A \subset Y'$.
	\subsection{Entropy theory}
	
	In what follows we shall provide several definitions and results on entropy theory. For a more detailed exposition of these topics we refer the reader to~\cite{Downarowicz2011, pollicott1998dynamical, Walters1982} for ample background on entropy theory of $\ZZ$-actions, and to~\cite{KerrLiBook2016, Ollagnier1985book} for entropy theory of  actions of amenable and sofic groups.
	
	\subsubsection{Topological entropy for actions of amenable groups}
	
	Given two open covers $\mathcal{U},\mathcal{V}$ of $X$ we define their \define{join} by $\mathcal{U} \vee \mathcal{V} = \{U \cap V : U \in \mathcal{U}, V\in \mathcal{V}   \}$. For $g \in G$ let $g\mathcal{U} = \{gU : U \in \mathcal{U}\}$ and denote by $N(\mathcal{U})$ the smallest cardinality of a subcover of $\mathcal{U}$. If $F$ is a nonempty finite subset of $G$, denote by $\mathcal{U}^F$ the join
	
	$$\mathcal{U}^F = \bigvee_{g \in F}g^{-1}\mathcal{U}.$$

		Let $G$ be an amenable group, $G \curvearrowright X$ an action, $\mathcal{U}$ an open cover of $X$ and $\{F_n\}_{n \in \NN}$ a F\o lner sequence for $G$. We define the \define{topological entropy of
			$G \curvearrowright X$ with respect to $\mathcal{U}$} as%
		\[
		\htop(G \curvearrowright X,\mathcal{U})=\lim_{n\rightarrow\infty}\frac{1}{\left\vert
			F_{n}\right\vert }\log N(\mathcal{U}^{F_{n}}).
		\]
	The function $F \mapsto \log N(\mathcal{U}^{F})$ is subadditive and thus the limit exists and does not depend on the choice of F\o lner sequence, see for instance~\cite{OrnWei1987,Krieger2007} \cite[page 220]{KerrLiBook2016}. The \define{topological entropy} of $G \curvearrowright X$ is defined as \[
	\htop(G \curvearrowright X)=\sup_{\mathcal{U}}\htop(G \curvearrowright X,\mathcal{U}).
	\]
	
	\subsubsection{Naive topological entropy}

	 Let $G \curvearrowright X$ be an action and $\mathcal{U}$ an open cover of $X$. We define the
		\define{naive topological entropy of $G\curvearrowright X$ with respect to $\mathcal{U}$} as%
		\[
		\hnaive(G \curvearrowright X,\mathcal{U})=\inf_{\varnothing \neq F\Subset G}\frac{1}{\left\vert
			F\right\vert }\log N(\mathcal{U}^{F}).
		\]
  The \define{naive topological
		entropy} of $G\curvearrowright X$ is defined as
	\[
	\hnaive(G \curvearrowright X)=\sup_{\mathcal{U}}\hnaive(G \curvearrowright X,\mathcal{U}).
	\]

	The notion of naive entropy was introduced by Burton~\cite{burton2017naive}. He showed that in the case of a non-amenable group $\hnaive(G \curvearrowright X)$ can only take the values $\{0, +\infty\}$.

	\subsubsection{Sofic topological entropy}
	
	The following notion of topological entropy for sofic group actions was introduced by Kerr and Li~\cite{KerrLi2011} following the breakthrough of Bowen for probability-measure-preserving actions~\cite{Bowen2010_2}.
	
    Let $G\curvearrowright X$ be an action.
	Let $F\Subset G$, $\delta>0,$ $n\in\mathbb{N}$, and $\sigma\colon G\to
	\Sym(n)$. We define $\Map(d,F,\delta,\sigma)$ as the set of all maps
	$\varphi\colon \left\{  1,\dots,n\right\}  \to  X$ such that
	\[
	\left(  \frac{1}{n}\sum_{v=1}^{n}d(\varphi(\sigma(s)v),s\varphi
	(v))^{2}\right)  ^{1/2}\leq\delta \ \mbox{for every }s \in F.
	\]

Write $N_{\varepsilon}(Y, d_{\infty})$ for the maximum cardinality of a subset $Y'$ of $Y \subset X^{\{1,\dots,n\}}$ such that whenever $\varphi_1,\varphi_2$ are distinct in $Y'$ then $\max_{v \in \{1,\dots,n\}} d(\varphi_1(v),\varphi_2(v)) \geq \varepsilon$.

	Let $G$ be a sofic group and $\Sigma = \{\sigma_i \colon G \to \Sym(n_i)\}$ a sofic approximation sequence for $G$. The \define{topological sofic entropy of }$G\curvearrowright X$
	\define{with respect to }$\Sigma$ is
	\[
	\hsof(G\curvearrowright X)=\sup_{\varepsilon>0}\inf_{F\Subset G}\inf_{\delta>0}\limsup_{i\to\infty}\frac{1}{n_{i}}\log N_{\varepsilon}%
	(\Map(d,F,\delta,\sigma_i),d_{\infty}).
	\]	
The value of $\hsof(G\curvearrowright X)$ does not depend on the choice of $d$ by \cite[Proposition 10.25]{KerrLiBook2016}.

\begin{remark}
	If $G$ is a countable amenable group, then the classical topological entropy of a $G$-action coincides with the naive topological entropy~\cite[Theorem 6.8]{DownarowiczFrejRomagnoli2016} and the topological sofic entropy~\cite[Theorem 5.3]{kerrli2013soficity} \cite[Theorem 10.37]{KerrLiBook2016} for any sofic approximation sequence.
\end{remark}

\subsubsection{Measure-theoretical entropy and the variational principle}

Let $G \curvearrowright (X,\mu)$ be a p.m.p. action. For a finite partition $\mathcal{P}$ of $X$ consisting of Borel sets the \define{Shannon entropy} of $\mathcal{P}$ with respect to $\mu$ is given by \[ H_{\mu}(\mathcal{P}) = \sum_{A \in \mathcal{P}} -\mu(A)\log \mu(A). \]
For an amenable group $G$, the \define{measure-theoretical entropy of $G \curvearrowright (X,\mu)$ with respect to $\mathcal{P}$} is given by\[ h_{\mu}(G\curvearrowright X, \mathcal{P}) = \lim_{n\rightarrow\infty}\frac{1}{\left\vert
	F_{n}\right\vert }\log H_{\mu}(\bigvee_{g \in F_n} g^{-1}\mathcal{P}), \]
where $\{F_n\}_{n \in \NN}$ is a F\o lner sequence. The \define{measure-theoretical entropy} of $G\curvearrowright (X,\mu)$ is the supremum of $h_{\mu}(G\curvearrowright X, \mathcal{P})$ taken over all finite partitions: \[h_{\mu}(G\curvearrowright X) = \sup_{\mathcal{P} \mbox{ finite}}h_{\mu}(G\curvearrowright X, \mathcal{P}).\]

The variational principle relates the topological entropy with the measure-theoretical entropy through the following formula.

\begin{theorem}\cite[Theorem 5.2.7]{Ollagnier1985book}
	Let $G$ be an amenable group and $G \curvearrowright X$ an action, we have\[
	\htop(G \curvearrowright X) = \sup_{\mu \in \mathcal{M}(G\curvearrowright X)}h_{\mu}(G\curvearrowright X).
	\]
	where $\mathcal{M}(G\curvearrowright X)$ denotes the space of all Borel $G$-invariant probability measures on $X$.
\end{theorem}

	For a p.m.p. action $G \curvearrowright (X,\mu)$ of a sofic group $G$, there is the notion of the \define{measure-theoretical sofic entropy} $h_{\mu}^{\Sigma}(G \curvearrowright X)$ with respect to a sofic approximation sequence $\Sigma$ for $G$ introduced by Bowen~\cite{Bowen2010_2}. An analogous variational principle for sofic entropy was established by Kerr and Li~\cite{KerrLi2011}.
	
\subsection{Local entropy theory}
	
In the seminal papers~\cite{Blanchard1992, Blanchard1993}, Blanchard introduced
	the notion of entropy pairs. These can be used to give a local characterization of positive entropy. This work was the
	birth of what is now called local entropy theory (see~\cite{GlasnerYe2009}
	for a survey). A combinatorial counterpart of that theory is the notion of independence which can also be used in the context of non-amenable group actions. For an introduction to the subject see~\cite[Chapter 12]{KerrLiBook2016}.
	
\begin{notation}
	 For each $k\in \NN$, we represent by $\triangle_k(X)$ the diagonal of $X^k$, that is, the set of all tuples $(x,x,\dots,x) \in X^k$.
\end{notation}	

	\subsubsection{Orbit independence entropy tuples}

		Let $G\curvearrowright X$ be an action and $\mathbf{A}=(A_{1},\dots,A_{n})$ a tuple of
		subsets of $X$. We say $J\subset G$ is an \define{independence set}\textit{
			for} $\mathbf{A}$ if for every nonempty $I\Subset J$ and every $\phi \colon I\rightarrow\{1,\dots,n\}$ we have
		\[
		\bigcap_{s\in I}s^{-1}A_{\phi(s)}\neq\varnothing\text{.}%
		\]
		We define the \define{independence density of} $\mathbf{A}$ (over $G$)
		to be the largest $q\geq0$ such that every set $F\Subset G$ has a
		subset of cardinality at least $q\left\vert F\right\vert $ which is an
		independence set for $\mathbf{A}$.

	\begin{definition}\cite[Definition 3.2]{KerrLi2013}
		Fix an integer $k \geq 1$. We say a tuple $\mathbf{x}=(x_{1},\dots,x_{k})\in X^{k}$ is an \define{orbit independence entropy tuple (orbit IE-tuple)} if for every product neighborhood
		$U_{1}\times \dots \times U_{k}$ of $\mathbf{x}$ the tuple $\textbf{U}= (U_{1},\dots,U_{k})$ has
		positive independence density. We denote the set of orbit IE-tuples of length $k$ by
		$\IE_{k}(X,G)$.
	\end{definition}

		An open cover $\mathcal{U}$ of $X$ is called standard if $\mathcal{U}=\{U_1, U_2\}$ such that none of $U_1, U_2$ is dense in $X$. We say an action $G\curvearrowright X$ has \define{naive uniform positive entropy (naive UPE)} if for
		each standard  cover $\mathcal{U}$ we have that $\hnaive(G\curvearrowright X, \mathcal{U})>0$. The notion of UPE was defined in the context of $\ZZ$-actions by Blanchard~\cite{Blanchard1992} and naive UPE is a natural generalization in the context of naive topological entropy.
	
		\begin{theorem}\label{theorem_naive_IE_cites}
		Let $G\curvearrowright X$ be an action.
		
		\begin{enumerate}
			\item~\cite[Theorem 2.5]{LiRong2019} If $\IE_2(X,G)\setminus \triangle_2(X) \neq \varnothing$, then $\hnaive(G \curvearrowright X)>0$.

			\item~\cite[Theorems 12.19 and 12.23]{KerrLiBook2016} If $G$ is amenable and $\htop(G\curvearrowright X) >0$, then for each $k\in \NN$ there is some $(x_1, \dots, x_k)\in \IE_k(X,G)$ such that $x_1, \dots, x_k$ are distinct.
			
			\item~\cite[Theorem 2.5]{LiRong2019} $\mathtt{IE}_2(X,G)=X^{2}$ if and only if $G\curvearrowright X$ has naive UPE.
			
			\item~\cite[Lemma 6.2]{KerrLi2013}
			We have that $\overline{\bigcup_{\mu\in \mathcal{M}(G \curvearrowright X)}\mathtt{supp}(\mu)}\subset \mathtt{IE}_1(X,G)$, where $\mathcal{M}(G \curvearrowright X)$ is the set of $G$-invariant Borel probability measures on $X$ and $\mathtt{supp}(\mu)$ denotes the support of $\mu$.
			
		\end{enumerate}
	\end{theorem}
	
	Consequently with the third point of~\Cref{theorem_naive_IE_cites}, for $k\ge 2$ we shall say that $G \curvearrowright X$ has \define{naive UPE of order $k$} if $\mathtt{IE}_k(X,G)=X^{k}$. We also say that $G \curvearrowright X$ has \define{naive UPE of all orders} if it has naive UPE of order $k$ for all $k\ge 2$.

	\subsubsection{Sofic independence entropy tuples}

	Let $G$ be a sofic group, $G\curvearrowright X$ an action and $\mathbf{A}=(A_{1},\dots,A_{k})$ a tuple of
	subsets of $X$. Given $F\Subset G$, $\delta>0$, $n\in\mathbb{N}$ and
	$\sigma\colon G\rightarrow \Sym(n)$, we say $J\subset\left\{  1,\dots,n\right\}$ is a
	$(d,F,\delta,\sigma)$-\define{independence set for $\mathbf{A}$} if for every
	$\omega\colon J\rightarrow\left\{  1,\dots,k\right\}  $ there exists $\varphi\in
	\Map(d,F,\delta,\sigma)$ such that $\varphi(a)\in A_{\omega(a)}$ for every $a\in
	J$.

\begin{definition} \cite[Definitions 12.33 and 12.34]{KerrLiBook2016}
	Let $G$ be a sofic group, $G\curvearrowright X$ an action and $\Sigma=\left\{  \sigma_{i} \colon G \to \Sym(n_{i})\right\}
	_{i\in\mathbb{N}}$ a sofic approximation sequence for $G$. We say  a tuple $\mathbf{A}=(A_{1},\dots,A_{k})$ of subsets of $X$
has  \define{positive upper independence density over $\Sigma$} if there exists $q>0$ such that for every $F\Subset G$ and
	$\delta>0$ there exists an infinite set of  $i$ for which $\mathbf{A}$ has a $(d,F,\delta,\sigma_{i})$-independence set of cardinality at least $qn_{i}$. This property does not depend on the choice of the metric $d$, see \cite[Lemma 10.24]{KerrLiBook2016}.

We say $\mathbf{x}=(x_{1},\dots,x_{k})\in X^{k}$ is a \define{sofic independence entropy tuple with respect to $\Sigma$ ($\Sigma$-IE-tuple)} if for every product
	neighborhood $U_{1}\times\cdot\cdot\cdot\times U_{k}$ of
	$\mathbf{x}$ the tuple $\mathbf{U}=(U_1,\dots,U_k)$ has positive upper independence density over $\Sigma$.
We denote
	the set of $\Sigma$-\define{IE-tuples} of length $k$ by $\IE_{k}^{\Sigma}(X,G)$.
\end{definition}

	\begin{theorem}\cite[Theorem 12.39]{KerrLiBook2016} \label{theorem_sofic_IE_cites}
		 Let $G$ be sofic, $\Sigma$ a sofic approximation sequence for $G$,  and
		$G\curvearrowright X$ an action. Then
		\begin{enumerate}
			\item $\IE_{1}^{\Sigma}(X,G) \neq \varnothing$ if and only if $\hsof(G \curvearrowright X)\geq 0$.
			\item $\IE_{2}^{\Sigma}(X,G) \setminus \triangle_2(X) \neq \varnothing$ if and only if $\hsof(G \curvearrowright X)>0$.
		\end{enumerate}
	\end{theorem}

	Given a sofic group $G$ and a sofic approximation sequence $\Sigma$ for $G$, for $k\ge 2$ we say $G\curvearrowright X$ has \define{sofic UPE of order $k$} if $\mathtt{IE}_k^{\Sigma}(X,G)=X^{k}$. We also say that $G \curvearrowright X$ has \define{sofic UPE of all orders} if it has sofic UPE of order $k$ for all $k\ge 2$.

\begin{remark}
	Every $\Sigma$-IE-tuple is an orbit IE-tuple~\cite[Proposition 4.6]{KerrLi2013} but there might exist orbit IE-tuples that are not $\Sigma$-IE-tuples for any $\Sigma$. Nonetheless, if $G$ is amenable then every orbit IE-tuple is a $\Sigma$-IE-tuple for every $\Sigma$~\cite[Theorem 4.8]{KerrLi2013}. In this case we call naive/sofic UPE simply \define{UPE}.
\end{remark}

	\subsection{Shift spaces and the pseudo-orbit tracing property}
	
		Given a nonempty finite set $\Lambda$, we say $X \subset \Lambda^G$ is a \define{$G$-subshift} or \define{$G$-shift space} if $X$ is closed under the product topology and $G$-invariant under the left shift action $G\curvearrowright \Lambda^G$ given by
		\[  gx(h) = x(g^{-1}h) \mbox{ for every }g,h \in G. \]
		
		Elements $x \in \Lambda^G$ are called \define{configurations}. A \define{pattern} is an element $p \in \Lambda^A$ for some $A \Subset G$. For a pattern $p\in \Lambda^A$ let us denote by $[p] = \{x \in \Lambda^G : x|_A = p \}$ the \define{cylinder} centered at $p$. A set $X \subset \Lambda^G$ is a subshift if and only if there exists a set $\mathcal{F}$ of patterns which generates $X$, that is, $X = X_{\mathcal{F}}$ where,
		
		\[ X_{\mathcal{F}} = \Lambda^G \setminus \bigcup_{g \in G, p \in \mathcal{F}} g[p].  \]
				In other words, $X$ is the set of all configurations $x \in \Lambda^G$ where no translation of a pattern $p \in \FF$ appears in $x$. If there exists a finite set $\mathcal{F}$ of patterns such that $X = X_{\mathcal{F}}$, then we say that $X$ is a \define{subshift of finite type (SFT)}.
		
			Let $G\curvearrowright X$ be an action, $\delta>0$ and $S\Subset G$. An $(S,\delta)$
			\define{pseudo-orbit} is a sequence $\{x_{g}\}_{g\in G}$ of elements in $X$ such that $d(s
			x_{g},x_{sg})< \delta$ for every $s\in S$ and $g\in G.$ We say a
			pseudo-orbit is $\mathbf{\varepsilon}$-\define{traced} \define{by }$x\in X$ if
			$d(gx,x_{g})\leq\varepsilon$ for every $g\in G.$

			An action $G\curvearrowright X$ has the \define{pseudo-orbit tracing property (POTP)} if
			for  every $\varepsilon>0$ there exist $\delta>0$ and $S \Subset G$ such
			that  any $(S,\delta)$ pseudo-orbit is $\varepsilon$-traced by some point
			$x\in X.$
	The POTP is also known as \define{shadowing}. For a shift space, POTP coincides with the notion of SFT~\cite{Walters1978,Oprocha2008} \cite[Theorem 3.2]{ChungLee2018}.

\section{Topological Markov properties: results and examples}\label{section:resultados_basicos}

	\subsection{Topological Markov properties}
	
	In what follows the following notation shall be useful.

\begin{notation}
For an action $G \curvearrowright X$, a set $K \subset G$, and $x,y \in X$, we shall write
	\[d_{K}(x,y) = \sup_{g \in K}d(gx,gy).\]
\end{notation}	

	\begin{definition}
		Let $G \curvearrowright X$ be an action, and $\varepsilon, \delta>0$. We say that $B \Subset G$ is an
		\define{$(\varepsilon,\delta)$-memory set} for $A\subset B$ if for every pair
		$x,y\in X$  satisfying $d_{B\setminus A}(x,y)\leq\delta$ there exists
		$z\in X$ such that $d_B(x,z)\leq\varepsilon$ and
		$d_{G \setminus A}(y,z)\leq\varepsilon$.
	\end{definition}
	
In~\cite{ChandgotiaHanMarcusMeyerovitchPavlov2014} the authors introduced the notion of topological Markov field. A \define{topological Markov field} is a subshift $X \subset \Lambda^{\ZZ}$ which has the following property: for every pair of configurations $x,y \in X$ for which there are $n,m \in \ZZ$ so that $n <m$, $x_n = y_n$ and $x_m = y_m$, we have that the configuration $z \in \Lambda^{\ZZ}$ given by \[ z(i) = \begin{cases}
x(i) & \mbox{ if } i \in \{ n,\dots,m\}\\
y(i) & \mbox{ if } i \in \ZZ \setminus \{ n,\dots,m\}\\
\end{cases} \]
belongs to $X$. In terms of memory sets, this property is stating that for an adequate choice of $\varepsilon,\delta$ (which depends upon $d$) an $(\varepsilon,\delta)$-memory set for an interval $A = \{n+1,\dots, m-1\}$ is the interval $B = \{n,\dots,m\}$. The notion of topological Markov field was in turn generalized to the concept of (strong) topological Markov property for subshifts on arbitrary countable groups in~\cite{Richifest}. We shall give a natural generalization of these properties to the non-symbolic setting.

	\begin{definition}
		An action $G \curvearrowright X$  has the:
		\begin{enumerate}
			\item \define{Topological Markov property (TMP)} if for
			every $\varepsilon>0$ there exists $\delta>0$ such that every $A\Subset G$
			admits an $(\varepsilon,\delta)$-memory set for $A$.
			\item \define{Strong topological Markov property (strong TMP)} if for
			every $\varepsilon>0$ there exist $\delta>0$ and $F \Subset G$ containing the identity,  such that
			every $A\Subset G$ admits $FA$ as an $(\varepsilon,\delta)$-memory set
			for $A$.
			\item \define{Mean topological Markov property (mean TMP)} if for
			every $\varepsilon>0$ there exists $\delta>0$ and an increasing sequence $\{F_{n}\}_{n \in\mathbb{N}}$ of finite subsets of $G$ with union $G$ such that  for each $n \in\mathbb{N}$ there is an
			$(\varepsilon,\delta)$-memory set $\widetilde{F_{n}}$ for $F_{n}$ so that
			$|\widetilde{F_{n}}\setminus F_{n}| = o(|F_{n}|)$.
			
		\end{enumerate}
	\end{definition}

	Now we introduce uniform versions of the topological Markov properties. These are mainly technical conditions needed for the results in the following sections.
	
	Let $B \Subset G$. We say that a set $V \subset G$ is \define{$B$-separated} if $Bv_1 \cap Bv_2 = \varnothing$ whenever $v_1 \neq v_2$ and $v_1,v_2 \in V$.
	
	\begin{definition}
		Let $G\curvearrowright X$ be an action, $\varepsilon, \delta>0$,  and $A \Subset G$. We say that an $(\varepsilon,\delta)$-memory set $B$ for  $A$ is \define{uniform}, if for any $B$-separated $V\Subset G$, $x_v\in X$ for $v\in V$ and  $y \in X$  such that $d_{(B\setminus A)v}(x_v,y) \leq \delta$ for every $v \in V$, there exists $z \in X$ such that $d_{Bv}(z,x_v) \leq \varepsilon$ for every $v \in V$ and $d_{G \setminus AV}(z,y) \leq \varepsilon$.
	\end{definition}
	
	\begin{remark}
		If an $(\varepsilon,\delta)$-memory set $B$ is uniform, it follows by compactness that in fact the above property holds for any $B$-separated $V\subset G$, even if it is not finite.
	\end{remark}
	
	\begin{definition}
		We say that an action $G\curvearrowright X$ has the \textbf{uniform (mean/strong) topological Markov property} if it satisfies the (mean/strong) topological Markov property with a memory set which is uniform.
	\end{definition}
	
	If two metrics $d$ and $d'$ on $X$ are both compatible with the given compact topology on $X$, then for any $\varepsilon>0$ there is some $\varepsilon'>0$ such that for any $x, y\in X$ satisfying $d'(x, y)\le \varepsilon'$ one has $d(x, y)\le \varepsilon$. It follows that the topological Markov properties do not depend upon the choice of the metric as long as it is compatible with the topology of $X$. Hence, they are invariants of topological conjugacy.

\subsection{Structural results}

	We say that a continuous pseudometric $\rho $ on $X$ is \define{dynamically
		generating} for an action $G \curvearrowright X$ \cite[Definition 9.35]{KerrLiBook2016} if for any distinct $x,y\in X$ one has $\sup_{s\in G}\rho
	(sx,sy)>0$.
As in the case of a metric, for $K \subset G$ we denote $\rho_K(x,y)= \sup_{s \in K}\rho(sx,sy)$.

\begin{proposition}
	\label{P-generating} Let $\rho $ be a dynamically generating continuous
	pseudometric on $X$ for an action $G\curvearrowright X$. Then the following hold.
	
	\begin{enumerate}
		\item $G\curvearrowright X$ has the TMP if and only if for any $\varepsilon >0$
		there is some $\delta >0$ such that for any $A\Subset G$ there is
		some $B\Subset G$ containing $A$ so that for any $x,y\in X$ with $\rho_{B\setminus A}(x,y)\le \delta$ there is some $z\in X$ with $\rho_{B}(z,x)\le \varepsilon$ and $\rho_{G \setminus A}(z,y)\le \varepsilon$.
		
		\item $G\curvearrowright X$ has the strong TMP if and only if for any $\varepsilon >0$
		there are some $\delta >0$ and $F\Subset G$ containing $e_G$ such that for any $A\Subset G$ and any $x,y\in X$ with $\rho_{FA\setminus A}
		(x,y)\le \delta $ there is some $z\in X$ with $\rho_{FA} (z,x)\le \varepsilon $ and $\rho_{G\setminus A}(z,y)\le \varepsilon$.
	\end{enumerate}
\end{proposition}

\begin{proof}
	We shall prove (1). The proof for (2) is similar.
	
	Take a function $f\colon G\rightarrow (0,1]$ with $\sum_{s\in G}f(s)<\infty $ and $%
	f(e_{G})=1$. For any $x,y\in X$, put
	\begin{equation*}
	d(x,y)=\sum_{s\in G}f(s)\rho (sx,sy).
	\end{equation*}%
	Then it is easily checked that $d$ is a compatible metric on $X$ and $d\geq
	\rho $.
	
	We prove the \textquotedblleft only if\textquotedblright\ part first. Assume
	that $G\curvearrowright X$ has the TMP. Let $\varepsilon >0$. Then there is
	some $\delta >0$ such that for any $A\Subset G$ there is some
	$B_{A}\Subset G$ containing $A$ so that for any $x,y\in X$ with $d_{B_{A}\setminus A}(x,y)\le \delta $ there is some $z\in X$ with $d_{B_A}(z,x)\le \frac{\varepsilon}{2}$ and $d_{G\setminus A}(z,y)\le \frac{\varepsilon}{2}$. Take a $K\Subset G$ with $e_{G}\in K=K^{-1}$
	such that $\mathrm{diam}(X,\rho)\sum_{s\in G\setminus K}f(s)<\frac{\delta}{2}$.
	Put $\delta ^{\prime }=\frac{1}{2|K|}\min(\varepsilon,\delta)>0$. For $A \Subset G$, let $B^{\prime }=KB_{KA}$. Then $B^{\prime }$ is a finite subset of $G$ containing $A$, and $K(B_{KA}\setminus KA)\subset
	B^{\prime }\setminus A$. Let $x,y\in X$ with $\rho_{B^{\prime
		}\setminus A} (x,y)\le \delta^{\prime }$. For any $s\in B_{KA}\setminus
	KA$, we have
	\begin{equation*}
	d(sx,sy)=\sum_{t\in K}f(t)\rho (tsx,tsy)+\sum_{t\in G\setminus K}f(t)\rho
	(tsx,tsy)<|K|\delta ^{\prime }+\frac{\delta}{2}\leq \delta .
	\end{equation*}%
	Thus there is some $z\in X$ with $d_{B_{KA}}(z,x)\le \frac{\varepsilon}{2}$
	and $d_{G\setminus KA}(z,y)\le \frac{\varepsilon}{2}$. For any $s\in B_{KA}$, we have $\rho (sz,sx)\leq d(sz,sx)\le \frac{\varepsilon}{2}$. For any $s\in
	G\setminus KA$, we have $\rho (sz,sy)\leq d(sz,sy)\le \frac{\varepsilon}{2}$. For any $s\in KA\setminus A$, we have
	\begin{equation*}
	\rho (sz,sy)\leq \rho (sz,sx)+\rho (sx,sy)\leq d(sz,sx)+\delta^{\prime}<\frac{\varepsilon}{2}+\delta ^{\prime }\leq \varepsilon .
	\end{equation*}
	For any $s\in B^{\prime}\setminus B_{KA}$, we have
	\begin{equation*}
	\rho(sz, sx)\le \rho(sz, sy)+\rho(sy, sx)\le \frac{\varepsilon}{2}+\delta'\le \varepsilon.
	\end{equation*}
	This proves the \textquotedblleft only if\textquotedblright\ part.
	
	Next we prove the \textquotedblleft if\textquotedblright\ part. Let $%
	\varepsilon >0$. Take a  $K\Subset G$ with $e_{G}\in K=K^{-1}$ such that $\mathrm{diam}(X,\rho )\sum_{s\in G\setminus K}f(s)<\frac{\varepsilon}{4}$. Put $\varepsilon ^{\prime }=\frac{\varepsilon}{8|K|}>0$. By assumption, there is some $\delta ^{\prime }>0$ such that for any  $A\Subset G$ there is some $B_{A}^{\prime }\Subset G$ containing $A$ so that for any $x,y\in X$
	with $\rho_{B_{A}^{\prime }\setminus A}(x,y)\le \delta ^{\prime }$
	there is some $z\in X$ with $\rho_{B_A^{\prime}} (z,x)\le \varepsilon ^{\prime }$ and $\rho_{G \setminus A} (z,y)\le \varepsilon^{\prime}$. Put $\delta =\min (\delta ^{\prime },\varepsilon ^{\prime })>0$. Let $A\Subset G$. Put $B=B_{K^{2}A}^{\prime }\supset
	K^{2}A\supset A$. Let $x,y\in G$ with $d_{B\setminus
		A}(x,y)\le \delta $. Then
	\begin{equation*}
	\max_{s\in B_{K^{2}A}^{\prime }\setminus K^{2}A}\rho (sx,sy)\leq
	d_{B\setminus A}(x,y)\le \delta \leq \delta ^{\prime }.
	\end{equation*}%
	Thus there is some $z\in X$ with $\rho_{B^{\prime}_{K^{2}A}} (z,x)\le \varepsilon ^{\prime }$ and $\rho_{G\setminus K^{2}A}(z,y)\le \varepsilon ^{\prime }$. For any $s\in KA$, we have
	\begin{equation*}
	d(sz,sx)=\sum_{t\in K}f(t)\rho (tsz,tsx)+\sum_{t\in G\setminus K}f(t)\rho
	(tsz,tsx)<|K|\varepsilon ^{\prime }+\frac{\varepsilon}{4}\leq \frac{\varepsilon}{2}.
	\end{equation*}%
	For any $\gamma \in K^{2}A\setminus A\subset B\setminus A$, we have
	\begin{equation*}
	\rho (\gamma z,\gamma y)\leq \rho (\gamma z,\gamma x)+\rho (\gamma x,\gamma
	y)\le \varepsilon ^{\prime }+d(\gamma x,\gamma y)\le \varepsilon ^{\prime }+\delta.
	\end{equation*}%
	Thus $\rho (\gamma z,\gamma y)\le \varepsilon ^{\prime }+\delta $ for all $%
	\gamma \in G\setminus A$. Then for any $s\in G\setminus KA$, we get
	\begin{equation*}
	d(sz,sy)=\sum_{t\in K}f(t)\rho (tsz,tsy)+\sum_{t\in G\setminus K}f(t)\rho
	(tsz,tsy)\le |K|(\varepsilon ^{\prime }+\delta )+\frac{\varepsilon}{4}\le \frac{\varepsilon}{2}.
	\end{equation*}%
	For any $s\in KA\setminus A$, we have
	\begin{equation*}
	d(sz,sy)\leq d(sz,sx)+d(sx,sy)\le \frac{\varepsilon}{2}+\delta <\varepsilon.
	\end{equation*}%
	For any $s\in B\setminus KA$, we have
	\begin{equation*}
	d(sz,sx)\leq d(sz,sy)+d(sy,sx)\le \frac{\varepsilon}{2}+\delta <\varepsilon.
	\end{equation*}%
	Therefore $G\curvearrowright X$ has the TMP. This proves the \textquotedblleft
	if\textquotedblright\ part.
\end{proof}

Let $X\subset \Lambda^G$ be a $G$-subshift. For any $x,y\in X$, put
\begin{equation*}
\rho (x,y)= \begin{cases}
1 & \mbox{ if } x(e_G) \neq y(e_G)\\
0 & \mbox{ otherwise}
\end{cases}.
\end{equation*}%
Then $\rho $ is a dynamically generating continuous pseudometric for $G\curvearrowright X$.
Thus from~\Cref{P-generating} we get the following corollary.

\begin{corollary}
	\label{C-subshift} Let $X\subset \Lambda^G$ be a $G$-subshift. The following hold.
	
	\begin{enumerate}
		\item $G\curvearrowright X$ has the TMP if and only if for any $A\Subset G$ there is some $B\Subset G$ containing $A$ so that for
		any $x,y\in X$ with $x=y$ on $B\setminus A$ the configuration $z\in \Lambda^G$ such that $z|_{A}=x|_{A}$ and $z|_{G \setminus A} =y|_{G\setminus A}$ lies in $X$.
		
		\item $G\curvearrowright X$ has the strong TMP if and only if there is some $F\Subset G
		$ containing the identity, such that for any $A\Subset G$ and any $x,y\in X$ with $x=y$ on $%
		AF\setminus A$ the configuration $z\in \Lambda^G$ such that $z|_{A}=x|_{A}$ and $z|_{G \setminus A} =y|_{G\setminus A}$ lies in $X$.
	\end{enumerate}
\end{corollary}

In particular, we recover the original definitions introduced for subshifts in~\cite[Section 2.5.1]{Richifest}. Before stating the relation between the POTP and the TMP properties, we shall state the following technical lemma.

\begin{lemma}
	\label{L-PO} Let $G \curvearrowright X$ be an action, $F\Subset G$ with $e_{G}\in F=F^{-1}$ and $\delta >0$.
	There exists $\eta >0$ such that for any $A\Subset G$, any $FA$-separated $%
	V\subset G$, any $x_{v}\in X$ for $v\in V,$ and $y\in X$ satisfying $%
	d_{(FA\setminus A)v}(x_{v},y)\leq\eta $ for all $v\in V$, we have that $\{w_{\gamma }\}_{\gamma \in G}$ is an $(F,\delta )$
	pseudo-orbit, where
	\begin{equation*}
	w_{\gamma }=\left\{
	\begin{array}{cc}
	\gamma x_{v} & \mbox{if \ }\gamma \in Av\mbox{ for some }v\in V\\
	\gamma y & \mbox{if }\gamma \in G\setminus AV%
	\end{array}%
	\right. .
	\end{equation*}
\end{lemma}

\begin{proof}
	Take $0<\eta < \delta $ such that, for any $x,y\in X$ with $d(x,y)\leq \eta$, we have $d_F(x,y)<\delta $. Let $A,V,x_{v}$ and $y$ be as in the statement. We need to check $%
	d(sw_{\gamma },w_{s\gamma })<\delta $ for all $s\in F$ and $\gamma \in G$.
	Fix $s\in F$. The only nontrivial cases are (1) $\gamma \in Av$ for some $%
	v\in V$ and $s\gamma \in G\setminus AV$, and (2) $\gamma \in G\setminus AV$
	and $s\gamma \in Av$ for some $v\in V$.
	
	Consider the case (1). Then $s\gamma \in (FA\setminus A)v$, and hence
	\begin{equation*}
	d(sw_{\gamma },w_{s\gamma })=d(s\gamma x_{v},s\gamma y)\leq\eta < \delta .
	\end{equation*}
	
	Now consider the case (2). Note that $s\gamma \in Av$ and $s^{-1}(s\gamma
	)\in G\setminus AV$, and hence from the case (1) we have
	\begin{equation*}
	d(s^{-1}w_{s\gamma },w_{\gamma })\leq\eta .
	\end{equation*}%
	From our choice of $\eta $ we get $d(w_{s\gamma },sw_{\gamma })<\delta $ as
	desired.
\end{proof}

\begin{proposition}\label{prop_POTP_implies_sTMP}
	Let $G \curvearrowright X$ be an action. The following implications hold.
	
	\begin{center}
		\begin{tikzpicture}
		\node (A) at (0,0) {$G \curvearrowright X$ \mbox{has POTP}};
		\node (B) at (5.5,0) {$G \curvearrowright X$ \mbox{has uniform strong TMP}};
		\node (C) at (11,0) {$G \curvearrowright X$ \mbox{has uniform TMP}};
		\node (D) at (5.5,-1.5) {$G \curvearrowright X$ \mbox{has strong TMP}};
		\node (E) at (11,-1.5) {$G \curvearrowright X$ \mbox{has TMP}};
		\draw[ -implies,double equal sign distance] (A) -- (B);
		\draw[ -implies,double equal sign distance] (B) -- (C);
		\draw[ -implies,double equal sign distance] (B) -- (D);
		\draw[ -implies,double equal sign distance] (D) -- (E);
		\draw[ -implies,double equal sign distance] (C) -- (E);
		\end{tikzpicture}
	\end{center}
\end{proposition}

\begin{proof}
	The fact that the uniform versions imply the non-uniform ones follows by taking $V = \{e_G\}$. The implications showing that (uniform) strong TMP implies (uniform) TMP are also trivial. Let us show that if $G \curvearrowright X$ has the POTP then $G\curvearrowright X$ has the uniform strong TMP.
	
	Let $\varepsilon >0$. By the POTP, there exist $\delta >0$ and $S \Subset G$ such that every $(S,\delta)$ pseudo-orbit is $\frac{\varepsilon}{2}$-traced by some $z \in X$. Take $F\Subset G$ such that $F \supset S$ and $e_G \in F = F^{-1}$. Choosing $\delta \leq \frac{\varepsilon}{2}$ we can apply~\Cref{L-PO} to obtain $\eta \leq \delta$, such that for any $A\Subset G$, any $FA$-separated set $V\subset G$, any $x_v\in X$ for $v \in V$, and $y \in X$ satisfying $d_{(FA\setminus A)v}(x_v,y) \leq \eta$ for all $v \in V$, we have that $\{w_{\gamma}\}_{\gamma\in G}$ defined by
	\[ w_{\gamma} =  \begin{cases}
	\gamma x_v & \mbox{ if } \gamma \in Av \mbox{ for some }v\in V\\
	\gamma y & \mbox{ if }\gamma \in G\setminus AV
	\end{cases} \]
	is an $(S,\delta)$ pseudo-orbit. By the POTP, there exists $z \in X$ which $\frac{\varepsilon}{2}$-traces $\{w_{\gamma}\}_{\gamma\in G}$. Therefore, for every $v\in V$ and $g \in FAv$ we have that either $g \in Av$, and thus $d(gz,gx_v) = d(gz,w_g) \leq \frac{\varepsilon}{2}$, or $g \in (FA \setminus A)v$ and thus $d(gz,gx_v) \leq d(gz,gy)+d(gy,gx_v) \leq  d(gz,w_g)+\eta \leq \varepsilon$. For $g \in G\setminus AV$ we have $d(gz,gy)=d(gz,w_g) \leq \frac{\varepsilon}{2}$. This shows that $FA$ is a uniform $(\varepsilon,\eta)$-memory set for $A$ and hence that $G\curvearrowright X$ has the uniform strong TMP.\end{proof}

\begin{proposition}\label{P-amenable_implications_trivial}
	Let $G$ be an amenable group and $G \curvearrowright X$ an action. The following implications hold.
	\begin{center}
		\begin{tikzpicture}
		\node (A) at (0,0) {$\underset{\mbox{has uniform strong TMP}}{G \curvearrowright X}$};
		\node (B) at (5.5,0) {$\underset{\mbox{has uniform mean TMP}}{G \curvearrowright X}$};
		\node (C) at (11,0) {$\underset{\mbox{has uniform TMP}}{G \curvearrowright X}$};
		\node (AA) at (0,-1.5) {$G \curvearrowright X$ \mbox{has strong TMP}};
		\node (D) at (5.5,-1.5) {$G \curvearrowright X$ \mbox{has mean TMP}};
		\node (E) at (11,-1.5) {$G \curvearrowright X$ \mbox{has TMP}};
		\draw[ -implies,double equal sign distance] (A) -- (B);
		\draw[ -implies,double equal sign distance] (A) -- (AA);
		\draw[ -implies,double equal sign distance] (AA) -- (D);
		\draw[ -implies,double equal sign distance] (B) -- (C);
		\draw[ -implies,double equal sign distance] (B) -- (D);
		\draw[ -implies,double equal sign distance] (D) -- (E);
		\draw[ -implies,double equal sign distance] (C) -- (E);
		\end{tikzpicture}
	\end{center}
\end{proposition}

\begin{proof}
	As before, the uniform versions imply the non-uniform versions by choosing $V = \{e_G\}$. Since $G$ is amenable, there exists an increasing F\o lner sequence $\{F_n\}_{n \in \NN}$ for $G$ with union $G$. If $G\curvearrowright X$ has the (uniform) strong TMP, for any $\varepsilon>0$ there are $\delta>0$ and $K \Subset G$ containing the identity such that $KF_n$ is a (uniform) $(\varepsilon,\delta)$-memory set for $F_n$. As $\{F_n\}_{n\in \mathbb{N}}$ is F\o lner, we have that $|KF_n \setminus F_n| = o(|F_n|)$. Thus $G \curvearrowright X$ has the (uniform) mean TMP.
	
	If $G \curvearrowright X$ has the (uniform) mean TMP then for every $\varepsilon>0$ there are $0<\delta<\frac{\varepsilon}{2}$ and an increasing sequence $\{F_n\}_{n \in \NN}$ of finite subsets of $G$ with union $G$ for which there are finite (uniform) $(\frac{\varepsilon}{2},\delta)$-memory sets $\widetilde{F_n}$. It suffices to choose for every $A \Subset G$ a value of $n$ large enough such that $A \Subset F_n$ and thus $\widetilde{F_n}$ is a (uniform) $(\varepsilon,\delta)$-memory set for $A$.
\end{proof}

Now we will see that an expansive action has the TMP if and only if it has the uniform TMP.
Note that for any $F\Subset G$ and $\delta >0$, if $\{x_{t}\}_{t\in G}$ is an $(F,\delta )$ pseudo-orbit, then so is $\{x_{tg }\}_{t\in G}$ for any $g \in G$. The following lemma is only stated for $G={\mathbb{Z}}$ in the reference, but its proof works for all countable groups.

\begin{lemma}[Propositions 1 and 2 of~\cite{Achigar2019}] \label{L-close PO}
	 Let $G\curvearrowright X$ be an action and $c>0$. The following are equivalent:
	\begin{enumerate}
		\item $c$ is an expansivity constant for $G\curvearrowright X$.
		
		\item For any $\varepsilon >0$, there exist $F\Subset G$ and $\delta >0$
		such that for any two $(F,\delta )$ pseudo-orbits $\{x_{t}\}_{t\in G}$ and $%
		\{y_{t}\}_{t\in G}$, if $d(x_{t},y_{t})\leq c$ for all $t\in G$, then $%
		d(x_{t},y_{t})\leq \varepsilon $ for all $t\in G$.
		
		\item For any $\varepsilon >0$, there exists $W\Subset G$ such that for any $%
		x,y\in X$, if $d_W(x,y)\leq c$, then $d(x,y)<\varepsilon $.
	\end{enumerate}
\end{lemma}

\begin{theorem}\label{P-exp TMP to uniform TMP}
	Every expansive action $G \curvearrowright X$ with the TMP has the uniform TMP.
\end{theorem}

\begin{proof}
	Let $c>0$ be an expansivity constant for $G\curvearrowright X$. Let $0<\varepsilon \le c$. We have $F$ and $\delta $ in~\Cref{L-close PO}%
	.(2) for $\frac{\varepsilon}{2}$. Replacing $F$ by $F\cup F^{-1}\cup \{e_{G}\}$, we
	may assume that $e_{G}\in F=F^{-1}$. Then we have $\eta $ in~\Cref{L-PO}
	for $F$ and $\delta $.
	
	Since $G\curvearrowright X$ has the TMP, there is some $0<\tau <\min (\eta, \frac{\varepsilon}{2})
	$ such that for any $A\Subset G$ there is some $B_{A}\Subset G$ containing $A
	$ such that for any $x,y\in X$ satisfying $d_{B_{A}\setminus A}(x,y)\le \tau $, there is some $z\in X$ such that $d_{B_A}(z,x)\leq
	\frac{\varepsilon}{2}$ and $d_{G \setminus A}(z,y)\leq \frac{\varepsilon}{2}$. Take $W$ in~\Cref{L-close PO}.(3) for $\tau $. Replacing $W$ by $W\cup
	W^{-1}\cup \{e_{G}\}$, we may assume that $e_{G}\in W=W^{-1}$. Let $A\Subset G$. Then we have $B_{A}$ and $B_{WA}$ as above. Put $B=B_{A}\cup WB_{WA}\cup FA\Subset G$. We shall prove the following claim.\\
	
	\noindent \textbf{Claim:} For any $B$-separated $V\Subset G$ and any $%
	x_{v}\in X$ for $v\in V$ and $y\in X$ satisfying $d_{(B\setminus A)v}(x_{v},y)\le \tau $ for
	all $v\in V$, there is some $z\in X$
	such that $d_{Av}(z,x_{v})\leq \frac{\varepsilon}{2} $ for all $v\in V$
	and $d_{G\setminus AV}(z,y)\leq \frac{\varepsilon}{2} $.\\
	
	Assume the claim holds. As $\tau < \frac{\varepsilon}{2}$, we have that for any such finite set $V$, $d_{Bv}(z,x_{v})\leq \varepsilon$ for every $v \in V$ and $d_{G\setminus AV}(z,y)\leq \varepsilon$. This is exactly what we wanted to show.

	To prove the claim, we argue by induction on $|V|$. Consider first the case $%
	|V|=1$. Say, $V=\{v\}$. Let $x_{v},y\in X$ with $d_{(B\setminus A)v}(x_{v},y)\le \tau$. Then $d(s(vx_{v}),s(vy))\le \tau $ for all $%
	s\in B\setminus A$, in particular for all $s\in B_{A}\setminus A$. Thus by our choice of $B_{A}$ there is some $z\in X$ such that $d_{B_A}(z,vx_{v})\leq \frac{\varepsilon}{2}$ and $d_{G\setminus A}(z,vy)\leq
	\frac{\varepsilon}{2}$. Putting $z^{\prime} =  v^{-1}z$, we have
	$d_{B_Av}(z^{\prime },x_{v})\leq \frac{\varepsilon}{2}$ and $d_{G\setminus Av}(z^{\prime },y)\leq \frac{\varepsilon}{2}$. This
	proves the case $|V|=1$.
	
	Assume that the claim holds for $|V|=n$. Let $V\Subset G$ be $B$%
	-separated with $|V|=n+1$, and let $x_{v}\in X$ for $v\in V$ and $y\in X$
	with $d_{(B\setminus
		A)v}(x_{v},y)\le \tau $ for all $v\in V$. Take $v_{0}\in V$. Applying the inductive hypothesis to $V\setminus
	\{v_{0}\}$ we find some $u\in X$ such that $d_{Av}(u,x_{v})\leq \frac{\varepsilon}{2} $
	for all $v\in V\setminus \{v_{0}\}$, and $d_{G\setminus A(V\setminus \{v_{0}\})}(u,y)\leq
	\frac{\varepsilon}{2}$. For any $g\in B_{WA}\setminus WA$ and any $s\in W$, we have $sg\in WB_{WA}\setminus
	A\subset B\setminus A$, and hence
	\begin{equation*}
	d(sgv_{0}u,sgv_{0}x_{v_{0}})\leq
	d(sgv_{0}u,sgv_{0}y)+d(sgv_{0}x_{v_{0}},sgv_{0}y)\le \frac{\varepsilon}{2} +\tau \leq c.
	\end{equation*}%
	From our choice of $W$, we get $d_{B_{WA}\setminus WA}(v_{0}u,v_{0}x_{v_{0}})<\tau$. From the TMP and our choice of $B_{WA}$, we obtain there is $u'\in X$
	such that $d_{B_{WA}}(u',v_{0}x_{v_{0}})\leq \frac{\varepsilon}{2} $ and $d_{G\setminus WA}(u',v_{0}u)\leq \frac{\varepsilon}{2}$. Put $%
	z=v_{0}^{-1}u'\in X$. Then $d_{B_{WA}v_{0}}(z,x_{v_{0}})\leq \frac{\varepsilon}{2} $  and $d_{G\setminus WAv_{0}}(z,u)\leq \frac{\varepsilon}{2}$.
	For any $v\in V\setminus \{v_{0}\}$ and $s\in Av$, we have
	\begin{equation*}
	d(sz,sx_{v})\leq d(sz,su)+d(su,sx_{v})\leq \frac{\varepsilon}{2} +\frac{\varepsilon}{2} \leq c.
	\end{equation*}%
	For any $s\in Av_{0}\subset WAv_{0}$, we have $d(sz,sx_{v_{0}})\leq
	\frac{\varepsilon}{2} <c$. For any $s\in G\setminus (AV\cup WAv_{0})$, we have
	\begin{equation*}
	d(sz,sy)\leq d(sz,su)+d(su,sy)\leq \frac{\varepsilon}{2} +\frac{\varepsilon}{2} \leq c.
	\end{equation*}%
	For any $s\in WAv_{0}\setminus Av_{0}$, we have
	\begin{equation*}
	d(sz,sy)\leq d(sz,sx_{v_{0}})+d(sx_{v_{0}},sy)\le \frac{\varepsilon}{2} +\tau <c.
	\end{equation*}%
	Put $w_{\gamma }=\gamma x_{v}$ for all $v\in V$ and $\gamma \in Av$ and $%
	w_{\gamma }=\gamma y$ for all $\gamma \in G\setminus AV$. Then $d(\gamma
	z,w_{\gamma })\leq c$ for all $\gamma \in G$. Since $\tau <\eta $ and $%
	B\supset FA$, by our choice of $\eta $ we know that $\{w_{\gamma
	}\}_{\gamma \in G}$ is an $(F,\delta )$ pseudo-orbit. Then from our choice of
	$F$ and $\delta $ we conclude $d(\gamma z,w_{\gamma })\leq \frac{\varepsilon}{2} $ for
	all $\gamma \in G$. This finishes the induction step, and proves the claim.
\end{proof}

\begin{remark}
	The previous proof can be easily adapted to show that every expansive action with the strong TMP has the uniform strong TMP.
\end{remark}

\Cref{P-exp TMP to uniform TMP} does not hold if we do not assume expansivity. Indeed, Lind and Schmidt~\cite[Example 7.5]{LindSchmidt1999} constructed an algebraic action of $\ZZ^3$ with off-diagonal asymptotic pairs and zero topological entropy. As a consequence of Theorem~\ref{theorem_naive_IE_cites} and Proposition~\ref{P-A to IE} which we shall prove later on, this action does not have the uniform TMP, but it has the TMP by~\Cref{P_algebraic_have_wtmp}. Nonetheless, as we saw in~\Cref{prop_POTP_implies_sTMP}, whenever an action has the POTP it will have the uniform TMP even if it is not expansive.

\subsection{Examples}

Classic examples of actions with the POTP are subshifts of finite type (SFTs). Schmidt showed that expansive actions of polycyclic-by-finite groups on zero-dimensional compact metrizable groups by continuous automorphisms are SFTs \cite[Corollary 2.3, Theorems 3.8 and 4.2]{Schmidt1995}, thus have the POTP.
Other examples of group actions with the POTP are Axiom A
diffeomorphisms~\cite[Theorem 1.2.1]{AokiHiraide1994}, and expansive principal algebraic actions of countable groups~\cite[Theorem 1.5]{Meyerovitch2017}.

There are known examples of $\ZZ$-subshifts with the strong TMP but which are not SFTs, see~\cite[Proposition 3.6]{ChandgotiaHanMarcusMeyerovitchPavlov2014}. For any fixed countable group there are countably many SFTs, but in contrast, there are uncountably many $\ZZ^2$-subshifts $X$ such that $\ZZ^2 \curvearrowright X$ has the strong TMP~\cite{chandgotiamey}. Indeed, if $X \subset \Lambda^{\ZZ}$ is any $\ZZ$-subshift, then the $\ZZ^2$-subshift consisting of all $x \in \Lambda^{\ZZ^2}$ whose restriction to $\ZZ \times \{0\}$ is an element of $X$ and such that $\{0\}\times \ZZ$ acts trivially has the strong TMP but is not necessarily an SFT.

We shall see that subshifts with the strong TMP also arise naturally as supports of Markovian measures. A $G$-invariant Borel probability measure $\mu$ on a $G$-subshift $X$ is \define{Markovian} if there exists $F \Subset G$ containing the identity such that for every $A \Subset G$, $p \in \Sigma^A$ and $x \in \supp(\mu)$ we have that for every $B\Subset G$ which contains $AF$, \[ \mu( [p] \mid [x|_{B\setminus A}] ) = \mu( [p] \mid [x|_{AF \setminus A}] ), \]
where $\mu([p] \mid [q]) = \frac{\mu([p] \cap [q])}{\mu([q])}$ denotes the conditional probability of $[p]$ given $[q]$.

The following proposition is essentially a rephrasing of~\cite[Lemma 2.0.1]{chandgothese}.

\begin{proposition}\label{lemma_Markovian_measures_have_STMP}
	Let $X\subset \Lambda^G$ be a $G$-subshift and $\mu$ a Markovian measure on $X$. Then $G \curvearrowright \supp(\mu)$ has the strong TMP.
\end{proposition}

\begin{proof}
	As $\mu$ is Markovian there is $F\Subset G$ containing the identity such that for every $A \Subset G$ and $p \in \Lambda^A$ we have that for every $B\Subset G$ which contains $AF$ and $x \in \supp(\mu)$ we have $\mu( [p] \mid [x|_{B\setminus A}] ) = \mu( [p] \mid [x|_{AF \setminus A}] )$.
	Let us fix $A \Subset G$ and let $x,y \in \supp(\mu)$ such that $x|_{AF\setminus A} = y|_{AF\setminus A}$. We shall show that $z \in \Lambda^G$, defined by \[ z(g) = \begin{cases}
	x(g) & \mbox{ if } g \in A\\
	y(g) & \mbox{ otherwise }\\
	\end{cases}, \]
	 is in $\supp(\mu)$. Indeed, it suffices to show that for every large enough $B \Subset G$ we have $\mu([z|_{B}])>0$. For any $B \Subset G$ which contains $AF$ we have\begin{align*}
	\mu([z|_{B}])  = \mu([z|_{A}] \mid [z|_{B\setminus A}]) \cdot \mu([z|_{B\setminus A}]) = \mu([x|_{A}] \mid [y|_{B\setminus A}]) \cdot \mu([y|_{B\setminus A}]).
	\end{align*}
	As $y \in \supp(\mu)$, we have $\mu([y|_{B\setminus A}]) >0$. Furthermore, by our choice of $F$ we also get\begin{align*}
	\mu([x|_{A}] \mid [y|_{B\setminus A}]) = \mu([x|_{A}] \mid [y|_{AF\setminus A}])= \mu([x|_{A}] \mid [x|_{AF\setminus A}])= \frac{\mu([x|_{AF}])}{\mu([x|_A])}.
	\end{align*}
	Since $x\in \supp(\mu)$, we obtain $\mu([x|_{A}] \mid [y|_{B\setminus A}])>0$ and thus $\mu([z|_B])>0$.\end{proof}
	



\begin{example}[mean TMP but not strong TMP]
	Let $G$ be a locally finite group (for instance $G = \bigoplus_{i \in \NN}\ZZ/2\ZZ$) and $X$ be the subshift $\{0^{G},1^G\}\subset \{0, 1\}^G$. Recall that $G$ is infinite as assumed at the beginning of Section~\ref{section:preliminaries}. We claim that $G \curvearrowright X$ does not have the strong TMP. Indeed, assume $G \curvearrowright X$ has the strong TMP. Then for every $\varepsilon>0$ there are $\delta >0$ and $F \Subset G$ containing the identity, such that for every $A\Subset G$ we have that $FA$ is an $(\varepsilon,\delta)$-memory set for $A$. In particular taking $A = \langle F \rangle$ the subgroup generated by $F$, we get that $\langle F \rangle$ is an $(\varepsilon,\delta)$-memory set for $\langle F \rangle$ which is clearly false. On the other hand, $G \curvearrowright X$ does have the mean TMP. Indeed, for any sufficiently small $\delta>0$, any $A\Subset G$ and any $g \in G\setminus A$ we have that $A\cup\{g\}$ is an  $(\varepsilon,\delta)$-memory set for $A$.
\end{example}

\begin{example}[TMP but not mean TMP. Lemmas 4.2 and 4.3 of~\cite{Meyerovitch2017}] \label{example_TOM}
Let $G_i$ be a sequence of finite groups of order $a_i$ and consider their direct sum $G = \bigoplus_{i \in \NN} G_i$. Identify $g \in G_i$ with the tuple in $G$ for which the only non-identity coordinate is $g$ at position $i$. Consider the subshift $X \subset (\ZZ/2\ZZ)^G$ consisting of all $x\in (\ZZ/2\ZZ)^G$ such that for every $i \in \NN$ and $h \in G$ we have $\sum_{g \in G_i}{x(hg)} = 0$. This example has no off-diagonal asymptotic pairs and, if $a_i$ grows rapidly enough (so that $\prod_{i \geq 1}(1-a_i^{-1}) >0)$ then $G \curvearrowright X$ has positive topological entropy.  We shall later see that, by~\Cref{P_algebraic_have_wtmp} the action $G \curvearrowright X$ has the TMP, while as a consequence of Remark~\ref{R-exp mTMP IE to asym} it does not satisfy the mean TMP.
\end{example}

Another example with the TMP and not the strong TMP can be found in~\cite[Example 2.4]{Richifest}. This example has both positive topological entropy and off-diagonal asymptotic pairs.

\begin{example}[not TMP] \label{example_sunny_side_up}
	Consider the sunny-side up $G$-subshift consisting on all points $x \in \{0,1\}^G$ where $x(g)=1$ for at most one value of $g$, that is, \[X_{\leq 1} = \left\{ x \in \{0,1\}^G : |x^{-1}(1)|\leq 1   \right\}. \]
	We claim that $G\curvearrowright X_{\leq 1}$ does not have the TMP. Let $A = \{e_G\}$. For every $B\Subset G$ which contains $A$, let $h \in G\setminus B$ and $x,y \in X_{\leq 1}$ such that $x(e_G)=1$ and $y(h)=1$. Then $x|_{B \setminus A} = y|_{B \setminus A}$ but $z \in \{0,1\}^G$ such that $z|_{A}=x|_{A}$ and $z|_{G \setminus A} = y|_{G\setminus A}$ does not belong to $X_{\leq 1}$. By Corollary~\ref{C-subshift}
the action $G\curvearrowright X_{\leq 1}$ does not have the TMP.\end{example}

When $G = \ZZ$ the sunny-side up subshift $X_{\leq 1}$ is a topological factor of an SFT. Indeed, let $X \subset \{a,b\}^{\ZZ}$ be the subshift of all configurations which satisfy that if $x(n)=b$, then $x(n+1)=b$ for every $n \in \ZZ$. Since $X$ is an SFT, it satisfies the TMP by~\Cref{prop_POTP_implies_sTMP}. The map $\pi\colon X \to X_{\leq 1}$ given by \[\pi(x)(n) = \begin{cases}
1 & \mbox{ if } x(n-1)=a \mbox{ and } x(n)=b\\
0 & \mbox{ otherwise }
\end{cases} \mbox{ for every }n \in \ZZ, \]
is a factor map. It follows that even for $\ZZ$-actions the property of having the TMP does not pass to topological factors.

\section{Markovian properties of algebraic actions}\label{section:algebraico}

In this section we prove Theorem~\ref{T-expansive algebraic to STMP}, which establishes the strong TMP for finitely presented expansive algebraic actions of a large class of amenable groups.

\begin{proposition}\label{P_algebraic_have_wtmp}
	Every action $G \curvearrowright X$ of a group $G$ on a compact metrizable group $X$ by continuous automorphisms has the TMP.
\end{proposition}

\begin{proof}
    We can find a compatible metric $d$ on $X$ which is translation-invariant, that is, $d(zxw, zyw)=d(x, y)$ for all $x, y, z, w\in X$. Indeed, given any compatible metric $d'$ on $X$, we can define $d(x,y) = \sup_{z, w \in X}d'(zxw, zyw)$ which satisfies the property.
Let $\varepsilon>0$ be arbitrary and fix $\delta = \frac{\varepsilon}{2}$. Let $\mathcal{C} = \{C_1,C_2,\dots,C_n\}$ be a $\delta$-cover of $X$, that is, for every $C_i$, $\sup_{x,y \in C_i}d(x,y) \leq \delta$, and for $A,B \Subset G$ define $L_{A|B}$ as the set of all functions $\varphi\colon A \to \{1,\dots,n\}$ such that there exists $z \in X$ satisfying:
	\begin{enumerate}
		\item For every $g \in A$ we have $gz \in C_{\varphi(g)}$.
		\item We have $d_{B\setminus A}(z,e_{X})\leq \delta$, where $e_X$ denotes the identity element of $X$.

	\end{enumerate}
	Clearly if $B_1 \subset B_2$ then $L_{A|B_1} \supset L_{A|B_2}$, also, these sets are all finite. It follows that if we fix an enumeration $g_0,g_1,\dots$ of $G$ and set $B_n = \{g_0,\dots,g_n\}$ then the decreasing sequence
	\[   \dots \subset L_{A|B_n} \subset \dots \subset L_{A|B_1} \subset L_{A|B_0}\]
	stabilizes for some $N_{A} \in \NN$. Let us fix $B = A \cup B_{N_A}$.
	
	Suppose we have $x \in X$ such that $d_{B\setminus A}(x,e_{X})\leq \delta$. By definition any function $\varphi_x\colon A \to \{1,\dots,n\}$ such that $\varphi_x(g) = k$ implies $gx \in C_k$ is in $L_{A |B}$. By our choice of $N_A$, for each $m > N_A$ we may extract a point $z_m\in X$ for which $d_{B_m\setminus A}(z_m,e_{X})\leq \delta$  and $gz_m \in C_{\varphi_x(g)}$ for every $g \in A$. As $\mathcal{C}$ is a $\delta$-cover, we obtain that $d_A(x, z_m) \leq \delta$ for every $m > N_A$. By compactness of $X$, we may extract an accumulation point $z \in X$ which satisfies that $d_{G\setminus A}(z,e_{X})\leq \delta < \varepsilon$  and $d_A(z, x)\leq \delta < \varepsilon$. Then $d_{B\setminus A}(z, x)\leq d_{B\setminus A}(z,e_X)+d_{B\setminus A}(e_X, x) \leq 2\delta \leq \varepsilon$.
	
	Now let $x,y \in X$ be arbitrary points such that $d_{B \setminus A}(x,y)\leq \delta$. Then $x' = y^{-1}x$ satisfies that $d_{B\setminus A}(x',e_{X}) \leq \delta$. Consequently we may extract $z'$ with the aforementioned properties. Let $z = yz'$. As the metric is translation-invariant, by left-multiplying by $y^{-1}$ we have that $d_B(z, x) = d_B(z', x')\leq \varepsilon$  and $d_{G\setminus A}(z, y)=d_{G\setminus A}(z',e_{X}) \leq \varepsilon$.
\end{proof}

	We say $G\curvearrowright X$ is an \define{algebraic action} if $X$ is a compact metrizable abelian group and $G$ acts by continuous automorphisms.

 Schmidt~\cite[Theorems 3.8 and 4.2]{Schmidt1995} showed that every group shift of a polycyclic-by-finite group  has the POTP (i.e. it is an SFT). Nonetheless, Bhattacharya~\cite[Section 3]{Bhattacharya2019} showed that not every expansive algebraic action of a polycyclic group has the POTP (negatively answering~\cite[Question 3.11]{Meyerovitch2017}). We will now show that a larger class of algebraic actions, including Bhattacharya's example, always satisfy the strong TMP.

 Let us briefly recall the notion of group ring. Given a  group $G$ and a unital ring $R$, the \define{group ring} $RG$ consists of  all functions $f\colon G \to R$ of finite support. We shall write the elements of $RG$ as $f=\sum_{s\in G}f_s s$, where $f_s\in R$ is zero except for finitely many $s\in G$. The addition and multiplication of $RG$ are given by
 \begin{align}
 \sum_{s\in G}f_s s+\sum_{s\in G}g_s s &=\sum_{s\in G}(f_s+g_s)s, \nonumber \\
 (\sum_{s\in G}f_s s)(\sum_{s\in G}g_ss)&=\sum_{s\in G}(\sum_{t\in G}f_tg_{t^{-1}s})s. \label{E-convolution}
 \end{align}
We refer the reader to \cite{Passman1985} for general information about group rings. The ring $\CC G$ has also a $*$-operation given by
\begin{align} \label{E-*}
(\sum_{s\in G}f_s s)^*=\sum_{s\in G}\overline{f_s}s^{-1}.
\end{align}
Under these operations $\CC G$ becomes a $*$-algebra, that is, $(fg)^*=g^*f^*$, $(f+g)^*=f^*+g^*$, and $(\lambda f)^*=\bar{\lambda}f^*$ for all $f, g\in \CC G$ and $\lambda\in \CC$. For $g\in (\RR/\ZZ)^G$ and $f\in \ZZ G$, the convolution products $fg, gf\in (\RR/\ZZ)^G$ are also defined via \eqref{E-convolution}.

  Let us also denote for $p \in [1,+\infty)$ by $\ell^p(G) = \{ f\colon G \to \CC : \sum_{s \in G}|f_s|^p < \infty \}$ and $\ell^{\infty}(G) = \{ f \colon G \to \CC : \sup_{s \in G}|f_s| < \infty \}$ together with the norms $\|\cdot\|_p$ on $\ell^p(G)$ and $\|\cdot \|_\infty$ on $\ell^\infty(G)$. The convolution products $fg, gf\in \ell^p(G)$ are defined by \eqref{E-convolution} for all $1\le p\le \infty$ and $f\in \ell^1(G)$, $g\in \ell^p(G)$.
  The $*$-operation also extends to $\ell^p(G)$ for $p \in [1,+\infty]$ via \eqref{E-*}. Then $\ell^1(G)$ is also a $*$-algebra.

  For $m,n \in \NN$ and $a\in M_{m, n}(\CC G)$, denote by $\ker a$ the kernel of the bounded linear operator $M_{n\times 1}(\ell^2(G))\rightarrow M_{m\times 1}(\ell^2(G))$ sending $z$ to $az$, and by $P$ the orthogonal projection from $M_{n\times 1}(\ell^2(G))$ to $\ker a$. Then $P=(P_{jk})_{1\le j, k\le n}\in M_{n\times n}(\mathcal{B}(\ell^2(G)))$, where $\mathcal{B}(\ell^2(G))$ denotes the algebra of all bounded linear operators from $\ell^2(G)$ to itself. We have the canonical orthonormal basis $\{\delta_s\}_{s\in G}$ of $\ell^2(G)$, where $\delta_s$ is the unit vector in $\ell^2(G)$ taking value $1$ at $s$ and $0$ everywhere else. The \define{von Neumann dimension} of $\ker a$ is
  $$\dim_{\rm vN}\ker a:=\sum_{j=1}^n \left< P_{jj} \delta_{e_G}, \delta_{e_G}\right>.$$

  For a gentle introduction to group von Neumann algebras and a definition of the von Neumann dimension in a general context we refer the reader to Sections 1.1.1 to 1.1.3 of~\cite{Luck2002}. The following lemma is well known. For convenience of the reader, we give a proof here.

  \begin{lemma} \label{L-dim basic}
  Let $a\in M_{m, n}(\CC G)$. The following hold:
  \begin{enumerate}
  \item $\dim_{\rm vN}\ker a\in [0, n]$,
  \item $\dim_{\rm vN}\ker a=0$ if and only if $\ker a=\{0\}$,
  \item $\dim_{\rm vN}\ker a=n$ if and only if $a=0$.
  \end{enumerate}
  \end{lemma}
  \begin{proof} For each $z\in \ell^2(G)$ and $1\le j\le n$ denote by $z\otimes \delta^j$ the column vector in $M_{n\times 1}(\ell^2(G))$ which is equal to $z$ at the $j$-th row and $0$ everywhere else. Then
  $$\dim_{\rm vN}\ker a=\sum_{j=1}^n \left< P(\delta_{e_G}\otimes \delta^j), \delta_{e_G}\otimes \delta^j\right>.$$

  Since $P$ is an orthogonal projection, we have $P^2=P=P^*$. Thus
  \begin{align} \label{E-basic 1}
  \dim_{\rm vN}\ker a=\sum_{j=1}^n \left< P^2(\delta_{e_G}\otimes \delta^j), \delta_{e_G}\otimes \delta^j\right>=\sum_{j=1}^n \left< P(\delta_{e_G}\otimes \delta^j), P(\delta_{e_G}\otimes \delta^j)\right>\ge 0.
  \end{align}
  Denote by $I$ the identity map from $M_{n\times 1}(\ell^2(G))$ to itself. Then $I-P$ is the orthogonal projection from $M_{n\times 1}(\ell^2(G))$ to the orthogonal complement of $\ker a$, and hence $(I-P)^2=I-P=(I-P)^*$. Thus
  \begin{align} \label{E-basic 2}
  n-\dim_{\rm vN}\ker a&=\sum_{j=1}^n \left< I(\delta_{e_G}\otimes \delta^j), \delta_{e_G}\otimes \delta^j\right>-\sum_{j=1}^n \left< P(\delta_{e_G}\otimes \delta^j), \delta_{e_G}\otimes \delta^j\right>\\
 \nonumber &=\sum_{j=1}^n \left< (I-P)(\delta_{e_G}\otimes \delta^j), \delta_{e_G}\otimes \delta^j\right>\\
 \nonumber &=\sum_{j=1}^n \left< (I-P)^2(\delta_{e_G}\otimes \delta^j), \delta_{e_G}\otimes \delta^j\right>\\
  \nonumber &=\sum_{j=1}^n \left< (I-P)(\delta_{e_G}\otimes \delta^j), (I-P)\delta_{e_G}\otimes \delta^j\right>\ge 0.
    \end{align}
  This proves $\dim_{\rm vN}\ker a\in [0, n]$.

  For each $s\in G$, we have the unitary $\rho_{s^{-1}, n}$ on $M_{n\times 1}(\ell^2(G))$ given by $\rho_{s^{-1}, n}z=zs$.
	Since the bounded linear operator $T:M_{n\times }(\ell^2(G))\rightarrow M_{m\times 1}(\ell^2(G))$ sending $z$ to $az$ satisfies  $T\rho_{s^{-1}, n}=\rho_{s^{-1}, m}T$, we have $\rho_{s^{-1}, n}\ker a=\ker a$, and hence the projections $P$  and $I-P$ commute with $\rho_{s^{-1}, n}$.

  If $\ker a=\{0\}$, then $P=0$, and hence $\dim_{\rm vN}\ker a=0$. Conversely, assume that $\dim_{\rm vN}\ker a=0$. From \eqref{E-basic 1} we
  have $P(\delta_{e_G}\otimes \delta^j)=0$ for all $1\le j\le n$, and hence
  $$P(\delta_s\otimes \delta^j)=P\rho_{s^{-1}, n}(\delta_{e_G}\otimes \delta^j)=\rho_{s^{-1}, n}P(\delta_{e_G}\otimes \delta^j)=0$$
  for all $s\in G$ and $1\le j\le n$. Since $\delta_s\otimes \delta^j$ for all $s\in G$ and $1\le j\le n$ is an orthonormal basis of $M_{n\times 1}(\ell^2(G))$, we conclude that $P=0$. Therefore $\ker a=\{0\}$.

  If $a=0$, then $P=I$, and hence $\dim_{\rm vN}\ker a=n$. Conversely, assume that $\dim_{\rm vN}\ker a=n$. From \eqref{E-basic 2} we
  have $(I-P)(\delta_{e_G}\otimes \delta^j)=0$ for all $1\le j\le n$, and hence
  $$(I-P)(\delta_s\otimes \delta^j)=(I-P)\rho_{s^{-1}, n}(\delta_{e_G}\otimes \delta^j)=\rho_{s^{-1}, n}(I-P)(\delta_{e_G}\otimes \delta^j)=0$$
  for all $s\in G$ and $1\le j\le n$. Since $\delta_s\otimes \delta^j$ for all $s\in G$ and $1\le j\le n$ is an orthonormal basis of $M_{n\times 1}(\ell^2(G))$, we conclude that $I-P=0$. Then $\ker a=M_{n\times 1}(\ell^2(G))$, and  hence $a=0$.
  \end{proof}

  The strong Atiyah conjecture asserts that $\dim_{\rm vN}\ker a$ is in the subgroup of $\QQ$ generated by $1/|H|$ for $H$ ranging over all finite subgroups of $G$.
  We refer the reader to \cite[Chapter 10]{Luck2002} for information about the strong Atiyah conjecture and related conjectures.

For each locally compact abelian group $Y$, we denote by $\widehat{Y}$ its Pontryagin dual. It consists of all continuous group homomorphisms $Y\rightarrow \RR/\ZZ$, and becomes a locally compact abelian group under pointwise addition and the topology of uniform convergence on compact subsets.
Then $Y$ is compact metrizable if and only if $\widehat{Y}$ is countable discrete.

For a compact metrizable abelian group $X$, there is a natural one-to-one correspondence between algebraic actions of $G$ on $X$ and actions of $G$ on $\widehat{X}$ by automorphisms. There is also a natural one-to-one correspondence between the latter and the left $\ZZ G$-module structure on $\widehat{X}$.  Thus, up to isomorphism, there is a natural one-to-one correspondence between algebraic actions of $G$ and countable left $\ZZ G$-modules.
We say an algebraic action $G\curvearrowright X$ is \define{finitely generated} (\define{finitely presented} resp.) if $\widehat{X}$ is a finitely generated (finitely presented resp.) left $\ZZ G$-module. Every expansive algebraic action of $G$ is finitely generated~\cite[Proposition 2.2 and Corollary 2.16]{Schmidt1995}.

Using the $*$-operation it is easy to see  that $\ZZ G$ is left Noetherian if and only if it is right Noetherian.
Also note that if a unital ring $R$ is left Noetherian, then every finitely generated left $R$-module is finitely presented \cite[Proposition 4.29]{Lam1999}.

\begin{theorem}\label{T-expansive algebraic to STMP}
	Let $G$ be an amenable group. The following results hold:
	\begin{enumerate}
		\item If $\ZZ G$ is left Noetherian, then every expansive algebraic action of $G$ has the strong TMP.
		\item If $G$ satisfies the strong Atiyah conjecture and there is an upper bound on the orders of finite subgroups of $G$, then every finitely presented expansive algebraic action of $G$ has the strong TMP.
	\end{enumerate}
\end{theorem}

If $G$ satisfies the strong Atiyah conjecture and there is an upper bound on the orders of finite subgroups of $G$, then $\{\dim_{\rm vN}\ker a: a\in \ZZ G\}$  is a finite subset of $[0, 1]$ and thus every point of this set is isolated. Therefore
Theorem~\ref{T-expansive algebraic to STMP} follows from Lemmas~\ref{L-kernel} and \ref{L-divisor to sTMP} below.

For each nonempty $F\Subset G$, denote by $\CC F$ (resp. $\ZZ F$) the set of $f\in \CC G$ ($f\in \ZZ G$ resp.)  with support in $F$.
 We shall need the following result of Elek in~\cite{Elek2003} on the analytic zero divisor conjecture. The result is proven for a finitely generated amenable group and $m=n=1$, but the arguments work for any countable amenable group and any $m, n\in \NN$.

 \begin{lemma} \label{L-Elek}
 	Let $G$ be an amenable group, and $a\in M_{m\times n}(\CC G)$ for some $m, n\in \NN$. Then the following hold:
 \begin{enumerate}
\item For any $F\Subset G$, one has $\dim_{\rm vN}\ker a\ge \frac{\dim_{\CC} (\ker a\cap M_{n\times 1}(\CC F))}{|F|}$,
\item If $az=0$ for some nonzero $z\in M_{n\times 1}(\ell^2(G))$, then $ab=0$ for some nonzero $b\in M_{n\times 1}(\CC G)$.
 \end{enumerate}
 \end{lemma}

For amenable groups, in view of Lemmas~\ref{L-dim basic} and \ref{L-Elek}, the condition (2) of the following lemma implies that $\dim_{\rm vN}\ker a$ for $a\in \ZZ G$ is equal to $1$ or $0$ depending on $a=0$ or not,  which in turn implies the condition (3). Nevertheless we give a separate proof for the case (2) since the proof in this case is easier.

\begin{lemma} \label{L-kernel}
	Let $G$ be an amenable group. Assume that at least one of the following conditions holds:
\begin{enumerate}
\item $\ZZ G$ is right Noetherian,
\item  $\ZZ G$ is a domain,
\item $1$ is an isolated point in $\{\dim_{\rm vN}\ker a : a\in \ZZ G \}$.
\end{enumerate}
Let $f\in M_n(\ZZ G)$ be invertible in $M_n(\ell^1(G))$, and let $g\in M_{n\times 1}(\ZZ G)$. Then there is some $\varepsilon>0$ such that for any $x\in M_{1\times n}(\ZZ G)$, if
	$xf^{-1}g\not\in \ZZ G$, then $\|xf^{-1}g-y\|_\infty\ge \varepsilon$ for all $y\in \ZZ G$.
\end{lemma}
\begin{proof} Denote by $K$ the union of the supports of $f$ and $g$, which is a nonempty finite subset of $G$.
	
	For each nonempty  $F\Subset G$, denote by $V(F, \ZZ)$ ($V(F, \CC)$ resp.) the set of $(h, w)\in M_{n\times 1}(\ZZ F)\times \ZZ F$
	($(h, w)\in M_{n\times 1}(\CC F)\times \CC F$ resp.)
	satisfying $gw=fh$.
	Note that $gw=fh$ is a finite system of linear equations with coefficients in $\ZZ$. Thus $V(F, \CC)$ is the $\CC$-linear  span of $V(F, \ZZ)$.
	Since $f$ is invertible in $M_n(\ell^1(G))$, for any $(h, w)\in V(F, \CC)$, we have $h=f^{-1}gw$, and hence $h$ is determined by $w$.
	Denote by $W(F, \CC)$ the image of $V(F, \CC)$ under the projection $M_{n\times 1}(\CC F)\times \CC F\rightarrow \CC F$ sending $(h, w)$ to $w$, and by $W(F, \ZZ)$ the image of $V(F, \ZZ)$ under the same map. Then $W(F, \CC)$ is the $\CC$-linear span of $W(F, \ZZ)$.
	
	Consider the map $\varphi_F\colon  M_{n\times 1}(\CC F)\times \CC F\rightarrow M_{n\times 1}(\CC KF)$ sending $(h, w)$ to $fh-gw$. Then $V(F, \CC)$ is the kernel of $\varphi_F$. Thus
	\begin{align*}
	\dim_\CC W(F, \CC)&=\dim_\CC V(F, \CC)\\
	&=\dim_\CC \ker(\varphi_F)\\
	&\ge \dim_\CC(M_{n\times 1}(\CC F)\times \CC F)-\dim_\CC M_{n\times 1}(\CC KF)\\
	&= (n+1)|F|-n|KF|.
	\end{align*}

	As $G$ is amenable, there exists a left F{\o}lner sequence $\{F_k\}_{k\in \NN}$ such that $F_k\subseteq F_{k+1}$ for all $k\in \NN$ and $\bigcup_{k\in \NN}F_k=G$. Then
	\[\liminf_{k\to \infty}\frac{\dim_\CC W(F_k, \CC)}{|F_k|}\ge \lim_{k\to \infty}\frac{(n+1)|F_k|-n|KF_k|}{|F_k|}=1.\]
	
Denote by $W(\ZZ)$ (resp. $W(\CC)$) the union of $W(F, \ZZ)$ ($W(F, \CC)$ resp.) over all nonempty $F\Subset G$. Then $W(\ZZ)$ contains nonzero elements. Let $\mathcal{W}=\{w_1,\dots, w_m\}$ be a nonempty finite set of nonzero elements in $W(\ZZ)$ which we shall determine later.
Take $0<\varepsilon<\min_{1\le j\le m}\frac{1}{\|w_j\|_1}$. Let $x\in M_{1\times n}(\ZZ G)$ and $y\in \ZZ G$ with $\|xf^{-1}g-y\|_\infty<\varepsilon$. Then
	$$\|xf^{-1}gw_j-yw_j\|_\infty\le \|xf^{-1}g-y\|_\infty \|w_j\|_1<\varepsilon \|w_j\|_1<1.$$
	Since $w_j\in W(F_k, \ZZ)$, we have $f^{-1}gw_j\in M_{n\times 1}(\ZZ F_k)$, and hence $xf^{-1}gw_j-yw_j\in \ZZ G$. Therefore $xf^{-1}gw_j-yw_j=0$.
Put $b:=xf^{-1}g-y\in \ell^1(G)\subseteq \ell^2(G)$.
Then $b(w_1, \dots, w_m)=0$, and hence $\left[\begin{matrix} w_1^* \\ \vdots \\ w_m^* \end{matrix}\right]b^*=0$. Assume $b\neq 0$. Then by Lemma~\ref{L-Elek} there is some nonzero $c\in \CC G$ such that $\left[\begin{matrix} w_1^* \\ \vdots \\ w_m^* \end{matrix}\right]c^*=0$, and hence $c w_j=0$ for all $1\le j\le m$. Denote by $K'$ the support of $c$, which is a nonempty finite subset of $G$. Note that $z(w_1, \dots, w_m)=0$ for $z\in \CC K'$ is a finite system of linear equations with coefficients in $\ZZ$ and has a nonzero solution $c$. Thus it has a nonzero solution $a\in \ZZ K'$.

	Now consider the case $\ZZ G$ is a domain.  Since $\ZZ G$ is a domain, we have $aw\neq 0$ for all nonzero $w\in \ZZ G$, which is a contradiction.
Thus $xf^{-1}g=y\in \ZZ G$.
	
	Next consider the case $\ZZ G$ is right Noetherian.
Note that $W(\ZZ)$ is a right ideal of $\ZZ G$, and hence is finitely generated. We may take $\mathcal{W}$ to generate $W(\ZZ)$.
Since $aw=0$ for every $w\in \mathcal{W}$, we have $aw=0$ for every $w\in W(\ZZ)$, and hence $aw=0$ for all $w\in W(\CC)$. By Lemma~\ref{L-Elek} we have
$\dim_{\rm vN}\ker a\ge \sup_{k\in \NN}\frac{\dim_\CC W(F_k, \CC)}{|F_k|}\ge 1$. From Lemma~\ref{L-dim basic} we get $a=0$, which is a contradiction. Thus $xf^{-1}g=y\in \ZZ G$.

Finally consider the case where $1$ is an isolated point in $V=\{\dim_{\rm vN}\ker z : z\in \ZZ G \}$. Take $0<\delta<1$ close to $1$ such that $V\cap [\delta, 1]=\{1\}$. Take $k\in \NN$ such that $ \frac{\dim_\CC W(F_k, \CC)}{|F_k|}\ge \delta$. We may take $\mathcal{W}$ to be a basis of $W(F_k, \CC)$ contained in $W(F_k, \ZZ)$. Then $aw=0$ for all $w\in W(F_k, \CC)$. By Lemma~\ref{L-Elek} we have $\dim_{\rm vN}\ker a\ge \frac{\dim_\CC W(F_k, \CC)}{|F_k|}\ge \delta$. Then $\dim_{\rm vN}\ker a=1$. From Lemma~\ref{L-dim basic} we get $a=0$, which is a contradiction. Thus $xf^{-1}g=y\in \ZZ G$.
	\end{proof}

For each $n\in \NN$, we write $M_{1\times n}(\ZZ G)$ and $M_{1\times n}((\RR/\ZZ)^G)$ as $(\ZZ G)^n$ and $((\RR/\ZZ)^G)^n$ respectively.
For any finitely generated algebraic action $G\curvearrowright X$, if we write $\widehat{X}$ as $(\ZZ G)^n/J$ for some $n\in \NN$ and some left $\ZZ G$-submodule $J$ of $(\ZZ G)^n$, then we may identify $X$ with
$$\{x\in ((\RR/\ZZ)^G)^n : xg^*=0_{(\RR/\ZZ)^G} \mbox{ for all } g\in J\}$$
with the $G$-action on $X$ being the restriction of the left shift action of $G$ on $((\RR/\ZZ)^G)^n=((\RR/\ZZ)^n)^G$ to $X$ \cite[page 312]{KerrLiBook2016}.

\begin{lemma} \label{L-divisor to sTMP}
	Assume that for any $n\in \NN$, any $f\in M_n(\ZZ G)$ which is invertible in $M_n(\ell^1(G))$ and any $g\in M_{n\times 1}(\ZZ G)$, there is some $\varepsilon>0$ such that for any $x\in (\ZZ G)^n$, if
	$xf^{-1}g\not\in \ZZ G$, then $\|xf^{-1}g-y\|_\infty\ge \varepsilon$ for all $y\in \ZZ G$. Then every finitely presented expansive algebraic action of $G$ has the strong TMP.
\end{lemma}
\begin{proof} Let $G\curvearrowright X$ be a finitely presented expansive algebraic action of $G$. Then we can write $\widehat{X}$ as $(\ZZ G)^n/(\ZZ G)^kg$ for some  $n, k\in \NN$ and $g\in M_{k, n}(\ZZ G)$. Since $G\curvearrowright X$ is expansive, by~\cite[Theorem 3.1]{ChungLi2015} we can find an $f\in M_n(\ZZ G)$ such that $f$ is invertible in $M_n(\ell^1(G))$ and the rows of $f$ are all contained in $(\ZZ G)^kg$.
	Write the rows of $g$ as $g_1, \dots, g_k$. Then we may identify $X$ with
	$$\{x\in ((\RR/\ZZ)^G)^n : xg_j^*=0_{(\RR/\ZZ)^G} \mbox{ for all } 1\le j\le k\}.$$
	In particular, for each $x\in X$, we have $xf^*=(0_{(\RR/\ZZ)^G}, \dots, 0_{(\RR/\ZZ)^G})$.

	Applying our assumption to $f^*$ and $g_j^*$ we find an $\eta>0$ such that for any $x\in (\ZZ G)^n$ and $1\le j\le k$, if
	$x(f^*)^{-1}g_j^*\not\in \ZZ G$, then $\|x(f^*)^{-1}g_j^*-y\|_\infty\ge \eta$ for all $y\in \ZZ G$.
	
	Denote by $P$ the natural projection $(\RR^n)^G\rightarrow ((\RR/\ZZ)^G)^n$ modulo $\ZZ$. For each bounded function $u\colon G\rightarrow \RR^n$, we put
	$\|u\|_\infty=\sup_{s\in G}\|u_s\|_\infty$.
	Set $\|f\|_1=\sum_{i, j=1}^n\|f_{i, j}\|_1$.
	
	Consider the metric $\rho$ on $(\RR/\ZZ)^n$ given by
	$$\rho(a+\ZZ^n, b+\ZZ^n)=\min_{m\in \ZZ^n}\|a-b-m\|_\infty$$
	for all $a, b\in \RR^n$. Then we may think of $\rho$ as a continuous pseudometric on $((\RR/\ZZ)^n)^G$ via setting
	$$\rho(x, y)=\rho(x_{e_G}, y_{e_G})$$
	for all $x, y\in ((\RR/\ZZ)^n)^G$. This is a dynamically generating continuous pseudometric on $X$, thus~\Cref{P-generating} applies.
	
	Let $\varepsilon>0$.
	Take $0<\tau<\frac{\min(\varepsilon, \eta)}{\|f\|_1}$.
	Take a finite subset  $K_1$ of $G$ containing the support of $f^*$ and $\{e_G\}$.
	Take a large nonempty finite subset $K_2$ of $G$ containing $e_G$ such that
	\begin{align*}
	\sum_{s\in G\setminus K_2}\sum_{i=1}^n|((f^*)^{-1}g_j^*)_{i, s}|<\tau \mbox{ for all } 1\le j\le k, \mbox{ and }
	\sum_{s\in G\setminus K_2}\sum_{i, j=1}^n|((f^*)^{-1})_{i,j, s}|<\tau.
	\end{align*}
	Take $0<\delta<\min(\frac{1}{\|f\|_1}, \frac{\varepsilon}{2})$. Put $F=K_1K_2K_2^{-1}K_1^{-1}$, which is a  finite subset of $G$ containing $e_G$.
	
	Let $A\Subset G$ and $x, y\in X$ with $\rho_{FA\setminus A}(x, y)\le \delta$.
	Put $z=x-y\in X$. There is a unique $\tilde{z}\in ([-1/2, 1/2)^n)^G$ satisfying $P(\tilde{z})=z$.
	Then
	\begin{align*}
	\max_{s\in A^{-1}K_1K_2K_2^{-1}K_1^{-1}\setminus A^{-1}}\|\tilde{z}_s\|_\infty &=\max_{s\in K_1K_2K_2^{-1}K_1^{-1}A\setminus A}\|\tilde{z}_{s^{-1}}\|_\infty\\
	&=\max_{s\in FA\setminus A}\|\tilde{z}_{s^{-1}}\|_\infty\\
	&=\max_{s\in FA\setminus A}\rho(sz, 0_X)\\
	&=\max_{s\in FA\setminus A}\rho(sx, sy)\\
	&=\rho_{FA\setminus A}(x, y)\le \delta,
	\end{align*}
	and hence
	$$ \max_{s\in A^{-1}K_1K_2K_2^{-1}\setminus A^{-1}K_1}\|(\tilde{z}f^*)_s\|_\infty\le (\max_{s\in A^{-1}K_1K_2K_2^{-1}K_1^{-1}\setminus A^{-1}}\|\tilde{z}_s\|_\infty)\|f^*\|_1\le \delta \|f\|_1<1.$$
	Since $zf^*=(0_{(\RR/\ZZ)^G}, \dots, 0_{(\RR/\ZZ)^G})$, we have $\tilde{z}f^*\in (\ZZ^n)^G$. Therefore
	$(\tilde{z}f^*)_s=0_{\RR^n}$ for all $s\in A^{-1}K_1K_2K_2^{-1}\setminus A^{-1}K_1$. We also have
	$$ \|\tilde{z}f^*\|_\infty\le \|\tilde{z}\|_\infty\|f^*\|_1\le\frac{\|f\|_1}{2}.$$
	
	Define $u'\in (\ZZ G)^n$ by $u'_s=(\tilde{z}f^*)_s$ for all $s\in A^{-1}K_1$ and $u'_s=0_{\ZZ^n}$ for all $s\in G\setminus A^{-1}K_1$.
	Also put $v'=\tilde{z}f^*-u'$. Then the supports of $u'$ and $v'$ are contained in $A^{-1}K_1$ and $G\setminus A^{-1}K_1K_2K_2^{-1}$ respectively. Note that
	$$ \max(\|u'\|_\infty, \|v'\|_\infty)=\|\tilde{z}f^*\|_\infty\le\frac{\|f\|_1}{2}.$$
	
	Let $1\le j\le k$.
	Then
	$$\sup_{s\in G\setminus A^{-1}K_1K_2}|(u'(f^*)^{-1}g_j^*)_s|\le \|u'\|_\infty\sum_{s\in G\setminus K_2}\sum_{i=1}^n|((f^*)^{-1}g_j^*)_{i, s}|\le \frac{\tau \|f\|_1}{2}<\eta,$$
	and
	$$\max_{s\in  A^{-1}K_1K_2}|(v'(f^*)^{-1}g_j^*)_s|\le \|v'\|_\infty\sum_{s\in G\setminus K_2}\sum_{i=1}^n|((f^*)^{-1}g_j^*)_{i, s}|\le \frac{\tau \|f\|_1}{2}<\eta.$$
	Since $z\in X$, we have $zg_j^*=0_{(\RR/\ZZ)^G}$, and hence $u'(f^*)^{-1}g_j^*+v'(f^*)^{-1}g_j^*=\tilde{z}g_j^*\in \ZZ^G$.
	It follows that $\|u'(f^*)^{-1}g_j^*-u_j\|_\infty<\eta$ for some $u_j\in \ZZ G$, and hence $u'(f^*)^{-1}g_j^*\in \ZZ G$ by our choice of $\eta$.
	
	Put $u=P(u'(f^*)^{-1})\in ((\RR/\ZZ)^G)^n$ and $v=P(v'(f^*)^{-1})\in ((\RR/\ZZ)^G)^n$. Note that $u+v=z=x-y$. For each $1\le j\le k$, since $u'(f^*)^{-1}g_j^*\in \ZZ G$, we have $ug_j^*=0_{(\RR/\ZZ)^G}$.
	Thus $u\in X$. We have
	\begin{align*}
	\sup_{s\in G\setminus A^{-1}K_1K_2}\rho(u_s, 0_{(\RR/\ZZ)^n})&\le \sup_{s\in G\setminus A^{-1}K_1K_2}\|(u'(f^*)^{-1})_s\|_\infty\\
	&\le \|u'\|_\infty\sum_{s\in G\setminus K_2}\sum_{i, j=1}^n|((f^*)^{-1})_{i,j, s}|\le \frac{\tau \|f\|_1}{2}<\frac{\varepsilon}{2},
	\end{align*}
	and similarly
	$$ \max_{s\in A^{-1}K_1K_2}\rho(v_s, 0_{(\RR/\ZZ)^n})\le  \max_{s\in A^{-1}K_1K_2}\|(v'(f^*)^{-1})_s\|_\infty\le \|v'\|_\infty\sum_{s\in G\setminus K_2}\sum_{i, j=1}^n|((f^*)^{-1})_{i,j, s}|<\frac{\varepsilon}{2}.$$
	Then
	$$ \rho_{G\setminus K_2^{-1}K_1^{-1}A}(u+y, y)=\sup_{s\in G\setminus A^{-1}K_1K_2}\rho((u+y)_s, y_s)=\sup_{s\in G\setminus A^{-1}K_1K_2}\rho(u_s, 0_{(\RR/\ZZ)^n})<\frac{\varepsilon}{2},$$
	and
	\begin{align*}
	\rho_{K_2^{-1}K_1^{-1}A}(u+y, x)&=
	\max_{s\in  A^{-1}K_1K_2}\rho((u+y)_s, x_s)\\
	&= \max_{s\in  A^{-1}K_1K_2}\rho((x-v)_s, x_s)
	=  \max_{s\in  A^{-1}K_1K_2}\rho(v_s, 0_{(\RR/\ZZ)^n})<\frac{\varepsilon}{2}.
	\end{align*}
	Now we have
	\begin{align*}
	\rho_{K_2^{-1}K_1^{-1}A\setminus A}(u+y, y)&\le \rho_{K_2^{-1}K_1^{-1}A\setminus A}(u+y, x)+\rho_{K_2^{-1}K_1^{-1}A\setminus A}(x, y)\\
	&\le \rho_{K_2^{-1}K_1^{-1}A}(u+y, x)+\rho_{FA\setminus A}(x, y)
	<\frac{\varepsilon}{2}+\delta<\varepsilon,
	\end{align*}
	and
	\begin{align*}
	\rho_{FA\setminus K_2^{-1}K_1^{-1}A}(u+y, x)&\le \rho_{FA\setminus K_2^{-1}K_1^{-1}A}(u+y, y)+\rho_{FA\setminus K_2^{-1}K_1^{-1}A}(y, x)\\
	&\le \rho_{G\setminus K_2^{-1}K_1^{-1}A}(u+y, y)+\rho_{FA\setminus A}(y, x)
	<\frac{\varepsilon}{2}+\delta<\varepsilon.
	\end{align*}
	Finally,
	$$ \rho_{G\setminus A}(u+y, y)=\max(\rho_{G\setminus K_2^{-1}K_1^{-1}A}(u+y, y), \rho_{K_2^{-1}K_1^{-1}A\setminus A}(u+y, y))<\varepsilon,$$
	and
	$$\rho_{FA}(u+y, x)=\max(\rho_{ K_2^{-1}K_1^{-1}A}(u+y, x), \rho_{FA\setminus K_2^{-1}K_1^{-1}A}(u+y, x))<\varepsilon.$$
	From~\Cref{P-generating} we conclude that $G\curvearrowright X$ has the strong TMP.
\end{proof}

\section{Markovian properties of minimal actions}\label{section:minimial}

In this section we prove Theorem~\ref{T-minimal exp tmp iff noasym}, which characterizes the TMP for minimal expansive actions.
We say $G \curvearrowright X$ is \define{minimal} if every closed $G$-invariant subset of $X$ is either equal to $X$ or empty.

\begin{theorem} \label{T-minimal exp tmp iff noasym}
	Let $G \curvearrowright X$ be a minimal expansive action. Then $G \curvearrowright X$ has the TMP if and only if it has no off-diagonal asymptotic pairs.
\end{theorem}

Theorem~\ref{T-minimal exp tmp iff noasym} follows from Theorem~\ref{P-exp TMP to uniform TMP} and Propositions~\ref{P-no asym to TMP} and~\ref{P-minimal uniform TMP no asym} below.

The following result is \cite[Lemma 6.2]{ChungLi2015}, and also follows from the part (3) of Lemma~\ref{L-close PO}.

\begin{lemma} \label{lema_expansivo}
Let $G\curvearrowright X$ be an expansive action with an expansivity
	constant $c>0$. Then $\mathtt{A}^{c}_2(X,G) = \mathtt{A}_2(X,G).$
\end{lemma}

\begin{proposition} \label{P-no asym to TMP}
	If $G\curvearrowright X$ is expansive and has no off-diagonal asymptotic pairs, then $G \curvearrowright X$ has the TMP.
\end{proposition}

\begin{proof}
	Let $c>0$ be an expansivity constant for $G\curvearrowright X$. If $G \curvearrowright X$ does not have the TMP then there exists $\varepsilon >0$ such that for every $\delta>0$ there is $A \Subset G$ such that for every $B \Subset G$ which contains $A$ there exist $x,y \in X$ such that $d_{B \setminus A}(x,y) \leq \delta$ but for every $z \in X$ we have that either $d_B(x,z)>\varepsilon$ or $d_{G \setminus A}(y,z)>\varepsilon$. Choose $\delta < \min(c,\varepsilon)$ and consider an increasing sequence $\{B_n\}_{n \in \NN}$ of finite subsets of $G$ such that $\bigcup_{n \in \NN}B_n = G$ and $B_n \supset A$ and let $(x_n,y_n)$ be a pair for which $d_{B_n \setminus A}(x_n,y_n) \leq \delta$ but for every $z \in X$ we have that either $d_{B_n}(x_n,z)>\varepsilon$ or $d_{G \setminus A}(y_n,z)>\varepsilon$. By compactness of $X\times X$, we may extract an accumulation point $(\bar{x},\bar{y})$ of $\{(x_n,y_n)\}_{n \in \NN}$.
	
	From the choice of $B_n$ and $(x_n,y_n)$ it follows that for every $g \notin A$ we get that $d(g\bar{x},g\bar{y})\leq \delta$ and hence by the choice $\delta < c$ and~\Cref{lema_expansivo} we have that $(\bar{x},\bar{y})$ is an asymptotic pair. If $\bar{x}=\bar{y}$ then for every $\varepsilon' > 0$ there would be $n \in \NN$ such that $d_A(x_n,y_n) < \varepsilon'$. Taking $\varepsilon' < \delta$ yields a contradiction because $z = y_n$ would satisfy that $d_{B_n}(x_n,z)\leq\varepsilon$ and $d_{G \setminus A}(y_n,z)\leq\varepsilon$. Therefore $(\bar{x},\bar{y})$ is an off-diagonal asymptotic pair.
\end{proof}

Salo~\cite{salo} communicated to us a proof that minimal actions of finitely generated groups on subshifts do not have exchangeable patterns. We use some of his ideas to prove the following generalization of his result.

\begin{proposition}\label{P-minimal uniform TMP no asym}
Assume that $G\curvearrowright X$ has only finitely many minimal closed $G$-invariant
	subsets and has the uniform TMP. Then $X$ has no off-diagonal asymptotic pairs $(x_{1},x_{2})$ with $\overline{Gx_{1}}$ minimal.
\end{proposition}

\begin{proof}
	List the minimal closed $G$-invariant subsets of $X$ as $Y_{1},\dots ,Y_{N}$.
    Take a point $\omega _{j}\in Y_{j}$ for each $1\leq j\leq N$. Put $\Omega
	=\{\omega _{1},\dots ,\omega _{N}\}$. Assume that $X$ has an off-diagonal asymptotic pair $(x_{1},x_{2})$ such that $%
	\overline{Gx_{1}}$ is minimal. Take $0<\varepsilon <\frac{1}{5}d(x_{1},x_{2})$.
	
	Since $G\curvearrowright X$ has the uniform TMP, there is some $\tau >0$ such
	that for any $A\Subset G$ there is some $B\Subset G$ containing $A$ such
	that for any $B$-separated $V\Subset G$ and any $x_v\in X$ for $v\in V$ and $y\in X$ satisfying $d_{(B\setminus A)v}(x_{v},y)\le \tau $ for all $v\in V$, there is some $z\in X$ such that $d_{Bv}(z,x_{v})\leq
	\frac{\varepsilon}{2}$ for all $v\in V$ and $d_{G\setminus AV}(z,y)\leq
	\frac{\varepsilon}{2}$.
	
	Since $(x_{1},x_{2})$ is asymptotic, there is some $A\Subset G$ containing $e_{G}$ such that $d_{G \setminus A}(x_{1},x_{2})<\frac{\tau}{2}$. Then we have $B$ as above for $A$. Take $0<\theta <\frac{\varepsilon}{2}$ such that for any
	$x,y\in X$ with $d(x,y)\le \theta$, one has $d_{B \setminus A}(x,y)<\frac{\tau}{2}$. Since $\overline{Gx_{1}}$ is minimal, there is some $
	K_{1}^{\prime }\Subset G$ containing $e_{G}$ such that for any $y\in \overline{Gx_{1}}$ one has $\min_{s\in K_{1}^{\prime }}d(x_{1},sy)<\theta $.
	Put $K_{1}=BK_{1}^{\prime }\Subset G$. Then $K_{1}\supset B$.
	
	Since $G$ is infinite, we can take a $K_{1}$-separated $W_{2}\Subset G$ with
	$2^{|W_{2}|}>N|K_{1}|^{4}|W_{2}|^{4}$. Put $K_{2}=K_{1}W_{2}\Subset G$. Then
	$|K_{2}|\leq|K_{1}|\cdot |W_{2}|$ and $K_{2}\supset W_{2}$. Put $K_{3}=K_{2}K_{2}^{-1}K_{2}\Subset G$.
	
	Denote by $U$ the union of the open $d_{K_{1}^{\prime }K_{3}}$-balls of radius $\varepsilon$ around each $\omega \in \Omega$. Then $GU\supset \bigcup_{j=1}^{N}Y_{j}$%
	. If $GU\neq X$, then $X\setminus GU$ contains a minimal closed $G$%
	-invariant subset, which is impossible. Thus $GU=X$. Since $X$ is compact,
	this means that there is some $K_{4}\Subset G$ with $K_{4}^{-1}U=X$, i.e.
	for any $y\in X$ one has $\min_{s\in K_{4}}\min_{\omega \in \Omega
	}d_{K_{1}^{\prime }K_{3}}(\omega ,sy)<\varepsilon $. Take a maximal $K_{2}$%
	-separated subset $W_{4}$ of $K_{4}$. Then
	\begin{align} \label{E-no asym1}
	K_{2}^{-1}K_{2}W_{4}\supset K_{4}.
	\end{align}
	
	Since $W_4$ is $K_2$-separated and $K_2\supset W_2$, we know that $W_4$ is
	$W_2$-separated.
	
	We claim that $W_2W_4$ is $K_1$-separated. Let $\gamma_1, \gamma_2\in W_2W_4$
	with $K_1\gamma_1\cap K_1\gamma_2\neq \varnothing$. Say, $\gamma_j=h_js_j$ with $h_j\in W_2, s_j\in W_4$ for
	$j=1, 2$. From
	\begin{align*}
	K_2s_1\cap K_2s_2=K_1W_2s_1\cap K_1W_2s_2\supset K_1\gamma_1\cap
	K_1\gamma_2
	\end{align*}
	we know that $K_2s_1\cap K_2s_2\neq \varnothing$, and hence $s_1=s_2$. Then $%
	K_1h_1\cap K_1h_2\neq \varnothing$, and consequently $h_1=h_2$. Therefore $%
	\gamma_1=\gamma_2$. This proves our claim.
	
	Fix $y_{1}\in \overline{Gx_{1}}$. For each $\gamma \in W_{2}W_{4}$, by our choice of $K_{1}^{\prime }$ we can
	find some $g_{\gamma }\in K_{1}^{\prime }$ with $d(x_{1},g_{\gamma }\gamma
	y_{1})<\theta $. Put $V^{\prime }=\{g_{\gamma }\gamma : \gamma \in
	W_{2}W_{4}\}\Subset G$. For any $\gamma ,\gamma ^{\prime }\in W_{2}W_{4}$,
	if $Bg_{\gamma }\gamma \cap Bg_{\gamma ^{\prime }}\gamma ^{\prime }\neq
	\varnothing $, then using that $W_{2}W_{4}$ is $K_{1}$-separated we get $
	\gamma =\gamma ^{\prime }$. Thus $V^{\prime }$ is $B$-separated, and the map
	$W_{2}W_{4}\rightarrow V^{\prime }$ sending $\gamma $ to $g_{\gamma }\gamma $
	is a bijection.
	
	Let $s\in W_{4}$. Denote by $C_{s}$ the set of $t\in G$ satisfying $
	K_{3}t\supset K_{2}s$. For each $t\in C_{s}$ one has $t\in K_{3}^{-1}K_{2}s
	$, and hence
	\begin{align*}
	|C_{s}|\leq |K_{3}^{-1}K_{2}|\leq
	|K_{2}|^{4}\leq|K_{1}|^{4}|W_{2}|^{4}<\frac{1}{N}2^{|W_{2}|}.
	\end{align*}%
	Note that for any distinct maps $\varphi ,\varphi ^{\prime
	}\colon (W_{2})s\rightarrow \{1,2\}$, one has $\max_{\gamma \in (W_{2})s}d(x_{\varphi
		(\gamma )},x_{\varphi ^{\prime }(\gamma )})>5\varepsilon $. Thus for each $
	t\in C_{s}$ and $\omega \in \Omega $, there is at most one map $\varphi
	\colon (W_{2})s\rightarrow \{1,2\}$ which satisfies that $\max_{\gamma \in (W_{2})s}d(g_{\gamma
	}\gamma t^{-1}\omega ,x_{\varphi (\gamma )})<2\varepsilon $. Since $
	N|C_{s}|<2^{|W_{2}|}$, we can find some map $\varphi _{s}\colon (W_{2})s\rightarrow
	\{1,2\}$ such that for every $t\in C_{s}$ and $\omega \in \Omega $ one has
	\begin{align} \label{E-no asym2}
	\max_{\gamma \in (W_{2})s}d(g_{\gamma }\gamma t^{-1}\omega ,x_{\varphi
		_{s}(\gamma )})\geq 2\varepsilon .
	\end{align}
	
	Since $W_{4}$ is $W_{2}$-separated, $W_{2}W_{4}$ is the disjoint union of $
	(W_{2})s$ for $s\in W_{4}$. Then we can define a map $\varphi
	\colon W_{2}W_{4}\rightarrow \{1,2\}$ by taking $\varphi $ to be $\varphi _{s}$ on
	$(W_{2})s$ for all $s\in W_{4}$. Put $V=\{g_{\gamma }\gamma : \gamma \in
	\varphi ^{-1}(2)\}\subset V^{\prime }$. Then $V$ is $B$-separated. For
	each $\gamma \in \varphi ^{-1}(2)$, putting $x_{g_{\gamma }\gamma
	}=(g_{\gamma }\gamma )^{-1}x_{2}$ and using $d(x_{1},g_{\gamma }\gamma
	y_{1})<\theta $ we have
	\begin{align*}
	d(ux_{g_{\gamma }\gamma },uy_{1})& =d(u(g_{\gamma }\gamma
	)^{-1}x_{2},u(g_{\gamma }\gamma )^{-1}(g_{\gamma }\gamma y_{1})) \\
	& \leq d(u(g_{\gamma }\gamma )^{-1}x_{2},u(g_{\gamma }\gamma
	)^{-1}x_{1})+d(u(g_{\gamma }\gamma )^{-1}x_{1},u(g_{\gamma }\gamma
	)^{-1}(g_{\gamma }\gamma y_{1})) \\
	& <\frac{\tau}{2}+\frac{\tau}{2}=\tau
	\end{align*}
	for all $u\in (B\setminus A)g_{\gamma }\gamma $. Thus there is some $z\in X$
	satisfying $d_{Bg_{\gamma}\gamma}(z,x_{g_{\gamma }\gamma })\leq \frac{\varepsilon}{2}$ for all $
	\gamma \in \varphi ^{-1}(2)$ and $d_{G \setminus AV}(z,y_{1})\leq \frac{\varepsilon}{2}$.
	
	By our choice of $K_{4}$, there are some $t\in K_{4}$ and $\omega \in \Omega
	$ with $d_{K_{1}^{\prime }K_{3}}(\omega ,tz)<\varepsilon $. From \eqref{E-no
		asym1} we have $t\in K_{2}^{-1}K_{2}s$ for some $s\in W_{4}$. Then $s\in
	K_{2}^{-1}K_{2}t$, and hence $K_{2}s\subset K_{3}t$. Thus $t\in C_{s}$,
	and
	\begin{align*}
	\max_{\gamma \in (W_{2})s}d(g_{\gamma }\gamma t^{-1}\omega ,g_{\gamma }\gamma
	z)& \leq d_{K_{1}^{\prime }(W_{2})s}(t^{-1}\omega ,z) \\
	& \leq d_{K_{1}^{\prime }K_{2}s}(t^{-1}\omega ,z) \\
	& \leq d_{K_{1}^{\prime }K_{3}t}(t^{-1}\omega ,z) \\
	& =d_{K_{1}^{\prime }K_{3}}(\omega ,tz)<\varepsilon .
	\end{align*}
	From \eqref{E-no asym2} we can find some $\gamma \in (W_{2})s$ with $%
	d(g_{\gamma }\gamma t^{-1}\omega ,x_{\varphi _{s}(\gamma )})\geq
	2\varepsilon $. Then
	\begin{align} \label{E-no asym3}
	d(x_{\varphi _{s}(\gamma )},g_{\gamma }\gamma z)\geq d(g_{\gamma }\gamma
	t^{-1}\omega ,x_{\varphi _{s}(\gamma )})-d(g_{\gamma }\gamma t^{-1}\omega
	,g_{\gamma }\gamma z)>2\varepsilon -\varepsilon =\varepsilon .
	\end{align}
	
	Now we have $\varphi _{s}(\gamma )=1$ or $2$. Consider first the case $
	\varphi _{s}(\gamma )=2$. We have $d(g_{\gamma }\gamma z,x_{2})=d(g_{\gamma
	}\gamma z,g_{\gamma }\gamma x_{g_{\gamma }\gamma })\leq \frac{\varepsilon}{2}$,
	contradicting~\eqref{E-no asym3}. Next consider the case $\varphi
	_{s}(\gamma )=1$. We have $g_{\gamma }\gamma \in V^{\prime }\setminus
	V\subset G\setminus AV$. Then $d(g_{\gamma }\gamma z,g_{\gamma }\gamma
	y_{1})\leq \frac{\varepsilon}{2}$. Therefore
	\begin{equation*}
	d(x_{1},g_{\gamma }\gamma z)\leq d(x_{1},g_{\gamma }\gamma
	y_{1})+d(g_{\gamma }\gamma y_{1},g_{\gamma }\gamma z)<\theta +\frac{\varepsilon}{2}<\varepsilon ,
	\end{equation*}%
	again contradicting~\eqref{E-no asym3}. Thus $X$ has no off-diagonal
	asymptotic pairs $(x_{1},x_{2})$ with $\overline{Gx_{1}}$ minimal.
\end{proof}

\section{Topological entropy and asymptotic pairs}\label{section:ent_asympt}

	In this section we will explore the consequences of having Markovian properties on the relation between asymptotic pair and independence entropy pairs.

\subsection{From asymptotic pairs to independence entropy pairs}

In this subsection we shall give conditions under which the existence of an off-diagonal asymptotic pair gives rise to IE-pairs. We provide a result for orbit IE-pairs which applies to all groups (Theorem~\ref{T-A to IE}), and a result for $\Sigma$-IE-pairs  which applies to sofic groups (Theorem~\ref{Theorem_asymptotic_pairs_give_SOFIC_entropy_pairs}).

\begin{theorem} \label{T-A to IE}
	Let $G\curvearrowright X$ be an expansive action with the TMP. Let $k\in {\mathbb{N}}$ and $(x_{1},\dots ,x_{k},x_{k+1})\in X^{k+1}$ such that $(x_{1},\dots ,x_{k})\in \IE_{k}(X,G)$ and $(x_{k},x_{k+1})$ is an asymptotic pair. Then $(x_{1},\dots ,x_{k},x_{k+1})\in \IE_{k+1}(X,G)$.
\end{theorem}

Theorem~\ref{T-A to IE} follows from Theorem~\ref{P-exp TMP to uniform TMP} and Proposition~\ref{P-A to IE} below.

Whenever $G$ is amenable, we have $x \in \IE_1(X,G)$ if and only if $x$ is in the the support of an invariant Borel probability measure, see~\cite[Lemma 12.6]{KerrLiBook2016}. For a non-amenable group $G$ every element in the support of some $G$-invariant Borel probability measure is in $\IE_1(X,G)$ but the converse may not hold, see~\Cref{theorem_naive_IE_cites}.

A direct application of Theorems~\ref{T-A to IE} and~\ref{theorem_naive_IE_cites} yields the following result.

\begin{corollary}\label{color_naive_UPE}
	Let $G\curvearrowright X$ be an expansive action with the TMP.
	\begin{enumerate}
		\item Suppose that $(x,y)\in \R_2(X,G)\setminus \triangle_2(X)$ and $x \in \IE_1(X,G)$, then $(x,y) \in \IE_2(X,G)$. In particular we have $\hnaive(G\curvearrowright X) > 0$.
		\item If $\IE_1(X,G) = X$ and $\overline{\R_2(X,G)} = X^2$, then $G\curvearrowright X$ has naive UPE of all orders.
	\end{enumerate}
\end{corollary}

For $x \in X$ and $\delta>0$, denote by $B_{\delta}(x) = \{  y\in X : d(y,x) < \delta  \}$ the open ball of radius $\delta$ centered at $x$.

\begin{proposition}\label{P-A to IE}
    Suppose that $G\curvearrowright X$ has the uniform TMP. Let $k\in {\mathbb{N}}$ and $(x_{1},\dots ,x_{k},x_{k+1})\in X^{k+1}$ such that $(x_{k},x_{k+1})$ is an asymptotic pair. Then $(x_{1},\dots ,x_{k})\in \IE_{k}(X,G)$ if and only if $(x_{1},\dots ,x_{k},x_{k+1})\in \IE_{k+1}(X,G)$.
\end{proposition}

\begin{proof}
	The ``if'' part is trivial. We shall prove the ``only if'' part.
	Let $\varepsilon >0$. Then there is some $\tau >0$ such that for any $A\Subset G$ there is some $B^{\prime }\Subset G$ containing $A$ such that
	for any $B^{\prime }$-separated  $V\Subset G$ and any $x_v\in  X$ for $v\in V$ and $y\in X$ satisfying $d_{(B^{\prime }\setminus A)v}(x_{v},y)\le \tau $ for all $v\in V$, there is some $z\in X$ such that $d_{B'v}(z,x_{v})\leq \frac{\varepsilon}{2}$ for all $v\in V$ and $d_{G\setminus AV}(z,y)\leq \frac{\varepsilon}{2}$.
	
	Since $(x_{k},x_{k+1})$ is an asymptotic pair, there is some $A\Subset G$ containing
	$e_{G}$ such that $d_{G \setminus A}(x_{k},x_{k+1})<\frac{\tau}{2}$.
	Then we have $B^{\prime } \supset A$ as above. Take $0<\delta <\frac{\varepsilon}{2}$ such
	that for any $x,y\in X$ with $d(x,y)\le \delta $ one has $d_{B^{\prime }\setminus A}(x,y)<\frac{\tau}{2}$.
	
	Since $(x_{1},\dots ,x_{k})\in \IE_{k}(X,G)$, the tuple $%
	(B_{\delta}(x_{1}),\dots ,B_{\delta}(x_{k}))$ has independence density $q>0$.
	
	Let $F\Subset G$. Then there is some $F^{\prime }\subset F$ with $|F^{\prime }|\geq q|F|$ such that $F^{\prime }$ is an independence set for $(B_{\delta}(x_{1}),\dots ,B_{\delta}(x_{k}))$. Take a maximal $B^{\prime }$%
	-separated subset $J$ of $F^{\prime }$. Then $B^{\prime -1}B^{\prime
	}J\supset F^{\prime }$ and hence
	\begin{equation*}
	|J|\geq \frac{|F^{\prime }|}{|B^{\prime}|^2}\geq \frac{q|F|}{|B^{\prime}|^2}.
	\end{equation*}
	
	We claim that $J$ is an independence set for $(B_{\varepsilon}(x_{1}),\dots
	,B_{\varepsilon}(x_{k}),B_{\varepsilon}(x_{k+1}))$. Consider an arbitrary map $f \colon
	J\to \{1,\dots,k+1\}$ and put $V=f^{-1}(k+1)$, which is $B^{\prime}$-separated. Define $g\colon J\to
	\{1,\dots,k\}$ by $g=f$ on $J \setminus V$ and $g=k$ on $V$. Then there is
	some $y\in \bigcap_{s\in J}s^{-1}B_{\delta}(x_{g(s)})$. Put $x_{v}=v^{-1}x_{k+1}$ for each $v\in V$. For any $v\in V$  we have $d(x_{k},vy)\le \delta $, and hence for $s\in
	(B^{\prime }\setminus A)v$,
	\begin{align*}
	d(sx_{v},sy)& =d(sv^{-1}x_{k+1},sv^{-1}(vy))\leq d(sv^{-1}x_{k+1},sv^{-1}x_{k})+d(sv^{-1}x_{k},sv^{-1}(vy))<\frac{\tau}{2}+\frac{\tau}{2}=\tau .
	\end{align*}%
	Thus there is some $z\in X$ such that $d_{B'v}(z,x_{v})\leq \frac{\varepsilon}{2}$ for
	all $v\in V$ and $d_{G \setminus AV}(z,y)\leq \frac{\varepsilon}{2}$. For any $v\in f^{-1}(k+1)=V$, we have $d(vz,x_{k+1})=d(vz,vx_{v})\leq \frac{\varepsilon}{2}$. For any $v\in J\setminus V$,
	we have
	\begin{equation*}
	d(vz,x_{f(v)})\leq d(vz,vy)+d(vy,x_{f(v)})\leq \frac{\varepsilon}{2}+\delta
	<\varepsilon .
	\end{equation*}%
	Therefore $z\in \bigcap_{s\in J}s^{-1}B_{\varepsilon}(x_{f(s)})$. This proves our claim.
	
	Now we conclude that $(B_{\varepsilon}(x_{1}),\dots
	,B_{\varepsilon}(x_{k}),B_{\varepsilon}(x_{k+1}))$ has independence density at least $\frac{q}{|B^{\prime|^2}}$. Therefore $(x_{1},\dots ,x_{k},x_{k+1})\in \IE_{k+1}(X,G)$.
\end{proof}

Next we shall prove Theorem~\ref{Theorem_asymptotic_pairs_give_SOFIC_entropy_pairs}, which is the analogue of Theorem~\ref{T-A to IE} for sofic topological entropy.

\begin{theorem}\label{Theorem_asymptotic_pairs_give_SOFIC_entropy_pairs}
	Let $G$ be a sofic group and $\Sigma$ a sofic
	approximation sequence for $G$. Suppose that $G\curvearrowright X$ is expansive and has the TMP. Let $k\in {\mathbb{N}}$ and $(x_{1},\dots
	,x_{k},x_{k+1})\in X^{k+1}$ such that $(x_{1},\dots ,x_{k})\in \IE_{k}^{\Sigma }(X,G)$ and $(x_{k},x_{k+1})$ is an asymptotic pair. Then $%
	(x_{1},\dots ,x_{k},x_{k+1})\in \IE_{k+1}^{\Sigma }(X,G)$.
\end{theorem}

A direct consequence of Theorems~\ref{theorem_sofic_IE_cites} and \ref{Theorem_asymptotic_pairs_give_SOFIC_entropy_pairs} yields the following.

\begin{corollary}\label{color_sofic_UPE}
	Let $G$ be a sofic group,  $\Sigma$ a sofic approximation sequence for $G$, and $G\curvearrowright X$ an expansive action with the TMP.
	\begin{enumerate}
		\item Suppose that $(x,y)\in \R_2(X,G)\setminus \triangle_2(X)$ and $x \in \IE^{\Sigma}_1(X,G)$, then $(x,y) \in \IE^{\Sigma}_2(X,G)$. In particular we have  $\hsof(G\curvearrowright X) > 0$.
		\item If $\IE^{\Sigma}_1(X,G) = X$ and $\overline{\R_2(X,G)} = X^2$, then $G\curvearrowright X$ has sofic UPE of all orders.
	\end{enumerate}
\end{corollary}

\begin{example}\label{hsmodel}
	Let $G$ be a sofic group, $F \Subset G$ which does not contain $e_G$, and  $\Sigma$  a sofic approximation sequence for $G$. The \define{$F$ hard-square model} of $G$ is the subshift $X_{\mbox{hard}, F} \subset \{0,1\}^G$ consisting of all $x \in \{0,1\}^G$ for which $x^{-1}(1)$ is an independence set in the right Cayley graph of $G$ given by $F$, that is \[X_{\mbox{hard}, F} = \left\{ x\in \{0,1\}^G : \mbox{ for every } g \in G \mbox{ and } s \in F, x(g)x(gs)=0\right\}. \]
	
	$X_{\mbox{hard}, F}$ is a subshift of finite type, and therefore has the POTP and thus the TMP by Proposition~\ref{prop_POTP_implies_sTMP}. It is easy to see that any point which is asymptotic to $0^G$ is in $\IE_1^{\Sigma}(X_{\mbox{hard},F},G)$ and hence $\IE_1^{\Sigma}(X_{\mbox{hard},F},G) = X_{\mbox{hard},F}$, moreover, $\overline{\R_2(X_{\mbox{hard},F},G)} = (X_{\mbox{hard},F})^2$. By~\Cref{color_sofic_UPE} we obtain that $G\curvearrowright X_{\mbox{hard},F}$ has sofic UPE of all orders for every sofic approximation sequence.
	Although this result is not very surprising, as far as we now, this is the first proof of (uniform) positive topological sofic entropy of hard-square models on sofic groups.
\end{example}

To prove Theorem~\ref{Theorem_asymptotic_pairs_give_SOFIC_entropy_pairs} we need to make some preparations. The following is a result of  Karpovsky and Milman \cite{KarpovskyMilman1954}. See also~\cite[Lemma 12.14]{KerrLiBook2016}.

\begin{lemma} \label{L-KM}
Let $k\ge 2$ and $\lambda>1$. Then there is a $c>0$ such that for every $n\in \mathbb{N}$ and $S\subset \{1, 2, \dots, k\}^{\{1, 2, \dots, n\}}$ with
$|S|\ge ((k-1)\lambda)^n$ there is an $I\subset \{1, 2, \dots, n\}$ satisfying $|I|\ge cn$ and $S|_I=\{1, 2, \dots, k\}^I$.
\end{lemma}

\begin{lemma} \label{L-indep set are good}
Let $G$ be a sofic group and $\Sigma =\{\sigma_{i}\colon G \to \Sym(n_i)\}_{i\in {\mathbb{N}}}$ a sofic
	approximation sequence for $G$. Let $G\curvearrowright X$ be an action and $\mathbf{A}=(A_1, \dots, A_k)$ a tuple of subsets of $X$. Then $\mathbf{A}$ has positive upper independence  density over $\Sigma$ if and only if there exists $q>0$ such that for every $F\Subset G$ and $\delta>0$ there exists an  infinite set of  $i$ for which there is a set $J_i\subset \{1, \dots, n_i\}$ so that $|J_i|\ge qn_i$ and for every map $\omega: J_i\rightarrow \{1, \dots, k\}$ there is some $\varphi\in \Map(d, F, \delta, \sigma_i)$ satisfying that $\varphi(a)\in A_{\omega(a)}$ for all $a\in J_i$ and $d(s\varphi(a), \varphi(\sigma_i(s)a))\le \delta$ for all $a\in J_i$ and $s\in F$.
\end{lemma}
\begin{proof} The ``if'' part is trivial.

Assume that  $\mathbf{A}$ has  positive upper independence density over $\Sigma$. Then there exists $q>0$ such that for every $F\Subset G$ and $\delta>0$ there is an infinite set $I_{F, \delta}$ of $i$ for which $\mathbf{A}$ has a $(d, F, \delta, \sigma_i)$-independence set $J_i\subset\{1, \dots, n_i\}$ of cardinality at least $qn_i$.

Consider the case $k\ge 2$. Take $\eta>0$ small such that $k^{1-\eta}>k-1$.
Put $\lambda=k^{1-\eta}/(k-1)>1$. Then we have $c>0$ given by Lemma~\ref{L-KM} for $k$ and $\lambda$.
From Stirling's approximation formula (see example \cite[Lemma 10.1]{KerrLiBook2016}) it is easy to see that there is some $\tau'>0$ depending only on $k^{\eta q}$ such that $\sum_{0\le j\le \tau' n}\binom{n}{j}\le k^{\eta qn}$ for all  $n\in \mathbb{N}$. Put $\tau=\min(\tau', cq/2)>0$.

In the case $k=1$, we put $\tau=q/2$.

Let $F\Subset G$ and $\delta>0$. Put $\delta'=(\frac{\tau}{|F|+1})^{1/2}\delta >0$.

Let $i\in I_{F, \delta'}$. Then $\mathbf{A}$ has a $(d, F, \delta, \sigma_i)$-independence set $J_i\subset\{1, \dots, n_i\}$ of cardinality at least $qn_i$. For each $\omega: J_i\rightarrow \{1, 2, \dots, k\}$, take $\varphi_\omega\in \Map(d, F, \delta', \sigma_i)$ such that $\varphi_\omega(a)\in A_{\omega(a)}$ for all $a\in J_i$. Denote by $W_\omega$ the set of all $a\in \{1, \dots, n_i\}$ satisfying $d(s\varphi_\omega(a), \varphi_\omega(\sigma_i(s)a))\le \delta$ for all $s\in F$. Since $\varphi_\omega\in \Map(d, F, \delta', \sigma_i)$, we have $|W_\omega|/n_i\ge 1-\tau$.
In the case $k=1$, there is only one $\omega$ and setting $J_i'=J_i\cap W_\omega$ we have $|J_i'|/n_i\ge q/2$. Thus we may assume $k\ge 2$.
From our choice of $\tau$, we get that the number of choices for $W_\omega$ is at most $k^{\eta qn_i}$. Thus there is a set $\Omega\subset \{1, 2, \dots, k\}^{J_i}$ such that $W_\omega$ is the same for all $\omega\in \Omega$, which we denote by $W$, and
$$|\Omega|\ge k^{|J_i|}/k^{\eta qn_i}\ge k^{|J_i|}/k^{\eta |J_i|}=k^{(1-\eta)|J_i|}=((k-1)\lambda)^{|J_i|}.$$
By our choice of $c$, we find a $J_i'\subset J_i$ such that $|J_i'|\ge c|J_j|$ and $\Omega|_{J_i'}=\{1, 2, \dots, k\}^{J_i'}$.
Put $J_i''=J_i'\cap W$. Then
$$ |J_i''|/n_i\ge |J_i'|/n_i-(1-|W|/n_i)\ge cq-\tau\ge cq/2.$$
Let $g:J_i''\rightarrow \{1, 2, \dots, k\}$ be an arbitrary map. Take a $\omega\in \Omega$ such that $\omega|_{J_i''}=g$. Then $\varphi_\omega\in \Map(d, F, \delta', \sigma_i)\subset \Map(d, F, \delta, \sigma_i)$, $\varphi_\omega(a)\in A_{\omega(a)}=A_{g(a)}$ for all $a\in J_i''$, and for any $a\in J_i''$ and $s\in F$ we have $a\in W=W_{\omega}$, whence $d(s\varphi_\omega(a), \varphi_\omega(\sigma_i(s)a))\le \delta$. This proves the ``only if'' part.
\end{proof}

We are ready to prove Theorem~\ref{Theorem_asymptotic_pairs_give_SOFIC_entropy_pairs}.

\begin{proof}[Proof of Theorem~\ref{Theorem_asymptotic_pairs_give_SOFIC_entropy_pairs}]
	Let $c>0$ be an expansivity constant for $G\curvearrowright X$. Say, $\Sigma =\{\sigma_{i}\colon G \to \Sym(n_i)\}_{i\in {\mathbb{N}}}$.
	
	Let $0<\varepsilon<\frac{c}{2}$. It suffices to show that the tuple $(B_{\varepsilon}(x_{1}),\dots
	,B_{\varepsilon}(x_{k}),B_{\varepsilon}(x_{k+1}))$ has positive upper independence density over $\Sigma$. By~\Cref{P-exp TMP to uniform TMP} we have that $G \curvearrowright X$ satisfies the uniform TMP and thus there is some $\tau>0$ such that for any $A \Subset G$ there is some $B'\Subset G$ containing $A$ such that for any $B'$-separated  $V\Subset G$ and any $x_v\in X$ for $v\in V$ and $y\in X$ satisfying $d_{(B'\setminus A)v}(x_v, y)\le \tau$ for all $v\in V$, there is some $z\in X$ such that $d_{B'v}(z, x_v)\leq \frac{\varepsilon}{2}$ for all $v\in V$and $d_{G \setminus AV}(z, y)\leq \frac{\varepsilon}{2}$.
	
	Since $(x_k, x_{k+1})$ is asymptotic, there is some $A\Subset G$ containing $e_G$ such that $d_{G \setminus A}(x_k, x_{k+1})<\frac{\tau}{2}$. Then we have $B'$ for $A$ as above.
	Take $0<\theta<\frac{\varepsilon}{2}$ such that for any $x,y\in X$ with $d(x, y)\le 2\theta$ one has $d_{B'\setminus A}(x, y)<\frac{\tau}{2}$.
	
	Since $(x_1, \dots, x_k)\in \IE^\Sigma_k(X, G)$, the tuple $(B_{\theta}(x_1), \dots, B_{\theta}(x_k))$ has positive upper independence density over $\Sigma$. By Lemma~\ref{L-indep set are good}
there is some $q>0$ such that for every $F\Subset G$ and $\delta>0$ there is an infinite set $I_{F, \delta}$ of $i$ for which
there is a set $W_{F, \delta, i}\subset \{1, \dots, n_i\}$ so that $|W_{F, \delta, i}|\ge qn_i$ and for every map $\omega: W_{F, \delta, i}\rightarrow \{1, \dots, k\}$ there is a $\varphi\in \Map(d, F, \delta, \sigma_i)$ satisfying that $\varphi(a)\in B_\theta(x_{\omega(a)})$ for all $a\in W_{F, \delta, i}$ and $d(s\varphi(a), \varphi(\sigma_i(s)a))\le \delta$ for all $a\in W_{F, \delta, i}$ and $s\in F$.

	Let $F\Subset G$ and $\delta>0$.
	By~\Cref{L-close PO} there is some $K_1\Subset G$ such that for any $x,y \in X$, if $d_{K_1}(x, y)\leq c$, then $d(x, y)<\frac{\delta}{2}$. Take $0<\kappa\le \theta$ such that for any $x, y\in X$ with $d(x, y)\le \kappa$, one has $d_{K_1}(x, y)<\frac{c}{2}$.
	
	Put $K=B'\cup F\cup A^{-1}K_1\cup A^{-1}K_1F \Subset G$. Take $0<\eta<\frac{q}{2}$ such that $4 \eta \diam(X, d)^2\leq (\frac{\delta}{2})^2$.
	Put $\delta'=\kappa \sqrt{\eta/|K|}>0$.
	Let $i\in I_{K, \delta'}$ be sufficiently large so that $|W_i|\ge (1-\eta)n_i$ for
	\begin{align*}
	W_i=&\{a\in \{1,\dots,n_i\} :  \sigma_i(s)a\neq \sigma_i(t)a \mbox{ for all distinct } s,t \in K, \\
	& \quad \quad \mbox{ and } \sigma_i(s)\sigma_i(t)a=\sigma_i(st)a \mbox{ for all } s, t\in K\}.
	\end{align*}
	Take a maximal subset $J_i$ of $W_i\cap W_{K, \delta', i}$ subject to the condition that $\sigma_i(B')a\cap \sigma_i(B')b=\varnothing$ for all distinct $a, b\in J_i$. Then $(\sigma_i(B'))^{-1}\sigma_i(B')J_i\supset W_i\cap W_{K, \delta', i}$, and hence \[ |J_i|\ge \frac{|W_i\cap W_{K, \delta', i}|}{|B'|^2} \geq \frac{(q-\eta)n_i}{|B'|^2} \geq \frac{qn_i}{2|B'|^2}.\]
	
	Now it suffices to show that $J_i$ is a $(d, F, \delta, \sigma_i)$-independence set for $(B_{\varepsilon}(x_{1}),\dots
	,B_{\varepsilon}(x_{k}),B_{\varepsilon}(x_{k+1}))$. Let $f \colon J_i\rightarrow \{1,\dots,k+1\}$ be an arbitrary map and define $g\colon  J_i\rightarrow \{1,\dots,k\}$ by $g=f$ on $J_i\setminus f^{-1}(k+1)$ and $g=k$ on $f^{-1}(k+1)$.
	Then there is some $\varphi\in \Map(d, K, \delta', \sigma_i)$ such that $\varphi(a)\in B_{\theta}(x_{g(a)})$ for all $a \in J_i$ and
$d(s\varphi(a), \varphi(\sigma_i(s)a))\le \delta'$ for all $a\in W_{K, \delta', i}$ and $s\in K$.
	
	Denote by $W_\varphi$ the set of $a\in \{1, \dots, n_i\}$ satisfying $d(s\varphi(a), \varphi(\sigma_i(s)a))\le \kappa$ for all $s\in K$. Since $\varphi\in \Map(d, K, \delta', \sigma_i)$, we have $|W_\varphi|/n_i\ge 1-\eta$.
	
	Let $a\in W_i\cap W_\varphi$. Denote by $V_a$ ($V'_a$ resp.) the set of $t\in K$ satisfying $\sigma_i(t)a\in f^{-1}(k+1)$ ($\sigma_i(t)a\in J_i$ resp.). For each $t\in V'_a$, we have $\sigma_i(B't)a=\sigma_i(B')\sigma_i(t)a$. For any distinct $t_1, t_2\in V'_a$, we have $\sigma_i(t_1)a\neq \sigma_i(t_2)a$, and hence $\sigma_i(B't_1)a\cap \sigma_i(B't_2)a=\sigma_i(B')\sigma_i(t_1)a\cap \sigma_i(B')\sigma_i(t_2)a=\varnothing$, which implies that $B't_1\cap B't_2=\varnothing$. Thus $V'_a$ is $B'$-separated. Put $y_a=\varphi(a)$, and for each $t\in V_a$ put $x_{a, t}=t^{-1}x_{k+1}$. For any $t\in V_a$ and $s\in (B'\setminus A)t$, we have
	$$d(x_k, t\varphi(a))\le d(x_k, \varphi(\sigma_i(t)a))+d(\varphi(\sigma_i(t)a), t\varphi(a))\le \theta+ \kappa\le 2\theta,$$
	and hence
	\begin{align*}
	d(sx_{a, t}, sy_a)&=d(st^{-1}x_{k+1}, st^{-1}(t\varphi(a)))\le d(st^{-1}x_{k+1}, st^{-1}x_k)+d(st^{-1}x_k, st^{-1}(t\varphi(a)))< \frac{\tau}{2}+  \frac{\tau}{2}=\tau.
	\end{align*}
	Then there is some $\psi(a)\in X$ such that $d_{B't}(\psi(a), x_{a, t})\le \frac{\varepsilon}{2}$ for all $t\in V_a$ and $d_{G \setminus AV_a}(\psi(a), y_a)\le \frac{\varepsilon}{2}$.
	
	For $a\in \{1,\dots,n_i\}\setminus (W_i\cap W_\varphi)$, take $\psi(a)$ to be any point in $X$. Then we get a map $\psi\colon \{1,\dots,n_i\}\to X$. We claim that $\psi\in \Map(d, F, \delta, \sigma_i)$.
	
	Indeed, let $s\in F$ and put
	$W=(W_i\cap W_\varphi) \cap (\sigma_i(s))^{-1}(W_i\cap W_\varphi)$.
	We have
	\[|\{1,\dots,n_i\}\setminus W|\le 2|\{1,\dots,n_i\}\setminus W_i|+2|\{1,\dots,n_i\}\setminus W_\varphi|\le 4\eta n_i.\]
	Let $a\in W$ and $\gamma \in K_1$. If $\gamma s\in At$ for some $t\in V_a$, then
	$ts^{-1}\in A^{-1}K_1\subset K$ and
	$\sigma_i(ts^{-1})\sigma_i(s)a=\sigma_i(t)a\in f^{-1}(k+1)$, and hence $ts^{-1}\in V_{\sigma_i(s)a}$ and $\gamma \in Ats^{-1}$, from which we get
	\begin{align*}
d(\gamma s\psi(a), \gamma \psi(\sigma_i(s)a))&\le d(\gamma s\psi(a), \gamma sx_{a, t})+d(\gamma sx_{a, t},\gamma x_{\sigma_i(s)a, ts^{-1}})+ d(\gamma x_{\sigma_i(s)a, ts^{-1}}, \gamma \psi(\sigma_i(s)a))\\
	&\le \frac{\varepsilon}{2}+d(\gamma s t^{-1}x_{k+1}, \gamma st^{-1}x_{k+1})+\frac{\varepsilon}{2}=\varepsilon<c.
	\end{align*}
	If $\gamma \in At$ for some $t\in V_{\sigma_i(s)a}$, then $ts\in A^{-1}K_1F\subset K$ and $\sigma_i(ts)a=\sigma_i(t)\sigma_i(s)a\in f^{-1}(k+1)$, and hence $ts\in V_a$ and
	$\gamma s\in A(ts)$. Thus, if
	$\gamma s \in G\setminus AV_a$, then $\gamma \in G\setminus A V_{\sigma_i(s)a}$, and using  $d(s \varphi(a), \varphi(\sigma_i(s)a))\le \kappa$  and $\gamma \in K_1$ we have
	\begin{align*}
	d(\gamma s\psi(a), \gamma \psi(\sigma_i(s)a))&\le d(\gamma s\psi(a), \gamma s y_a)+d(\gamma s y_a, \gamma y_{\sigma_i(s)a})+d(\gamma y_{\sigma_i(s)a}, \gamma \psi(\sigma_i(s)a))\\
	&\le \frac{\varepsilon}{2}+d(\gamma s \varphi(a), \gamma \varphi(\sigma_i(s)a))+\frac{\varepsilon}{2}<\frac{\varepsilon}{2}+\frac{c}{2}+\frac{\varepsilon}{2}\le c.
	\end{align*}
	We conclude that $d_{K_1}(s\psi(a), \psi(\sigma_i(s)a))\leq c$ for every $a\in W$. From our choice of $K_1$, we obtain that $d(s\psi(a),  \psi(\sigma_i(s)a))<\frac{\delta}{2}$ for all $a\in W$. Therefore
	\begin{align*}
	\frac{1}{n_i}\sum_{a\in \{1,\dots,n_i\}}d(s\psi(a), \psi(\sigma_i(s)a))^2
	&= \frac{1}{n_i}\sum_{a\in W}d(s\psi(a), \psi(\sigma_i(s)a))^2+ \frac{1}{n_i}\sum_{a\in \{1,\dots,n_i\}\setminus W}d(s\psi(a), \psi(\sigma_i(s)a))^2\\
	&\le \frac{|W|}{n_i}\left(\frac{\delta}{2}\right)^2+\diam(X, d)^2\frac{|\{1,\dots,n_i\}\setminus W|}{n_i}\\
	&\le \left(\frac{\delta}{2}\right)^2+4\eta \diam(X, d)^2\le 2\left(\frac{\delta}{2}\right)^2<\delta^2.
	\end{align*}
	This proves our claim.
	
 We have $W_{K, \delta', i}\subset W_\varphi$, and hence $J_i\subset W_i\cap W_{K, \delta', i}\subset W_i\cap W_\varphi$.
	Since $e_G\in K$, we have $\sigma_i(e_G)a=a$ for all $a\in W_i$.
	For any $a\in f^{-1}(k+1)$, we have $e_G\in V_a$, and hence $d(\psi(a), x_{k+1})=d(\psi(a), x_{a, e_G})\le \frac{\varepsilon}{2}$.
	For any $a\in J_i\setminus f^{-1}(k+1)$, we have $e_G\in V'_a\setminus V_a\subset G\setminus AV_a$, whence $d(\psi(a), \varphi(a))=d(\psi(a), y_a)\le \frac{\varepsilon}{2}$ and consequently
	\[d(\psi(a), x_{f(a)})\le d(\psi(a), \varphi(a))+d(\varphi(a), x_{f(a)})\le \frac{\varepsilon}{2}+\theta<\varepsilon.\]
	We conclude that $\psi(a)\in B_{\varepsilon}(x_{f(a)})$ for all $a\in J_i$.
	This shows that $J_i$ is a $(d, F, \delta, \sigma_i)$-independence set for $(B_{\varepsilon}(x_{1}),\dots
	,B_{\varepsilon}(x_{k}),B_{\varepsilon}(x_{k+1}))$ as desired.
\end{proof}

	\subsection{From independence entropy pairs to asymptotic pairs}
	
	In this subsection we prove~\Cref{T-exp sTMP IE to asym} which gives sufficient conditions under which positive topological entropy for an action of an amenable group implies the existence of off-diagonal asymptotic pairs.

\begin{theorem} \label{T-exp sTMP IE to asym}
Let $G$ be an amenable group, and $G\curvearrowright X$ an expansive action with the strong TMP. Then  $\IE_k(X, G)\subseteq \overline{\mathtt{A}_k(X,G)}$ for every $k\ge 2$.
\end{theorem}

Theorem~\ref{T-exp sTMP IE to asym} follows from Proposition~\ref{P-amenable_implications_trivial}, Lemma~\ref{lema_expansivo} and  Proposition~\ref{ida} below.

Combining Theorem~\ref{T-exp sTMP IE to asym} and part (2) of Theorem~\ref{theorem_naive_IE_cites} we obtain the following corollary.
Part (1) of it  was proven by Meyerovitch~\cite[Theorem 1.4]{Meyerovitch2017} under the stronger assumption that $G\curvearrowright X$ has the POTP.

	\begin{corollary} \label{C-exp sTMP IE to asym}
		Let $G$ be an amenable group, and $G\curvearrowright X$ an expansive action with the strong TMP.
\begin{enumerate}
\item If $\htop(G\curvearrowright X)>0$, then for every $k\ge 2$ there exists an $(x_1, \dots, x_k)\in \mathtt{A}_k(X,G)$ with $x_1, \dots, x_k$ distinct.
\item If $G\curvearrowright X$ has UPE of order $k$ for some $k\ge 2$, then $\mathtt{A}_k(X,G)$ is dense in $X^k$.
\end{enumerate}
	\end{corollary}

\begin{remark} \label{R-exp mTMP IE to asym}
Theorem~\ref{T-exp sTMP IE to asym} and Corollary~\ref{C-exp sTMP IE to asym} also hold with the strong TMP replaced by the mean TMP.
\end{remark}

The following result was proven in~\cite[Theorem 4.2]{bg2018hierarchy} under the stronger assumption of the POTP. The proof is very similar but it is in fact much more natural in this general context.

	\begin{proposition}\label{ida}
Let $G\curvearrowright X$ be an action with the mean TMP. For every $k\ge 2$ and $\varepsilon>0$ we have $\IE_k(X,G)\subset \overline{\mathtt{A}^{\varepsilon}_k(X,G)}$.
	\end{proposition}
	
	\begin{proof}
		Let $(x_1,  \dots, x_k)\in\IE_k(X,G)$.
It suffices to show that for any $\varepsilon>0$ there exists an $\varepsilon
		$-asymptotic tuple in $B_{\varepsilon}(x_1)\times \dots \times B_{\varepsilon}(x_k)$. By the mean TMP there exist $\delta >0$ and an increasing sequence
		$\left\{ F_{n}\right\}_{n \in \NN}$ of finite subsets of $G$ with union $G$ such that for each $n \in \NN$ there is an $(\frac{\varepsilon}{2},\delta)$-memory set $\widetilde{F_n}$ for $F_n$ so that $|\widetilde{F_n}\setminus F_n| = o(|F_n|)$. Let $\mathcal{C}=\left\{  C_{i}\right\}  $ be a finite $\delta$-cover
		of $X$, that is, for every $C_i$, $\sup_{z, w\in C_i}d(z, w)\le \delta$. For every $n\in\NN$ we define
		\begin{align*}
			K_{n}  =\bigvee\nolimits_{g\in F_{n}}g^{-1}\mathcal{C} \ \ \mbox{ and } \ \  \partial K_{n}  =\bigvee\nolimits_{g\in\widetilde{F_n}\setminus F_n
			}g^{-1}\mathcal{C}.
		\end{align*}
		Let $\mathbf{U}=(B_{\frac{\varepsilon}{2}}(x_1), \dots, B_{\frac{\varepsilon}{2}}(x_k))$. Since $(x_1, \dots, x_k)$
		is an orbit IE-tuple, there exists $q > 0$ such that for every $n\in \NN$ there exists an independence set $J_n \subset F_n$ for $\mathbf{U}$ such that $|J_n| \geq q|F_n|$. We also have that for all $n\in\NN$.
		\[
		\left\vert\partial K_{n}\right\vert\leq\left\vert \mathcal{C}\right\vert ^{\left\vert \widetilde{F_n}\setminus F_n\right\vert }.
		\]
		Since $|\widetilde{F_n}\setminus F_n| = o(|F_n|)$, for all
		sufficiently large $m \in \NN$ we have $|\widetilde{F_m}\setminus F_m| < q\frac{\log(k/(k-1))}{\log(|\mathcal{C}|)}|F_m|$ and thus,
		\[ (\frac{k}{k-1})^{\left\vert J_m\right\vert} \geq (\frac{k}{k-1})^{q\left\vert F_{m}\right\vert}
		>\left\vert \mathcal{C}\right\vert ^{\left\vert \widetilde{F_m}\setminus F_{m}\right\vert } \geq \left\vert\partial K_{m}\right\vert.\]
		In particular, as $J_m$ is an independence set for $\mathbf{U}$, there exist
		$\gamma \in J_{m}$, $C \in\partial K_{m}$ and $x^{\prime}_1,\dots, x^{\prime}_k\in C$ such that
		\begin{align*}
			\gamma x^{\prime}_j \in B_{\frac{\varepsilon}{2}}(x_j) \ \ \mbox{ for every } \ \  1\le j\le k.
		\end{align*}

		Let $2\le j\le k$. Since $x^{\prime}_1, x^{\prime}_j\in C$, for every $g\in \widetilde{F_m}\setminus F_m$ we have that $d(gx^{\prime}_1, gx^{\prime}_j)\leq\delta$ and thus $d_{\widetilde{F_m}\setminus F_m}(x^{\prime}_1, x^{\prime}_j) \leq \delta$. As $\widetilde{F_m}$ is an $(\frac{\varepsilon}{2},\delta)$-memory set for $F_m$, there exists $z^{\prime}_j \in X$ such that  $d_{\widetilde{F_m}}(x^{\prime}_j,z^{\prime}_j)\leq\frac{\varepsilon}{2}$ and $d_{G\setminus F_m}(x^{\prime}_1, z^{\prime}_j)\leq\frac{\varepsilon}{2}$. By definition, $(x^{\prime}_1, z^{\prime}_j)$ is an $\frac{\varepsilon}{2}$-asymptotic pair and as $\gamma \in F_m \subset \widetilde{F_m}$ we obtain that $d(\gamma x^{\prime}_j, \gamma z^{\prime}_j)\leq \frac{\varepsilon}{2}$ and hence $\gamma z^{\prime}_j \in B_{\varepsilon}(x_j)$.

Defining $z_j = \gamma z^{\prime}_j$ for $2\le j\le k$ and $z_1 = \gamma x^{\prime}_1$ yields that $(z_1, \dots, z_k) \in \R^{\varepsilon}_k(X,G) \cap ( B_{\varepsilon}(x_1)\times \dots \times B_{\varepsilon}(x_k))$. \end{proof}

	\begin{remark}\label{rtheorem_A_to_IE}
		Proposition~\ref{ida} does not hold in general for actions with the
		(uniform) TMP. Indeed,~\Cref{example_TOM} is an expansive algebraic action of a countable amenable group with positive topological entropy and no off-diagonal asymptotic pairs. By~\Cref{P_algebraic_have_wtmp} and~\Cref{P-exp TMP to uniform TMP} we conclude that this action has the uniform TMP but cannot have the mean TMP as it would contradict Remark~\ref{R-exp mTMP IE to asym}.
	\end{remark}

	Corollary~\ref{C-exp sTMP IE to asym} does not hold in general for actions of non-amenable groups and sofic topological entropy. In order to illustrate this, we need to introduce a few concepts from measurable dynamics. Given a sofic group $G$, a sofic approximation sequence $\Sigma$ for $G$, and an action $G\curvearrowright X$ and a $G$-invariant Borel probability measure $\mu$ on $X$, we say $G \curvearrowright (X,\mu)$ is \define{Bernoulli} if it is isomorphic to a p.m.p. action of the form $G \curvearrowright (Y^G,\nu^G)$ where $Y$ is a compact metrizable space, $\nu$ is a Borel probability measure on $Y$, $\nu^G$ is the product measure on $Y^G$ and the action $G\curvearrowright Y^G$ is given by $gy(h) = y(g^{-1}h)$ for every $y \in Y^G$ and $g \in G$. We say that $G \curvearrowright (X,\mu)$ has \define{completely positive entropy} with respect to $\Sigma$ if every nontrivial measurable factor $G \curvearrowright (X',\mu')$ of $G \curvearrowright (X,\mu)$ has positive measure-theoretical sofic entropy with respect to $\Sigma$ (by nontrivial we mean that $\nu$ does not have an atom of full measure). A theorem of Kerr~\cite[Theorem 2.6]{Kerr2014} states that every Bernoulli measure has completely positive entropy with respect to any sofic approximation sequence $\Sigma$.

\begin{example}\label{E-pm}
	Let $F_{2} = \langle a,b \mid \varnothing \rangle$ be the free group on two generators. The \define{perfect matchings subshift} $X_{\textrm{pm}}$ is the SFT consisting of the configurations $x \in \{a,b,a^{-1},b^{-1}\}^{F_2}$ such that for every $g\in F_2$ we have $(x(g))^{-1} = x(g\cdot x(g))$. By a result of Lyons
and Nazarov there is an $F_2$-invariant Borel probability measure $\mu_{\textrm{pm}}$ on $X_{\textrm{pm}}$ which is a nontrivial factor of a Bernoulli measure~\cite[Theorem 1.1]{LyonsNazarov2011}. By the result of Kerr, we obtain that $h_{\mu_{\textrm{pm}}}^{\Sigma}(F_2 \curvearrowright X_{\textrm{pm}})>0$ for every sofic approximation sequence $\Sigma$. By the variational principle for actions of sofic groups, we conclude that $F_2 \curvearrowright X_{\textrm{pm}}$ has positive topological sofic entropy for every sofic approximation sequence $\Sigma$. However, the next proposition shows that $F_2 \curvearrowright X_{\textrm{pm}}$ has no off-diagonal asymptotic pairs.
\end{example}

\begin{proposition} \label{P-pm}
The action $F_2\curvearrowright X_{\textrm{pm}}$ has no off-diagonal asymptotic pairs.
\end{proposition}
\begin{proof}
	Let $(x,y)\in (X_{\textrm{pm}})^2$ be an asymptotic pair. Let $F \Subset G$ be the set of $g \in F_2$ such that $x(g)\neq y(g)$. If $x\neq y$ then $F \neq \varnothing$. Let $g \in F$ such that the length of the reduced word $w_1w_2\dots w_k \in \{a,b,a^{-1},b^{-1}\}^*$ which represents $g$ is maximized. Then we have that $gs \notin F$ for every $s \in  \{a,b,a^{-1},b^{-1}\} \setminus \{ w_k^{-1} \}$. If $x(g)=s$ for some $s \in  \{a,b,a^{-1},b^{-1}\} \setminus \{w_k^{-1} \}$, then $y(gs)=x(gs)  = s^{-1}$ and thus $y(g)=y(gss^{-1}) = s$ which contradicts $g \in F$. Then we have $x(g)=w_k^{-1}$. Using the same argument we get $y(g) = w_k^{-1}$ and thus $x(g)=y(g)$ contradicting again that $g \in F$. Therefore $F = \varnothing$ and thus $x = y$.
\end{proof}

\section{Applications}\label{section:applications}

In this section we shall put together the results of~\Cref{section:resultados_basicos,section:algebraico,section:minimial,section:ent_asympt} to obtain several results about Markovian measures, algebraic actions, and minimal actions.

\subsection{Supports of Markovian measures}\label{subsection_gibbs}

Recall that a p.m.p. action $G\curvearrowright (X,\mu)$ of an amenable group has completely positive entropy if every nontrivial measurable factor of $G\curvearrowright (X, \mu)$ has positive measure-theoretical entropy.
\begin{corollary}
	Let $G$ be an amenable group, $X$ a $G$-subshift, and $\mu$ a $G$-invariant Markovian measure on $X$. The following hold:
	\begin{enumerate}
		\item If $h_{\mu}(G \curvearrowright X)>0$, then $G\curvearrowright$ $\supp(\mu)$ has off-diagonal asymptotic pairs.
		\item If $G \curvearrowright (X,\mu)$ has completely positive entropy, then the asymptotic pairs of $G\curvearrowright$ $\supp(\mu)$ are dense in $\supp(\mu) \times \supp(\mu)$.
	\end{enumerate}
\end{corollary}

\begin{proof}
	If $h_{\mu}(G \curvearrowright X)>0$, we have that \[h_{\text{top}}(G \curvearrowright \supp(\mu)) \geq h_{\mu}(G \curvearrowright \supp(\mu)) = h_{\mu}(G\curvearrowright X) >0,\]
where the first inequality comes from the variational principle.

If $G \curvearrowright (X,\mu)$ has completely positive entropy, then $\IE_2(\supp(\mu), G)=\supp(\mu) \times \supp(\mu)$ by~\cite[Theorems 2.27 and 2.21]{KerrLi2009}.

By~\Cref{lemma_Markovian_measures_have_STMP} we have that $G\curvearrowright \supp(\mu)$ has the strong TMP.
Now the corollary follows from Corollary~\ref{C-exp sTMP IE to asym}.
\end{proof}

\subsection{Algebraic actions}

	For an action $G\curvearrowright X$ of $G$  on a compact metrizable group $X$ by continuous automorphisms, we say $x\in X$ is a \define{homoclinic point} \cite{LindSchmidt1999} if $sx\to e_X$ as $G\ni s\to \infty$, where $e_X$ denotes the identity element of $X$.
The homoclinic points form a $G$-invariant normal subgroup of $X$,  called the \define{homoclinic group} and denoted  by $\Delta(X,G)$.
Using the fact that $X$ admits a translation-invariant compatible metric (see the proof of Proposition~\ref{P_algebraic_have_wtmp}),
it is easy to see that for each $k\ge 2$, one has
\begin{align*}
\R_k(X, G)&=\{(yx_1, \dots, yx_k): x_1, \dots, x_k\in \Delta(X, G), y\in X\}\\
&=\{(x_1y, \dots, x_ky): x_1, \dots, x_k\in \Delta(X, G), y\in X\}.
\end{align*}
We say $x\in X$ is an \define{IE point} \cite[Definition 7.2]{ChungLi2015} \cite[page 247]{KerrLi2013} if $(x,e_X)\in \IE_2(X,G)$.
The IE points form a $G$-invariant closed normal subgroup of $X$ \cite[Theorem 6.4]{KerrLi2013}, called the \define{IE group} and denoted  by $\IE(X,G)$.
Furthermore, for each $k\in \NN$, we have \cite[Theorem 6.4]{KerrLi2013}
\begin{align*}
\IE_k(X, G)&=\{(yx_1, \dots, yx_k): x_1, \dots, x_k\in \IE(X, G), y\in X\}\\
&=\{(x_1y, \dots, x_ky): x_1, \dots, x_k\in \IE(X, G), y\in X\}.
\end{align*}

\begin{corollary}\label{P_homoclinic_points}
	Let $G\curvearrowright X$ be an expansive action of $G$ on a compact metrizable group $X$ by continuous automorphisms. We have:
	\begin{enumerate}
		\item $\Delta(X, G)\subset \IE(X, G)$. In particular, $\Delta(X,G)^k\subset \IE_{k}(X,G)$ for every $k\in \NN$.
		\item If $G$ is sofic and $\Sigma$ is any sofic approximation sequence for $G$, then $\Delta(X,G)^k\subset \IE^{\Sigma}_{k}(X,G)$ for every $k\in \NN$.
	\end{enumerate}
\end{corollary}
\begin{proof}
	It is quite easy to see that $e_X \in \IE_1(X,G)$ and $e_X \in \IE_1^{\Sigma}(X,G)$. By~\Cref{P_algebraic_have_wtmp}  we have that $G\curvearrowright X$ has the TMP. Therefore, using~\Cref{T-A to IE} and~\Cref{Theorem_asymptotic_pairs_give_SOFIC_entropy_pairs} we obtain respectively that $\{e_X\}\times \Delta(X,G) \subset \IE_{2}(X,G)$ and $\{e_X\}\times \Delta(X,G) \subset \IE^{\Sigma}_{2}(X,G)$. From the first inclusion we get (1). From the second inclusion  we get  that $\Delta(X,G)\subset \IE^{\Sigma}_{1}(X,G)$. Therefore by repeating the above argument we obtain that $\Delta(X,G)^2\subset \IE^{\Sigma}_{2}(X,G)$. Iterating the above argument yields (2).\end{proof}

In the case $X$ is abelian and $G$ is amenable,  part (1) of~\Cref{P_homoclinic_points} was known first as a consequence of
\cite[Theorems 5.6 and 7.8]{ChungLi2015} and the fact that every expansive algebraic action of $G$ is finitely generated~\cite[Proposition 2.2 and Corollary 2.16]{Schmidt1995}. In general, when $X$ is abelian, part (1) of~\Cref{P_homoclinic_points} is also a consequence of~\cite[Theorem 6.5]{KerrLi2013} and \cite[Theorem 5.6]{ChungLi2015}.

A straightforward application of~\Cref{P_homoclinic_points} together with Theorems~\ref{theorem_naive_IE_cites} and \ref{theorem_sofic_IE_cites} yields the following result:

\begin{corollary}
	\label{coralg}
	Let $G\curvearrowright X$ be an expansive action of $G$ on a compact metrizable group $X$ by continuous automorphisms. The following hold:
	\begin{enumerate}
		\item If $\Delta(X,G)\neq \{e_X\}$, then $\hnaive(G \curvearrowright X) >0$.
		\item If $\Delta(X,G)$ is dense in $X$, then $G \curvearrowright X$ has naive UPE of all orders.
	\end{enumerate}
Furthermore, if $G$ is sofic with $\Sigma$ a sofic approximation sequence for $G$, then:
	\begin{enumerate}
		\item If $\Delta(X,G)\neq \{e_X\}$, then $\hsof(G \curvearrowright X) >0$.
		\item If $\Delta(X,G)$ is dense in $X$, then $G \curvearrowright X$ has sofic UPE of all orders.
	\end{enumerate}
\end{corollary}

\begin{remark}
	Note that we do not need to assume that $X$ is abelian for any of the results so far in this section.
\end{remark}

The previous corollaries are false without the  expansivity assumption. In~\cite[Example 7.5]{LindSchmidt1999} Lind and Schmidt constructed an  algebraic action of $\ZZ^3$ with zero topological entropy and
nontrivial homoclinic points.

Possibly the most interesting corollary in this section is the following.

\begin{corollary}\label{corolario_chingon}
	Let $G\curvearrowright X$ be an expansive algebraic action of an amenable group such that either:
	\begin{enumerate}
		\item $\ZZ G$ is left Noetherian, or
		\item $G$ satisfies the strong Atiyah conjecture, there is an upper bound on the orders of finite subgroups of $G$, and $G \curvearrowright X$ is finitely presented.
	\end{enumerate}
	Then $\overline{\Delta(X,G)} = \IE(X,G)$.
\end{corollary}	

\begin{proof}
	By~\Cref{T-expansive algebraic to STMP} the hypotheses above imply that $G\curvearrowright X$ has the strong TMP. Using Theorem~\ref{T-exp sTMP IE to asym} and Corollary~\ref{P_homoclinic_points} yields the result.
\end{proof}

For an action $G\curvearrowright X$ of an amenable group $G$ on a compact metrizable group $X$ by continuous automorphisms, denoting by $\mu_X$ the normalized Haar measure of $X$, one has $h_{\text{top}}(G\curvearrowright X)=h_{\mu_X}(G \curvearrowright X)$ \cite[Theorem 2.2]{Deninger2006}. See also~\cite[Proposition 13.2]{KerrLiBook2016}.
For such an action,
$\IE(X, G)\neq \{e_X\}$ exactly when  $h_{\text{top}}(G\curvearrowright X)>0$~\cite[Theorem 7.3]{ChungLi2015}. Furthermore, $\IE(X, G)=X$ exactly when $G\curvearrowright X$ has completely positive topological entropy in the sense that every nontrivial (i.e. not reduced to a single point) topological factor of $G\curvearrowright X$ has positive topological entropy~\cite[Theorem 7.4]{ChungLi2015}, also exactly when the p.m.p. action $G\curvearrowright (X, \mu_X)$ has completely positive entropy in the sense that every nontrivial measurable factor of $G\curvearrowright (X, \mu_X)$ has positive measure-theoretical entropy~\cite[Theorem 8.1]{ChungLi2015}.
Now Theorem~\ref{T-main} follows from Corollary~\ref{corolario_chingon}.

We will see now that Theorem~\ref{T-main} does not hold for naive topological entropy if the group is not amenable (even if the finitely presented expansive algebraic action is an SFT).

	\begin{example}
		\label{example1}
		Let $F_{2} = \langle a,b \mid \varnothing \rangle$ be the free group on two generators. For $g \in F_2$ denote by $|g|$ the length of the reduced word on $\{a,b,a^{-1},b^{-1}\}^*$ representing $g$. Also, for $g,h \in F_2$ define their distance by $\delta(g,h)=|g^{-1}h|$. Write $B_n = \{ g \in F_2 : |g|\leq n\}$ and $\partial B_n = B_n \setminus B_{n-1}$.
		
		Consider the $5$-dot shift in $F_2$ given by	\[
		X_{\cross}=\left\{  x\in (\ZZ/ 2\ZZ)^{F_{2}}%
		: \sum_{s\in B_1}x(gs)=0\text{ for every }g\in
		F_{2}\right\}.\]
				Let $\FF_{\cross}$ be the finite set of all $p \in (\ZZ/ 2\ZZ)^{B_1}$ such that $\sum_{s \in B_1} p(s) \neq 0$. Clearly \[X_{\cross} = (\ZZ/2\ZZ)^{F_2}\setminus \bigcup_{g \in F_2, p \in \FF_{\cross}}g[p].\]
Thus $X_{\cross}$ is a subshift of finite type. It is also clear that $F_2\curvearrowright X_{\cross}$ is an algebraic action and $\widehat{X_{\cross}}=\ZZ F_2/J$, where $J$ is the left ideal of $\ZZ F_2$ generated by $2$ and $\sum_{s\in B_1}s$. Thus  $F_2\curvearrowright X_{\cross}$ is a finitely presented expansive algebraic action.
	\end{example}
	\begin{proposition}
		$X_{\cross}$ has no off-diagonal asymptotic pairs.
	\end{proposition}
	
	\begin{proof}
		Suppose there exists an asymptotic pair $(x,y)$ such that $x \neq y$. Let $n$ be the smallest integer such that $x|_{F_2 \setminus B_n} = y|_{F_2 \setminus B_n}$. It follows that there is $g \in \partial B_n$ such that $x(g)\neq y(g)$.
		
		Let $u \in B_1$ such that $gu \in \partial B_{n+1}$. For every $h \in B_1 \setminus \{u^{-1},e_{F_2}\}$ we have $guh \in \partial B_{n+2}$. By definition, we have that every $z \in X_{\cross}$ satisfies $\sum_{s \in B_1}z(gus) = 0$ and thus $z(g) = \sum_{s \in B_1 \setminus \{ u^{-1}\}} z(gus)$. Hence, using that $x|_{F_2 \setminus B_n} = y|_{F_2 \setminus B_n}$ we obtain \[ x(g) = \sum_{s \in B_1 \setminus \{ u^{-1}\}} x(gus) = \sum_{s \in B_1 \setminus \{ u^{-1}\}} y(gus) = y(g), \]
		which contradicts $x(g) \neq y(g)$.
	\end{proof}

		Let $X \subset \Lambda^G$ be a $G$-subshift and let $\FF \subset \bigcup_{A \Subset G}\Lambda^A$. Fix $F \Subset G$. The set of globally admissible patterns of support $F$ is $L_F(X) = \{p \in \Lambda^F : [p] \cap X \neq \varnothing  \}$. The set of $\FF$-locally admissible patterns of support $F$ is $L^{\texttt{loc}}_F(\FF) = \{p \in \Lambda^F : \mbox{ for every } q \in \FF \mbox{ and } g \in G, [p] \not\subset g[q] \}$. Note that whenever $\FF$ generates $X$ in the sense that $X = \Lambda^G \setminus \bigcup_{g \in G, p \in \FF}g[p]$, we have $L_F(X) \subset L^{\texttt{loc}}_F(\FF)$.

	For $s, t, g \in F_2$ let $(s, t)_g = \frac{1}{2}( \delta (s,g)+\delta (t,g)-\delta (s, t))$. For $F \Subset F_2$, the \define{span} of $F$ is the set $\texttt{Span}(F)$ of all $g \in F_2$ for which there are $s, t \in F$ such that $(s, t)_g =0$. We say that $F$ is \define{connected} if $F = \texttt{Span}(F)$. This is the same as saying that $F$ is connected  in the right Cayley graph of $F_2$ given by the generators $\{a,b,a^{-1},b^{-1}\}$.

	\begin{lemma}\label{claim_example1}
		Let $A, B$ be connected finite subsets of $F_2$ such that $A \subset B$ and $|B\setminus A| =1$. For every $p \in L^{\texttt{loc}}_A(\FF_{\cross})$ there exists $q \in L^{\texttt{loc}}_B(\FF_{\cross})$ such that $q|_A = p$.
	\end{lemma}
	
	\begin{proof}
		As $|B\setminus A| =1$ and $B \supset A$ is connected, there exists $g \in A$ and $s \in \partial B_1$ such that $gs \in B \setminus A$. As $A$ is connected, if $u \in B_1 \setminus \{s^{-1} \}$ then $gsu \notin A$.
		
		For $h \in A$ we let $q(h) = p(h)$. If $gB_1 \subset B$, then we define $q(gs) = \sum_{v \in  B_1 \setminus \{s \}} p(gv)$, otherwise we define $q(gs) = 0$. By definition, $q|_A = p$. Let us show that $q \in L^{\texttt{loc}}_B(\FF_{\cross})$.
		
		Suppose that $q \notin L^{\texttt{loc}}_B(\FF_{\cross})$. Then there is some $h\in B$ such that $hB_1\subset B$ and $\sum_{v\in B_1}q(hv)\neq 0$. As $q|_A = p \in L^{\texttt{loc}}_A(\FF_{\cross})$, we have $gs\in hB_1$. Since for every $u \in B_1 \setminus \{s^{-1} \}$ we have $gsu \notin A$,  the only possibility is that $h=g$.  Then $gB_1\subset B$ and by definition of $q(gs)$ we have $\sum_{v\in B_1}q(gv)=0$, which is a contradiction.
 \end{proof}
	
	\begin{lemma}\label{claim_example4}
		Let $F$ be a connected finite subset of $F_2$. Then $L^{\texttt{loc}}_F(\FF_{\cross}) = L_F(X_{\cross})$.
	\end{lemma}
	
	\begin{proof}
		Let $p \in L^{\texttt{loc}}_F(\FF_{\cross})$ and  consider an increasing sequence $\{A_n\}_{n \in \NN}$ of finite subsets of $F_2$ with union $F_2$ such that $A_1=F$. In order to show that $p \in L_F(X_{\cross})$ it suffices to construct a sequence $p_n \in L^{\texttt{loc}}_{A_n}(\FF_{\cross})$ such that $p_1 = p$ and for $n \geq 1$, $p_{n+1}|_{A_{n}} = p_{n}$. Indeed, by definition $[p_{n}] \supset [p_{n+1}]$ and thus $\bigcap_{n \in \NN}[p_n]$ is nonempty. As the sets $A_n$ increase to $F_2$, the set $\bigcap_{n \in \NN}[p_n]$ is a singleton, say $\{x\}$. If $x\in g[p]$ for some $g \in F_2$ and $q \in \FF_{\cross}$, then $p_n\subseteq g[q]$ for some large $n$, which is impossible as $p_n \in L^{\texttt{loc}}_{A_n}(\FF_{\cross})$. Therefore
$x \in X_{\cross} \cap [p]$.
		
		For $m \geq 0$, consider $C_m = FB_m$. Let $h_1,\dots,h_{k(m)}$ be an enumeration of $C_{m+1} \setminus C_{m}$ and for $i \in \{1,\dots,k(m)\}$ let $D_{m,i} = C_m \cup \{h_1,\dots,h_i\}$ and note that $D_{m,k(m)} = C_{m+1}$. For $n\in \NN$ define recursively $A_1 = C_1 = F$ and \[A_{n+1} = \begin{cases}
		D_{m,i+1} & \mbox{ if } A_{n} = D_{m,i} \mbox{ and } i < k(m)\\
		D_{m,1} & \mbox{ if } A_{n} = C_m 	
		\end{cases}.\]
		
		Note that each $A_n$ is a connected subset of $F_2$, $|A_{n+1}\setminus A_n|=1$ for every $n\in \NN$ and that $A_n$ increases to $F_2$. Assume inductively that we have $p_{n}  \in L^{\texttt{loc}}_{A_n}(\FF_{\cross})$. By~\Cref{claim_example1} there is $p_{n+1} \in L^{\texttt{loc}}_{A_{n+1}}(\FF_{\cross})$ such that $p_{n+1}|_{A_n} = p_n$.
\end{proof}

		We say a pair of sets $A, B\subset F_2$ are $k$-separated if $\delta (A,B) = \inf_{(g,h) \in A \times B}\delta (g,h) \geq k$. We say a collection $A_1,\dots, A_n \subset F_2$ is $k$-separated if it is pairwise $k$-separated.

	\begin{lemma}\label{claim_example2}
		Let $A,B$ be a $3$-separated pair of connected subsets of $F_2$. Then for every $p \in L^{\texttt{loc}}_{A \cup B}(\FF_{\cross})$ there exists $q \in L^{\texttt{loc}}_{\texttt{Span}(A \cup B)}(\FF_{\cross})$ such that $q|_{A \cup B} = p$.
	\end{lemma}
	
	\begin{proof}
		Let $h_A \in A$ and $h_B \in B$ such that $\delta (h_A,h_B) = \inf_{(g,h) \in A \times B}\delta (g,h)$. As $A,B$ are connected  and $F_2$ has no cycles, it follows that $\texttt{Span}(A \cup B) = A \cup B \cup \texttt{Span}(\{h_A,h_B\}).$
		Let $h_0,h_1,\dots,h_n,h_{n+1} \in F_2$ be an enumeration of $\texttt{Span}(\{h_A,h_B\})$ such that $h_0= h_A$, $h_{n+1} = h_B$ and $A \cup \{h_1,\dots, h_k\}$ is connected for every $1 \leq k\leq n$. As $A,B$ are $3$-separated, it follows that $\delta (h_A,h_B)\geq 3$ and thus $n \geq 2$. By~\Cref{claim_example1} there are $p_A \in L^{\texttt{loc}}_{A \cup \{h_1\}}(\FF_{\cross})$ and $p_B \in L^{\texttt{loc}}_{B \cup \{h_{n}\}}(\FF_{\cross})$ such that $p_A|_{A} = p|_A$ and $p_B|_{B} = p|_B$. Let $q \in \{0,1\}^{\texttt{Span}(A \cup B)}$ be defined by \[q(g) = \begin{cases}
		p_A(g) & \mbox{ if }g \in A \cup \{ h_1\} \\
		p_B(g) & \mbox{ if }g \in B \cup \{ h_n\} \\
		0 & \mbox{ otherwise.}
		\end{cases}   \]
		It can be verified directly from the definition that $q|_{A \cup B} = p$ and that $q \in L^{\texttt{loc}}_{\texttt{Span}(A \cup B)}(\FF_{\cross})$.\end{proof}
	
	\begin{lemma}\label{claim_example3}
		Let $\mathcal{A}$ be a $5$-separated finite collection of finite connected subsets of $F_2$. Let  $p_A \in L^{\texttt{loc}}_{A}(\FF_{\cross})$ for each $A\in \mathcal{A}$. Then there exists $p \in L^{\texttt{loc}}_{\texttt{Span}(\bigcup_{A \in \mathcal{A}} A)}(\FF_{\cross})$ such that $p|_{A}= p_A$ for all $A\in \mathcal{A}$.
	\end{lemma}
	
	\begin{proof}
		We proceed by induction on the size of $\mathcal{A}$. If $|\mathcal{A}|=1$ the result is trivial. Let $|\mathcal{A}|=n >1$ and assume the result for all $5$-separated collections of cardinality at most $n-1$. By~\Cref{claim_example2} it is enough to show that there is $A \in \mathcal{A}$ which is $3$-separated from $\texttt{Span}(\bigcup_{C \in \mathcal{A} \setminus \{A\}}C)$.
		
		Let us verify the above property. Pick $A_0 \in \mathcal{A}$ and choose $A \in \mathcal{A}$ such that $\delta (A,A_0)$ achieves the maximum. We claim that $\delta (A, \texttt{Span}(\bigcup_{C \in \mathcal{A} \setminus \{A\}}C)) \geq 3$. Indeed, let $g_0 \in A_0$ and $g \in A$ such that $\delta (A,A_0) = \delta (g,g_0)$. Assume that $\delta (A, \texttt{Span}(\bigcup_{C \in \mathcal{A} \setminus \{A\}}C)) \leq 2$. Then there exist $h \in \texttt{Span}(\bigcup_{C \in \mathcal{A} \setminus \{A\}}C) \setminus \bigcup_{C \in \mathcal{A} \setminus \{A\}}C$ and $\bar{g}\in A$ such that $\delta (\bar{g},h)=\delta (A, \texttt{Span}(\bigcup_{C \in \mathcal{A} \setminus \{A\}}C)) \leq 2$.
It follows that there are distinct $A_1, A_2\in \mathcal{A}\setminus \{A\}$ and $g_1\in A_1$, $g_2\in A_2$ such that $h\in \texttt{Span}(\{g_1, g_2\})$ and $\texttt{Span}(\{g_1, g_2\})\cap \bigcup_{C \in \mathcal{A} \setminus \{A\}}C=\{g_1, g_2\}$. Without loss of generality we may assume that $A_1\neq A_0$. If $\bar{g}\neq g$, then
\begin{align*}
\delta(A_1, A_0)=\delta(g_1, g_0)=\delta(g_1, \bar{g})+\delta(\bar{g}, g)+\delta(g, g_0)>\delta(g, g_0)=\delta(A, A_0),
\end{align*}
which contradicts our choice of $A$. Thus $\bar{g}=g$. We have $\texttt{Span}(\{g, g_0\})\cap \texttt{Span}(\{g, g_1\})=\texttt{Span}(\{g, h_1\})$ for some $h_1$. If $h_1\in \texttt{Span}(\{g, h\})$, then
\begin{align*}
\delta(A_0, A_1)&=\delta(g_0, g_1)\\
&=\delta(g_0, h_1)+\delta(h_1, g_1)\\
&=\delta(g_0, g)-\delta(h_1, g)+\delta(g, g_1)-\delta(g, h_1)\\
&\ge \delta(g_0, g)+\delta(g, g_1)-2\delta(h, g)\\
&\ge \delta(g_0, g)+5-4>\delta(g_0, g)=\delta(A_0, A),
\end{align*}
which contradicts our choice of $A$. Thus $h_1\not\in \texttt{Span}(\{g, h\})$. Then $A_2\neq A_0$, and $\texttt{Span}(\{g, g_0\})\cap \texttt{Span}(\{g, g_2\})=\texttt{Span}(\{g, h_2\})$ for some $h_2\in \texttt{Span}(\{g, h\})$. Similar to the above, we get $\delta(A_0, A_2)>\delta(A_0, A)$, which again contradicts our choice of $A$.
\end{proof}

	\begin{proposition}
		$F_2\curvearrowright X_{\cross}$ has naive UPE of all orders.
	\end{proposition}
	
	\begin{proof}
It suffices to show $\IE(X_{\cross}, F_2)=X_{\cross}$.
		Let $(x_1, x_2) \in (X_{\cross})^2$. Then it suffices to show that for every $n \in \NN$ there exists $\alpha >0$ such that for every $D \Subset F_2$ there exists a set $A \Subset D$ such that $|A| \geq \alpha |D|$ and for any map $\varphi \colon A \to \{1, 2\}$ we have \[X_{\cross}\cap \bigcap_{a \in A} a[x_{\varphi(a)}|_{B_n}] \neq \varnothing.\]
		
		Put $\alpha=|B_{4+2n}|^{-1}$. Fix $D \Subset F_2$ and let $A$ be any maximal $(5+2n)$-separated subset of $D$. Then $AB_{4+2n}\supset D$, and hence $|A| \geq |B_{4+2n}|^{-1}|D|=\alpha |D|$. Consider the collection $\mathcal{A} = \{aB_n : a \in A\}$ and for each $a \in A$ the pattern $p_a = (ax_{\varphi(a)})|_{aB_n}$. By definition, $\mathcal{A}$ is $5$-separated and each $p_a \in L_{aB_n}(X_{\cross})$, therefore, by~\Cref{claim_example3}, there exists $p \in L^{\texttt{loc}}_{\texttt{Span}(\bigcup_{a \in A} aB_n)}(\FF_{\cross})$ such that $p|_{aB_n}= p_a$ for every $a \in A$. In other words, $[p] \subset \bigcap_{a \in A} a[x_{\varphi(a)}|_{B_n}]$.
		
		Furthermore, as $\texttt{Span}(\bigcup_{a \in A} aB_n)$ is connected, by~\Cref{claim_example4} we have $[p]\cap  X_{\cross} \neq \varnothing$ and therefore $X_{\cross}\cap \bigcap_{a \in A} a[x_{\varphi(a)}|_{B_n}] \neq \varnothing$.
\end{proof}

\subsection{Minimal actions}

\begin{corollary}\label{cor_minimalhasentropy0}
	Let $G$ be an amenable group and $G \curvearrowright X$ a minimal expansive action with the strong TMP. Then $\htop(G\curvearrowright X) = 0$.
\end{corollary}

\begin{proof}
By Proposition~\ref{P-amenable_implications_trivial} the action $G\curvearrowright X$ has the TMP. Then by  Theorem~\ref{T-minimal exp tmp iff noasym}
there is no off-diagonal asymptotic pair. From Corollary~\ref{C-exp sTMP IE to asym} we conclude that $\htop(G\curvearrowright X) = 0$.
\end{proof}

The previous result was proven for minimal ${\ZZ}^{d}$-SFTs in~\cite[Corollary 2.3]{QuasTrow2000} and for minimal $G$-SFTs of any amenable group in~\cite[Corollary 3.17]{Barbieri2019} using the formalism of group quasi-tilings. Our result jointly generalizes these previous theorems from the context of subshifts  and the POTP.

\begin{remark}
	\Cref{cor_minimalhasentropy0} also holds if we just assume that $G \curvearrowright X$ has the mean TMP instead of the strong TMP.
\end{remark}

We do not know whether the result fails for non-amenable groups. We believe the following question might not be easy.
\begin{question}
	Does there exist a sofic group (and some sofic approximation sequence) for which there exists a minimal SFT with positive topological sofic entropy?
\end{question}

\bibliographystyle{abbrv}
\bibliography{ref}

\Addresses
	
\end{document}